\documentclass[10pt]{amsart}
\DeclareRobustCommand{\subtitle}[1]{\\#1}
  
 \makeatletter
\def\@tocline#1#2#3#4#5#6#7{\relax
  \ifnum #1>\c@tocdepth 
  \else
    \par \addpenalty\@secpenalty\addvspace{#2}%
    \begingroup \hyphenpenalty\@M
    \@ifempty{#4}{%
      \@tempdima\csname r@tocindent\number#1\endcsname\relax
    }{%
      \@tempdima#4\relax
    }%
    \parindent\z@ \leftskip#3\relax \advance\leftskip\@tempdima\relax
    \rightskip\@pnumwidth plus4em \parfillskip-\@pnumwidth
    #5\leavevmode\hskip-\@tempdima
      \ifcase #1
       \or\or \hskip 1em \or \hskip 2em \else \hskip 3em \fi%
      #6\nobreak\relax
    \hfill\hbox to\@pnumwidth{\@tocpagenum{#7}}\par
    \nobreak
    \endgroup
  \fi}
\makeatother 

\usepackage[margin=1.25 in]{geometry} 
\usepackage[utf8]{inputenc}

\usepackage{graphicx}
\usepackage{url}
\usepackage[pagebackref=false,colorlinks]{hyperref}
\hypersetup{citecolor={rgb,256:red,0;green,50;blue,100},
linkcolor={rgb,256:red,;green,60;blue,160}}
\usepackage{mathptmx}     
\usepackage{amsmath,amsfonts,amsthm,amssymb,color}
\usepackage{mathtools}
\usepackage{tikz-cd}

\usepackage{xparse}
\usepackage{calrsfs}
\usepackage{cleveref}
\DeclareMathAlphabet{\pazocal}{OMS}{zplm}{m}{n}

\NewDocumentCommand{\tens}{e{_^}}{%
  \mathbin{\mathop{\otimes}\displaylimits
    \IfValueT{#1}{_{#1}}
    \IfValueT{#2}{^{#2}}
  }%
}
\newcommand{\F}{\mathbb{F}_2}
\newcommand{\Fp}{\mathbb{F}_p}
\newcommand{\lie}{\mathcal{L}}
\newcommand{\Lie}{\mathrm{Lie}^s}
\newcommand{\tiLie}{\mathrm{Lie}^{s,\mathrm{ti}}}
\newcommand{\free}{\mathrm{Free}}
\newcommand{\E}{\mathbb{E}}

\newcommand{\R}{\bar{\pazocal{R}}}
\newcommand{\rg}{\bar{\pazocal{R}}_{>0}}
\newcommand{\B}{\mathrm{Bar}_{\bullet}}
\newcommand{\g}{\mathfrak{g}}
\newcommand{\AR}{\mathrm{AR}_\bullet}
\newcommand{\Q}{\bar{Q}}
\newcommand{\N}{\mathbb{N}}
\newcommand{\Mod}{\mathrm{Mod}}
\newcommand{\CE}{\mathrm{CE}}
\newcommand{\AQ}{\mathrm{HQ}^{\mathrm{Lie}^s_{\bar{\pazocal{R}}}}}
\newcommand{\TQ}{\mathrm{TQ}^{s\mathcal{L}}}
\newcommand{\id}{\mathrm{id}}

\newcommand{\A}{\pazocal{A}}
\newcommand{\wt}{\mathrm{wt}}
\newcommand{\p}{\pazocal{R}}

\usepackage{todonotes}

\newtheorem{theorem}{Theorem}[section]
\newtheorem{corollary}[theorem]{Corollary}
\newtheorem{conjecture}[theorem]{Conjecture}
\newtheorem{lemma}[theorem]{Lemma}

\newtheorem{proposition}[theorem]{Proposition}

\theoremstyle{definition}
\newtheorem{definition}[theorem]{Definition}

\newtheorem{remark}[theorem]{Remark}
\newtheorem{construction}[theorem]{Construction}
\newtheorem{example}[theorem]{Example}
\newtheorem{notation}[theorem]{Notation}

\title{Quillen homology of spectral Lie algebras \subtitle{with application to mod $p$ homology of labeled configuration spaces} }

\author{Adela YiYu Zhang}
\address{Department of Mathematical Sciences, University of Copenhagen, Universitetsparken 5, 2100 Copenhagen, Denmark}
\email{yz@math.ku.dk}

\begin{document}
\begin{abstract}
We provide a general method computing the mod $p$ Quillen homology of algebras over the monad that parametrizes the structure of mod $p$ homology of spectral Lie algebras, including the construction of a May-type spectral sequence when $p=2$.  This is the $E^2$-page of the bar spectral sequence converging to the mod $p$ topological Quillen homology of spectral Lie algebras. As an application, we study the mod $p$ homology of the labeled configuration space $B_k(M;X)$ of $k$ points in a  manifold $M$ with labels in a spectrum $X$, which is the mod $p$ topological Quillen homology of a certain spectral Lie algebra by a result of Knudsen. We obtain general upper bounds for the mod $p$ homology of $B_k(M;X)$, as well as explicit computations for $k\leq 3$.  When $p$ is odd, we observe that the mod $p$ homology of $B_k(M^n;S^r)$ for small $k$ depends only on the cohomology ring of the one-point compactification of $M$ when $n+r$ is even. This supplements and contrasts with the result of B\"{o}digheimer-Cohen-Taylor when $n+r$ is odd.
\end{abstract}
\maketitle
\renewcommand{\subtitle}[1]{}
\section{Introduction}
Spectral Lie algebras generalize the notion of Lie algebras over a field $k$ to the $(\infty-)$category of spectra. They are parametrized by the spectral Lie operad $s\lie$, whose underlying symmetric sequence $\{\partial_n(\mathrm{Id})\}_n$ is given by the Goodwillie derivatives of the identity functor on the category of pointed spaces. The spectral Lie operad is Koszul dual to the nonunital $\mathbb{E}_{\infty}$-operad \cite{ching}. The homology operad $\{H_*(\partial_n(\mathrm{Id});k)\}_n$ of the spectral Lie operad  recovers the ordinary Lie operad over $k$ up to a shift  \cite{gk, fresse, ching}.

A natural next step is to study the mod $p$ \textit{topological Quillen homology} $$\TQ_*(L;\Fp):=\pi_*\big(|\B(\id,s\lie,L)|\tens \Fp\big)$$ of a spectral Lie algebra $L$, defined in analogy to the mod $p$ topological Andr\'{e}-Quillen homology of nonunital $\mathbb{E}_\infty$-algebras introduced by Basterra \cite{taq}. One  approach is to use the classical bar spectral sequence
\begin{equation}\label{bar}
    E^2_{s,t}=\pi_s\pi_t \B\big(\mathrm{id}, s\lie, L\tens \Fp)\Rightarrow \TQ_{s+t}(L;\Fp),
\end{equation}
   obtained by skeletal filtration of the geometric realization. 
   
To compute the $E^2$-page, it is necessary to understand the structure of the mod $p$ homology of spectral Lie algebras. In \cite{behrens},  Behrens constructed  Dyer-Lashof-type unary operations $\Q^j$ of degree $j-1$ on the mod 2 homology of connective spectral Lie algebras and determined the relations among these operations. Building on the work of Behrens, Antol\'{i}n-Camarena \cite{omar} identified the monad $\Lie_{\R}$ that parametrizes natural operations on the mod 2 homology of  spectral Lie algebras. An algebra over $\Lie_{\R}$ is an unstable module over the algebra $\R$ of Behrens' operations, along with a shifted Lie algebra structure such that brackets of operations always vanish and the self-bracket on an element $x$ is identified with the bottom nonvanishing operation $\Q_0:=\Q^{|x|}$ on $x$. Following their approach, Kjaer \cite{kjaer} constructed Dyer-Lashof-type  unary operations $\overline{\beta^\epsilon Q^j}$ on the mod $p$ homology of spectral Lie algebras for $p>2$ and proved that brackets of operations always vanish. Konovalov \cite{konovalov} computed the relations among unary operations by studying differentials in an algebraic Goodwillie spectral sequence and thus determined the entire structure of operations in the odd primary case. 

Hence the $E^2$-page of the bar spectral sequence (\ref{bar}) associated to a spectral Lie algebra $L$ is equivalent to the \textit{Quillen homology}  $$\AQ_*(H_*(L;\Fp)):=\pi_*\Big(\B\big(\mathrm{id},\Lie_{\R},H_*(L;\Fp)\big)\Big)$$ of the $\Lie_{\R}$-algebra $H_*(L;\Fp)$ when $p=2$. Equivalently, the $E^2$-page is obtained by applying to $H_*(L;\Fp)$ the total left derived functor $\pi_*(\mathbb{L}Q^{\Lie_{\R}}_{\Mod_{\Fp}}(-))$ of the indecomposable functor $Q^{\Lie_{\R}}_{\Mod_{\Fp}}: \Lie_{\R}\rightarrow \Mod_{\Fp}$ on the category of $\Lie_{\R}$-algebras which heuristically kills the $\Lie_{\R}$-algebra structure on a $\Lie_{\R}$-algebra.

The main challenge in computing  the Quillen homology of  $\Lie_{\R}$-algebras when $p=2$ arises from the identification of the self-bracket with the bottom operation $\Q_0$, which precludes a factorization of the free $\Lie_{\R}$-algebra functor as a composite of the free $\Lie_{\F}$-algebra functor followed by the free $\R$-algebra functor. Furthermore, since the category of $\Lie_{\F}$-algebras is nonabelian, we cannot resort to a Grothendieck spectral sequence. 

The method we use involves bounding the Quillen homology of  $\Lie_{\R}$-algebras by the Quillen homology of variants of $\Lie_{\R}$-algebras whose unary and binary operations are disentangled. An upper bound is obtained by constructing a May-type spectral sequence with respect to the length filtration of the homogeneous algebra $\R$ in \Cref{upperbound}. To compute the May $E^1$-page we further construct an algebraic Bockstein spectral sequence, whose $E^1$-page is given by the Quillen homology of a variant of $\Lie_{\R}$-algebras (\Cref{DefinetiLieR}). A lower bound can be produced by mapping into the Quillen homology of another variant of $\Lie_{\R}$-algebras (\Cref{DefinetiLieR>0}).
Then we apply a general result about factoring bar constructions against a composite monad (\Cref{factortiLierg}), which generalizes \cite[Proposition 4.19]{bhk}, in that we replace the bar construction computing the Quillen homology of these variant algebras with a smaller complex  obtained as the total complex of a double complex in \Cref{cor:factortiLierg} and \Cref{homotopytiLieR}. 

To compute the homotopy groups of these total complexes, we utilize the machinery of Koszul duality for additive Koszul algebras \cite{priddy} and Lie algebras \cite{bhk}\cite{CE}\cite{may}\cite{priddy}, as well as explicit understanding of the Bousfield-Cartan-Dwyer operations $$\gamma_i:\pi_{h+r,t}(\Lambda^h(V_\bullet))\rightarrow \pi_{2h+1+r+i,2t-1}(\Lambda^{2h+1}(V_\bullet)), 1\leq i \leq r$$ on the homotopy groups of the free simplicial shifted graded exterior algebra $\Lambda^\bullet (V_\bullet)$ on a simplicial $\F$-module $V_\bullet$ \cite{bokstedt, bousfield, cartan, dwyer, hm}. Thus we obtain general upper bounds for the Quillen homology of  $\Lie_{\R}$-algebras via \Cref{upperbound} and \Cref{homotopywhentrivialbracket}, as well as precise formulae in low weights in \Cref{smallweightisom}.

Furthermore, we are able to provide a full computation of the Quillen homology of  $\Lie_{\R}$-algebras in universal cases. Denote by $\free^{\Mod_{\R}}_{\Mod_{\F}}$ the free allowable $\R$-module functor. The category $\Mod_{\R}$ is stable under the desuspension functor $\Sigma^{-1}$ of $\F$-modules. Then for $1\leq n\leq \infty$, the $\R$-module $\Sigma^{-n}\free^{\Mod_{\R}}_{\Mod_{\F}}(\Sigma^{n+k}\F)$ is an $\Lie_{\R}$-algebra whose $\Lie_{\F}$-structure is trivial. Note that when $n=\infty$, this is the trivial $\Lie_{\R}$-algebra $\mathrm{colim}_{i\rightarrow\infty}\Sigma^{-n}\free^{\Mod_{\R}}_{\Mod_{\F}}(\Sigma^{n+k}\F)\simeq\Sigma^k\F$.

\begin{theorem}[\Cref{quillenhomologytrivial}]
 The Quillen homology $$\AQ_{*,*}(\Sigma^{-n}\free^{\Mod_{\R}}_{\Mod_{\F}}(\Sigma^{n+k}\F))\cong\pi_{*,*}\B(\id,\Lie_{\R},\Sigma^{-n}\free^{\Mod_{\R}}_{\Mod_{\F}}(\Sigma^{n+k}\F))$$ of the  $\Lie_{\R}$-algebra $\Sigma^{-n}\free^{\Mod_{\R}}_{\Mod_{\F}}(\Sigma^{n+k}\F), 1\leq n\leq \infty$ is isomorphic as a bigraded vector space to the shifted graded exterior algebra over $\F$ on generators $\gamma_I\Q_J(x_k)$ satisfying the following conditions:
 \begin{enumerate}
     \item $I=(i_1,\ldots, i_m)$ satisfies $i_l\geq 2i_{l+1}$ for $l<m$, $i_m\geq 2$, and $i_1-i_2-\cdots-i_m\leq r$;
     \item $J=(j_1,\ldots,j_r)$ satisfies 
  $0\leq j_l\leq j_{l+1}+1$ for $l<r$, $0\leq j_r<n$, and if $j_1=0$ then either $r=1$ or $i_m=2$. 
 \end{enumerate}

\end{theorem}
Note in particular that natural operations on a class of degree $k$ in the Quillen homology of $\Lie_{\R}$-algebras are given by the Quillen homology of the trivial $\Lie_{\R}$-algebras $\Sigma^k\F$, and the above theorem gives us a dimension count.

\

\noindent{\textbf{Application to labeled configuration spaces.}}
The second half of the paper makes use of the computation of the Quillen homology of $\Lie_{\R}$-algebras to study the mod $p$ homology of the labeled configuration spectrum $$B_k(M,X):=\Sigma^{\infty}_+\mathrm{Conf}_k(M)\tens_{h\Sigma_k} X^{\tens k}$$ of $k$ points in a parallelizable manifold $M$ with labels in a spectrum $X$. 

The study of labeled configuration spaces dates back to as early as Segal \cite{segal} and McDuff \cite{mcduff} as generalizations of the unordered configuration space of $k$ points in $M$.  The rational homology groups of labeled configuration spaces are relatively well understood in cases of interests via classical methods, see for instance \cite{BC, BCT, kriz, totaro, felix}.

In contrast, the mod $p$ homology groups of these objects have remained mostly intractable. Classically, the only known cases are the following:
\begin{itemize}
    \item $M=\mathbb{R}^\infty$ with arbitrary labeling spectra by May \cite{maysq, Einfinity} and Steinberger \cite[III]{bmms} who built on the work of \cite{adem, araki, car54b, dyerlashof}, and $M=\mathbb{R}^n$  by Cohen \cite[III]{CLM}. Then $\bigoplus_{k\geq 0} B_k(M;X)$ is the free $\mathbb{E}_n$-algebra on $X$. Its mod $p$ homology is captured by the Dyer-Lashof operations, a polynomial product, and the Browder brackets as a functor of $H_*(X;\mathbb{F}_p)$.
    \item Arbitrary manifold $M$ with labeling spectrum $X=\Sigma^{\infty}S^r$, where either $p=2$ or $p>2$ and $n+r$ is odd \cite{BCT, milgram, BCM}. In these cases, there is a homology decomposition
\begin{equation}\label{BCT}
   H_*(\bigoplus_{k\geq 0} B_k(M; \mathbb{S}^r);\Fp)\cong \bigotimes_i H_*(\Omega^i S^{n+r};\Fp)^{\tens \mathrm{\ dim}H_i(M)}. 
\end{equation}
In particular, the homology depends only on the $\mathbb{F}_p$-module $H_*(M;\mathbb{F}_p).$
\end{itemize}

The most recent developments in the computation of the homology of labeled configuration spaces originate from a result of Knudsen \cite{ben}. Using the machinery of factorization homology, he established an equivalence of spectra
\begin{equation}\label{knudsen}
    \bigoplus_{k\geq 1} B_k(M;V)  \simeq \mid \B\big(\mathrm{id}, s\lie, \free^{s\lie}(\Sigma^{n}V) ^{M^+}\big) \mid.
\end{equation}   
Here $M$ is a parallelizable $n$-manifold, $s\lie$ is the monad associated to the  free spectral Lie algebra functor $\free^{s\lie}$, and $(-)^{M^+}$  the cotensor with the one-point compactification of $M$ in the $\infty$-category of spectral Lie algebras. A rational version of this equivalence was proved by Ayala and Francis in \cite{af}.

Knudsen's result opens up a path for extracting information about the homology of labeled configuration spaces. In \cite{knudsen}, Knudsen provided a general formula for the Betti numbers of unordered configuration spaces by observing that the bar spectral sequence with rational coefficients for the bar construction (\ref{knudsen})  collapses at the $E^2$-page. Building on Knudsen's work, Drummond-Cole and Knudsen \cite{dck} computed the Betti numbers of unordered configuration spaces of surfaces, vastly improving on earlier works including the case of  once-punctured orientable surface by B\"{o}digheimer and Cohen \cite{BC}.  In \cite{bhk},  Brantner, Hahn, and Knudsen studied the Knudsen spectral sequence  with coefficients in  Morava $E$-theory at an odd prime using Brantner's results on the structure of the Morava $E$-theory of spectral Lie algebras \cite{brantner}. They computed the weight $p$ part of the labeled configuration spaces in $\mathbb{R}^n$ and punctured genus $g$ surfaces $\Sigma_{g,1}$ for $ g\geq 1$ with coefficient in a sphere. By letting the height go to infinity, they observed that the integral homology of $B_p(\Sigma_{g,1})$ is $p$-power-torsion free for any odd prime $p$.

In this paper, we adapt their approach and study the mod $p$ homology of $B_k(M,X)$ for $M$ a parallelizable $n$-manifold and $X$ any spectrum by examining the mod $p$ Knudsen spectral sequence, i.e., the bar spectral sequence (\ref{bar}) with coefficients in $\Fp$ applied to the bar construction (\ref{knudsen}). 

When $p=2$, our general understanding of the $E^2$-page, i.e., the Quillen homology of $\Lie_{\R}$-algebras, allows us to obtain an upper bound for $H_*(B_k(M,X);\F)$ in \Cref{E2upperbound} for arbitrary parallelizable manifold $M$ and spectrum $X$.  In the universal case $M=\mathbb{R}^\infty$ and $X=\mathbb{S}^r$, the bar spectral sequence has $E^2$-page given by \Cref{quillenhomologytrivial}.    Comparing with the computation  of the homology of free $\mathbb{E}_\infty$-algebras \cite{adem, dyerlashof, maysq, bmms}, we see that there are infinitely many higher differentials and conjecture the following universal pattern, which can be verified in low weights by sparsity arguments:
\begin{conjecture} [\Cref{conjecture}]
Each page of the spectral sequence
$$E^2_{s,t}=\AQ_{s,t}(\Sigma^k\F)\Rightarrow \pi_{s+t}\B(\id,s\lie,\Sigma^k\F)$$ is an exterior algebra. The higher differentials act on the exterior generators of the $E^2$-page as follows, see Figure \ref{pattern}:
\begin{enumerate}
    \item For an exterior generator $\alpha=\Q_{j_1}\cdots \Q_{j_m}(x_k)$ on the $E^2$-page,  we have
 $$d^{r+1}\gamma_{r+1}(\alpha)=\Q_{r}(\alpha)$$ for  $r<m$ and $r\leq j_1+1$.
 \item For an exterior generator $\beta=\gamma_{n+1}\Q_{j_1}\cdots \Q_{j_{m-1}}(x_k)$  on the $E^2$-page, we have
 \begin{enumerate}
     \item $d^{n+1}(\beta)=\Q_{n}\Q_{j_1}\cdots \Q_{j_{m-1}}(x_k)$,
     \item $d^{n+1}\gamma_{m+n}(\beta)=d^{n+1}(\beta)\tens \beta$,
     \item 
 $\gamma_{l} d^{n+1}(\beta)=d^{2n+1}\gamma_{n+l-1}(\beta)$ for $n+3\leq l\leq m$.
 \end{enumerate}
\end{enumerate}
  These generate all higher differentials under further applications of the $\gamma_i$ operations in accordance with (2).(b) and (2).(c), as well as the exterior product.
\end{conjecture}
\begin{remark}
 
While the pattern of universal differentials is similar to classical ones studied by Dwyer \cite{dwyerdiff} and Turner \cite{turner}, the operations $\Q_j$  on coalgebras over the comonad $\pi_{*,*}\B(\id,\Lie_{\R},-)$ increase filtration and hence cannot be constructed directly at the chain level, see \Cref{unviersaldiff}. In forthcoming work with Robert Burklund and Andrew Senger, we construct a suitable deformation of the comonad associated to the bar construction $|\B(\id,s\lie,-)|$ to the $\infty$-category of Postnikov-connective filtered $\F$-modules and use the comonad structure to prove the conjectured pattern of  differentials.

\end{remark}

For small $k$, we use sparsity arguments to show that the weight $k$ part of the Knudsen spectral sequence with $\F$ coefficients always collapses on the $E^2$-page. Thus we obtain an $\F$-bases of $H_*(B_k(M;X);\F)$  for any parallelizable manifold $M$ and spectrum $X$ when $k=2$ in \Cref{weight2}, and for closed parallelizable  $M$ when $k=3$ in \Cref{closedweight3}. In particular, we observe that the $\F$-module $H_*(B_k(M;X))$ depends on and only on the cohomology \textit{ring} $H^*(M^+;\F)$ when $H_*(X;\F)$ has at least two generators. This is in contrast to the case when $X=\mathbb{S}^r$, in that the equivalence (\ref{BCT}) depends only on the $\F$-module $H^*(M;\F)$ \cite{BCT}. As examples, we produce explicit bases for $H_*(B_k(M,X);\F), \ k=2,3$ when $X$ is an arbitrary spectrum, and $M$ is a closed torus or a punctured genus $g$ surface in \Cref{torus}, as well as the (punctured) real projective space $\mathbb{RP}^3$ in \Cref{RP3}. 

When $p>2$, we compute the weight $k\leq p$ part of the $E^2$-page of the Knudsen spectral sequence with $\Fp$-coefficients in terms of certain $\Lie_{\Fp}$-algebra homology in \Cref{oddE2}. We deduce the existence of a single $d_{p-2}$-differential in  the Knudsen spectral sequence when $M=\mathbb{R}^n$ with $n\geq 1$, $X=\mathbb{S}^{2l}$ and $k=p\geq 5$ in \Cref{odddiff}. Then we show that the mod $p$ Knudsen spectral sequence collapses when $k=2$ or $k=3$ and $p\geq 5$.

\begin{corollary} [\Cref{oddsmallweight}]
Let $M^n$ be a parallelizable manifold and $X$ any spectrum. Let  $\g$ be the $\Lie_{\Fp}$-algebra $\widetilde{H}^*(M^+;\Fp)\tens \Lie_{\p}(\Sigma^n H_*(X;\Fp))$ with brackets given by $[y\tens x, y'\tens x']:=(y\cup y')\tens [x,x']$, and $\CE(\g)$ the shifted Chevalley-Eilenberg complex (\Cref{FpCE}).   Denote by $\wt_n(H_{s,t}(\CE(\g))$ the weight $n$ part of the  $\Lie_{\Fp}$-algebra homology of $\g$, whose weight grading is induced by regarding $H_*(X;\Fp)$ as an $\Fp$-module in weight 1.
\begin{enumerate}
    \item For all $i$, there is an isomorphism of $\Fp$-modules $$H_i(B_2(M;X);\Fp)\cong\bigoplus_{s+t=i}\wt_2(H_{s,t}(\CE(\g)).$$ 
    \item If $p\geq 5$, then for all $i$ $$H_i(B_3(M;X);\Fp)\cong\bigoplus_{s+t=i}\wt_3(H_{s,t}(\CE(\g)).$$
\end{enumerate}
\end{corollary}
\begin{remark}
    When $X=\mathbb{S}^r$ and $k=2,3$, we observe that the $\Fp$-module $H_*(B_k(M;\mathbb{S}^r);\Fp)$ depends on and only on  the cohomology \textit{ring}  $H^*(M^+;\Fp)$ when $r+l$ is even, see \Cref{dependoncupproduct}. This is again in contrast to the case when $r+l$ is odd in the equivalence (\ref{BCT}) \cite{BCT}.
\end{remark}

In follow-up work with Matthew Chen \cite{cz}, we build on the odd primary method in this paper and Drummond-Cole-Knudsen's computation of the rational homology of the unordered configurations space $B_k(\Sigma_g;\mathbb{S})$ \cite{dck} to identify the higher differentials in the Knudsen spectral sequence for $H_*(B_k(\Sigma_g;\mathbb{S});\Fp)$. As a result, we show that for $p\geq 3$, the integral homology of $B_k(M;\mathbb{S})$ has no $p$-power torsion for $M$ a closed torus or a punctured genus $g\geq 1$ surface when $k\leq p$. In particular, the latter case serves as a simpler proof of \cite[Theorem 1.10]{bhk}. 
\subsection{Outline}
In \Cref{section2}, we recall the definition of spectral Lie algebras and the structure of their mod 2 homology as algebras over the monad $\Lie_{\R}$. Then we define the Quillen homology of $\Lie_{\R}$-algebras and the mod $p$ topological Quillen homology of spectral Lie algebras. The two are related by a bar spectral sequence.

In \Cref{section 3}, we provide general upper bounds for the  Quillen homology of $\Lie_{\R}$-algebras  and precise formula in low weights by comparing with the Quillen homology of two variant algebras when $p=2$. Then we explicitly compute the Quillen homology of the $\Lie_{\R}$-algebras $\Sigma^k\F$ and $\Sigma^{-n}\free^{\Mod_{\R}}_{\Mod_{\F}}(\Sigma^{n+k}\F)$.

In \Cref{section4}, we review Knudsen's result that expresses labeled configuration spaces in parallelizable manifolds as topological Quillen objects of certain spectral Lie algebras. In the universal case $M=\mathbb{R}^\infty$, we conjecture patterns of higher differentials.

In \Cref{section5}, we apply our understanding of the Quillen homology of $\Lie_{\R}$-algebras to extract explicit information about the mod 2 homology of labeled configuration spaces, including general upper bounds and low weight computations. Then we extend the methods to $p>2$ to study the odd primary homology of labeled configuration spaces in \Cref{section6}. 

\subsection{Acknowledgements.} The author would like to thank  Jeremy Hahn and Haynes Miller for many discussions and encouragement,  Lukas Brantner, Nir Gadish, Mike Hopkins, Ben Knudsen, Nikolai Konovalov,  and Andrew Senger for helpful conversations, as well as the anonymous referees for their patience and detailed comments. 

The author was partially supported by NSF Grant No. DMS-1906072 and the Danish National Research Foundation through the Copenhagen Centre for Geometry and Topology (DNRF151) during the course of this work.

\subsection{Conventions}
We assume that every object is graded and weighted whenever it makes sense.  For instance, $\Mod_{\Fp}$  stands for the ordinary category of weighted graded $\Fp$-modules.
A weighted graded $\F$-module  $M_\bullet$ is an $\N$-indexed collection of $\mathbb{Z}$-graded $\Fp$-modules $\{M(w)_\bullet\}_{w\in\N}$. The weight grading of an element $x\in M(w)_n$ is $w$, and the internal grading is $|x|=n$. Morphisms are weight preserving morphisms of graded $\Fp$-modules. The Day convolution $\tens$ makes $\Mod_{\Fp}$  a symmetric monoidal category. The Koszul sign rule $x\tens y=(-1)^{|x||y|}y\tens x$ for the symmetric monoidal product $\tens$ depends only on the internal grading and not the weight grading.

Similarly, a shifted Lie algebra $L$ over $\Fp$ is a weighted graded $\Fp$-module equipped with a shifted Lie bracket $[- ,-]:L_m\tens L_n\rightarrow L_{m+n-1}$ that adds weights, as well as satisfying graded commutativity $[x,y]=(-1)^{|x||y|}[y,x]$ and the graded Jacobi identity $$(-1)^{|x||z|}[x,[y,z]]+(-1)^{|y||x|}[y,[z,x]]+(-1)^{|z||y|}[z,[x,y]]=0.$$ When $p=3$, we further require that $[[x,x],x]=0$ for all $x\in L$. Denote by $\Lie_{\mathbb{F}_p}$ the category of shifted weighted graded Lie algebras over $\mathbb F_p$, as well as the monad associated to the free $\Lie_{\Fp}$-algebra functor.
When $p=2$, we use the abbreviation $\Lie=\Lie_{\F}$. We further consider the category $\tiLie$ of totally-isotropic $\Lie$-algebras, i.e.,  $\Lie$-algebras that have vanishing self-brackets. We use the notation $\langle-,-\rangle$ exclusively for $\tiLie$ brackets.

We mean by shifted graded exterior algebra over $\Fp$ a graded $\Fp$-module $M_\bullet$ together with a graded commutative product $M_m\wedge M_n\rightarrow M_{m+n-1}$ such that $x\wedge x=0$ for all $x\in M_\bullet$. We will often omit the adjectives shifted graded for the exterior algebra. 

We use $\Fp$ for both the field $\Fp$ and its Eilenberg-MacLane spectrum. The coefficients for the homology group $H_*(-)$ is $\F$ unless specifically stated.

We use $\pi_n(-)$ to denote the following functors: the functor taking the $n$th homotopy group of a spectrum, an $\Fp$-module spectrum, or a simplicial $\Fp$-module, as well as  the functor taking the $n$th homology group of a chain complex over $\Fp$. 

We use $\pi_{*,*}(-)$ to denote the functor taking the bigraded homotopy groups of a (weighted graded) bisimplicial $\Fp$-module, which is equivalent to taking the homology groups of the total complex of the associated double complex via the generalized Eilenberg-Zilber theorem. The bidegree $(s,t)$ is given by the pair (simplicial degree, internal degree).

\section{Preliminaries}\label{section2}

\subsection{The spectral Lie operad}

We begin with a brief review of the spectral Lie operad. 
Ching \cite{ching} and Salvatore \cite{salvatore} showed that the Goodwillie derivatives $\partial_n(\mathrm{Id})$ of the identity functor $\mathrm{Id}:\mathrm{Top}_*\rightarrow\mathrm{Top}_*$  form an operad $s\lie:=\{\partial_n(\mathrm{Id})\}_n$ in Spectra. This operad is Koszul dual to the nonunital commutative operad $\E^{\mathrm{nu}}_\infty$ via the operadic bar construction  $$s\lie\simeq \mathbb{D}\mathrm{Bar}(1,\E_\infty^{\mathrm{nu}}, 1).$$
For a description of the operadic bar construction, see \cite{ching} for a topological model using trees and \cite[Appendix D]{brantner} for an $\infty$-categorical construction along with a comparison with the topological model.

The $n$th-derivative $\partial_n(\mathrm{Id})$ admits an explicit description due to  Arone and Mahowald \cite{am}, following the work of Johnson \cite{johnson}. Let $\pazocal{P}_n$ be the poset of partitions of the set $\underline{n}=\{1,2,\ldots,n\}$ ordered by refinements, equipped with a $\Sigma_n$-action induced from that on $\underline{n}$. Denote by $\hat{0}$ the discrete partition and $\hat{1}$ the partition $\{\underline{n}\}$. Set $\Pi_n=\pazocal{P}_n-\{\hat{0},\hat{1}\}$. Regarding a poset $\pazocal{P}$ as a category, we obtain via the nerve construction a simplicial set $N_\bullet(\pazocal{P})$. The \textit{partition complex} $\Sigma|\Pi_n|^{\diamond}$, the reduced-unreduced suspension of the realization $|\Pi_n|$,  is modeled by the simplicial set
$$N_\bullet(\pazocal{P}_n)/(N_\bullet(\pazocal{P}_n-\hat{0})\cup N_\bullet(\pazocal{P}_n-\hat{1}))$$ for $n\geq 2$ and the simplicial 0-circle $S^0$ for $n=1$.
Then there is an equivalence $$\partial_n(\mathrm{Id})\simeq\mathbb{D}(\Sigma|\Pi_n|^{\diamond})$$  of spectra with $\Sigma_n$-action, where $\mathbb{D}$ denotes the Spanier-Whitehead dual of a spectrum.

\subsection{Operations on the mod 2 homology of spectral Lie algebras}
Next, we describe the structure on the mod 2 homology of an algebra $L$ over the spectral Lie operad. It consists of a $\Lie$-algebra structure along with Dyer-Lashof like unary operations.

The second structure map of a spectral Lie algebra $L$ is given by $$\xi:  \partial_2(\mathrm{Id})\tens_{h\Sigma_2} L^{\tens 2}\simeq \partial_2(\mathrm{Id})\tens L^{\tens 2}_{h\Sigma_2}\simeq \mathbb{S}^{-1}\tens L^{\tens 2}_{h\Sigma_2}\rightarrow L.$$ At the level of homology, this gives rise to a shifted Lie bracket $$[-,-]: H_m(L)\tens H_n(L)\rightarrow H_{m+n-1}(L),$$ making $H_*(L)$ a  graded shifted Lie algebra \cite[Proposition 5.2]{omar}.

For $L$ a connective spectral Lie algebra, Behrens defined unary operations of weight 2
$$\Q^j:H_d(L)\rightarrow H_{d+j-1}(L)$$ on the mod 2 homology of $L$ 
via $x\mapsto \xi_*\sigma^{-1}Q^j(x),$ where $Q^j:H_d(L)\rightarrow H_{d+j}(L^{\tens 2}_{h\Sigma_2})$ is an extended Dyer-Lashof operation $x\mapsto e_{j-d}\tens x\tens x$, $\sigma^{-1}:H_*(L^{\tens 2}_{h\Sigma_2})\rightarrow H_{*-1}(\partial_2(\mathrm{Id})\tens L^{\tens 2}_{h\Sigma_2})$ is the desuspension isomorphism, and $\xi$ is the second structure map \cite[Section 1.5]{behrens}\cite[Definition 5.4]{omar}. Furthermore, Behrens showed that  the quadratic relations
\begin{equation}
   \Q^r \Q^s=\sum^{r-s-1}_{l= 0}\binom{r-2l-1}{s-l} \Q^{r+s-l}\Q^l 
\end{equation}
for $s<r\leq2s$ generate all the relations among the unary operations on a class in some positive degree \cite[Theorem 1.5.1]{behrens}. By definition, for $x$ a homogeneous class $\Q^i (x)=0$ whenever $i< |x|$. Hence $\Q^r\Q^s(x)=0$ for $|x|\geq 1$ and $r\leq s$. 

Since the extended Dyer-Lashof operations are defined on the mod 2 homology of all nonconnective spectra, the operations $\Q^i$ for all $i\in \mathbb{Z}$  can be defined on the mod 2 homology of any spectral Lie algebra $L$ with $\Q^i (x)=0$ for any homogeneous class $x\in H_*(L)$ and $i<|x|$. Let $\R$ be the quotient algebra of the free algebra over $\F$ on generators $\{\Q^j\}_{j\in\mathbb{Z}}$ by the two sided ideal generated by the relations
\begin{equation}\label{beh}
   \Q^r \Q^s=\sum_{l\leq r-s-1}\binom{r-2l-1}{s-l} \Q^{r+s-l}\Q^l
\end{equation}
for all $r\leq 2s$.

The $\Lie$-bracket interacts with the unary operations in the following way.
\begin{proposition}\cite[Lemma 6.4, 6.5]{omar}
  For any $j\in\mathbb{Z}$ and $x,y$ homogeneous classes in the mod 2 homology of a spectral Lie algebra, we have $[\Q^j(x), y]=0$ and $\Q^{|x|}(x)=[x,x]$.  
\end{proposition}
\begin{remark}

It follows that $\Q^{2|x|-1}\Q^{|x|}(x)=[[x,x],[x,x]]=0.$ This is guaranteed by the Behrens' relations, since $r=2|x|-1\leq s=2|x|$ and the right hand side of (\ref{beh}) vanished due to instability of the extended Dyer-Lashof operations. 
\end{remark}

Sometimes it is more convenient to switch to the lower indexing $\Q_j(x):=\Q^{|x|+j}(x)$, which automatically takes into account the instability condition.

\begin{definition}
The lower indexed  $\R$-algebra is generated by symbols $\Q_j$ for $j\geq 0$  and relations \begin{equation}\label{lowerindexrelaitons}
   \Q_a\Q_b=\sum_{0\leq c< (a+2b-1)/3}\binom{a+b-2c-2}{b-c}\Q_{a+2b-2c}\Q_c
\end{equation} 
for $0\leq a\leq b+1$. When $j<0$ we set $\Q_j=0$. 
\end{definition}

\begin{definition}
An $\F$-module $M_\bullet$ over $\R$ is \textit{allowable} if for any  homogeneous element $x\in M_\bullet$ we have $\Q^{j_1}\Q^{j_2}\cdots\Q^{j_m}(x)=0$ whenever $j_1<j_2+\cdots+j_m+|x|$. 
Alternatively, an allowable $\R$-module $M$ is a module over the lower-indexed $\R$-algebra.
\end{definition}
Now we extend Behrens' results to all spectral Lie algebras.
\begin{proposition}\label{Rbarforall}
    For $L$ any spectral Lie algebra, its mod 2 homology $H_*(L)$ is an allowable module over $\R$. Furthermore, for all $k\geq 0$ and $n\in\mathbb{Z}$ there is an isomorphism of $\F$-modules 
    \begin{align*}
      H_*(\partial_{2^k}(\mathrm{Id})\tens_{h\Sigma_{2^k}}(\mathbb{S}^{n})^{\tens 2^k})&\cong \F\{ \Q^{j_1}\cdots\Q^{j_k}(x_n), j_l>2j_{l+1}\forall l<k,\  j_k\geq n\}\\&\cong\F\{\Q_{i_1}\cdots\Q_{i_k}(x_n), \ \forall l,\  i_l\geq0,\   i_{l}> i_{l+1}+1\}.  
    \end{align*}

\end{proposition}
\begin{proof}
The connectedness assumption in Behrens' proof of \cite[Theorem 1.5.1]{behrens} is necessary only because of the connectedness assumption on the following two inputs to the proof. Kuhn \cite[Example 7.6]{kuhn} (see also \cite[Lemma 1.4.3]{behrens}) showed that for $Y$ a connected space, the transfer $\tau: H_*(Y^{\tens4}_{h\Sigma_4})\rightarrow H_*(Y^{\tens4}_{h\Sigma_2\wr\Sigma_2})$ is given by \begin{equation}\label{eq: kuhntransfer}
   Q^r Q^s\mapsto Q^r \wr Q^s + \sum_{t}\Big[\binom{s-r+t}{s-t}+\binom{s-r+t}{2t-r} \Big]Q^{r+s-t}\wr Q^t.
\end{equation} 
On the other hand, Arone and Mahowald's computation \cite[Theorem 3.16]{am} $$H_*(\partial_{2^k}(\mathrm{Id})\tens_{h\Sigma_{2^k}}(\mathbb{S}^{n})^{\tens 2^k})\cong \F\{ \Q^{j_1}\cdots\Q^{j_k}(x_n), j_l>2j_{l+1}\forall l<k, j_k\geq n\}$$  works for any odd integer $n$, and extends to  positive even integers via the fiber sequence $$\partial_{2m}(\mathrm{Id})\tens_{h\Sigma_{2m}}(\mathbb{S}^n)^{\tens 2m}\xrightarrow{E}\Sigma^{-1}\partial_{2m}(\mathrm{Id})\tens_{h\Sigma_{2m}}(\mathbb{S}^{n+1})^{\tens 2m}\xrightarrow{H}\Sigma^{-1}\partial_{m}(\mathrm{Id})\tens_{h\Sigma_{m}}(\mathbb{S}^{2n+1})^{\tens m},$$ which was obtained by differentiating the EHP sequence \cite[Proposition 4.7]{am}\cite[Corollary 2.1.4]{behrens}. Behrens proved the relations by using the transfer formula and inductively checking that they are compatible with the operadic composition; then he provided a basis by comparing with Arone-Mahowald's answer. Hence we only need to remove the connectedness assumption on both inputs.

Note that Kuhn's transfer formula can be obtained as a consequence of the computation of the transfer map $\tau_0: H_*(B\Sigma_4)\rightarrow H_*(B\Sigma_2\wr\Sigma_2)$ on group homology by Priddy \cite[section 4]{transfer}. For any $j,n\in\mathbb{Z}$, the Dyer-Lashof operation $Q^j$ on a class $x$ in degree $n$ is defined via the canonical isomorphism $H_{n+j}((\Sigma^n\F)^{\tens 2}_{h\Sigma_2})\cong H_{j-n}(B\Sigma_2)[2n]$ \cite{maysq}, where $[k]$ denotes a shift in homological degree by $k$. Similarly, the wreath product $Q^r \wr Q^s$ and the weight 4 operation $Q^r Q^s$  are defined in $H_{j-n}(B\Sigma_2\wr\Sigma_2)[4n]$ and $H_{j-n}(B\Sigma_4)[4n]$ respectively, so the transfer map $\tau$ on a class in degree $n$ of any spectrum $Y$ is a shift of $\tau_0$ by $4n$. Hence the formula (\ref{eq: kuhntransfer}) holds for the transfer map $\tau$ on any spectrum $Y$.

Next we extend the computation of Arone-Mahowald to nonpositive spheres. We make use of the long exact sequence $$\cdots\rightarrow H_*(\Sigma^{-2}\partial_{m}(\mathrm{Id})\tens_{h\Sigma_{m}}(\mathbb{S}^{2n+1})^{\tens m})\xrightarrow{P_*} H_*(\partial_{2m}(\mathrm{Id})\tens_{h\Sigma_{2m}}(\mathbb{S}^{n})^{\tens 2m})\xrightarrow{E_*}H_*(\Sigma^{-1}\partial_{2m}(\mathrm{Id})\tens_{h\Sigma_{2m}}(\mathbb{S}^{n+1})^{\tens 2m})\xrightarrow{H_*}\cdots$$ and isomorphisms $$H_*\Big(\partial_{2m-1}(\mathrm{Id})\tens_{h\Sigma_{2m-1}}(\mathbb{S}^{2n})^{\tens (2m-1)}\Big)\cong H_*\Big(\Sigma^{-1}(\partial_{2m-1}(\mathrm{Id})\tens_{h\Sigma_{2m-1}}(\mathbb{S}^{2n+1})^{\tens (2m-1)})\Big)$$ for all $n$ obtained by Brantner in \cite[4.1.3]{brantner}, cf. \cite[Lemma 4.4]{kjaer}.  There is an equivalence of $\F$-module spectra with $\Sigma_m$-action 
$$\partial_{m}(\mathrm{Id})\tens_{h\Sigma_{m}}(\Sigma^{n}\F)^{\tens m})\simeq \Sigma^{2mn}\partial_{m}(\mathrm{Id})\tens_{h\Sigma_{m}}(\Sigma^{-n}\F)^{\tens m})$$
 for any integers $m, n\geq 0$, where the action on $\Sigma^{2mn}$ is trivial. Hence we obtain an isomorphism $$H_*(\partial_{m}(\mathrm{Id})\tens_{h\Sigma_{m}}(\mathbb{S}^{n})^{\tens m}))\cong H_*(\Sigma^{2mn}\partial_{m}(\mathrm{Id})\tens_{h\Sigma_{m}}(\mathbb{S}^{-n})^{\tens m}))$$ sending $\Q_{j_1}\cdots\Q_{j_k}(\iota_{n})$ to $\sigma^{2^{k+1}n}\Q_{j_1}\cdots\Q_{j_k}(\iota_{-n})$ when $m=2^k$, and both vanish when $m\neq 2^k$ for some $k\geq 0$. This addresses the case of the negative spheres.
 
 For $n=0$, we use the long exact sequence above. It follows from the case $n=1$ that $H_*(\partial_{m}(\mathrm{Id})\tens_{h\Sigma_{m}}(\mathbb{S}^0)^{\tens m})=0$ when $m$ is not a power of 2. Now suppose that $m=2^k$. By the \cite[Proposition 2.2.5]{behrens} (cf. remark after \cite[Proposition 4.3]{kjaer}), the maps $E_*$ and $P_*$ preserve the $\Q$ operations, sending the class $\Q^J(x_n)$ to $\sigma^{-1}\Q^J(x_{n+1})$ and $\sigma^{-2}\Q^Jx_{2n+1}$ to  $\Q^J\Q^n(x_{n})$ respectively. This addresses the case $n=0$. 
\end{proof}

Denote by $\Mod_{\R}$ the category of allowable $\R$-modules and $\free^{\Mod_{\R}}_{\Mod_{\F}}$ the free allowable $\R$-module functor, which is left adjoint to the underlying functor $U^{\Mod_{\R}}_{\Mod_{\F}}:\Mod_{\R}\rightarrow\Mod_{\F}$. We will suppress the adjective allowable from here on. Then there is an additive monad associated with the free $\R$-module functor, which we denote by $\pazocal{A}_{\R}$.

\begin{definition}\cite[Definition 6.1]{omar}\label{evenbracketvanish}
An $\Lie_{\R}$-\textit{algebra} is a graded $\F$-module $L_\bullet$ with a shifted Lie bracket and an (allowable) $\R$-module structure on $L_\bullet$ such that

(1) $\Q_0 (x)=\Q^k (x)=[x,x]$ if $x\in L_k$, and

(2) $[x, \Q^k (y)]=0$ for all $x,y\in L_\bullet$.
\end{definition}

Denote by $\Lie_{\R}$ the category of $\Lie_{\R}$-algebras. 
To describe the free $\Lie_{\R}$-algebra functor, we recall the construction of \textit{Lyndon words} on a set $S$, which provides a basis for the free $\tiLie$-algebra on an $\F$-module with $\F$-basis $S$. 

\begin{construction} \cite{hall}
The Lyndon words on a set $S$ is defined recursively as follows:
The elements of $S$ are Lyndon words of length one and given an arbitrary fixed total ordering. 
Suppose that we have defined Lyndon words of length less than $k$ with a total ordering. Then a Lyndon word of length $k$ is a formal bracket $\langle w_1, w_2\rangle$ such that
\begin{enumerate}
    \item $w_1,w_2$ are Lyndon words whose lengths add up to $k$;
    \item $w_1<w_2$ in the order defined thus far;
    \item To take into account the Jacobi identity, if $w_2=\langle w_3, w_4\rangle$ for some Lyndon words $w_3, w_4$, then we require $w_3\leq w_1$.
\end{enumerate}
To extend the total order to Lyndon words of weight at most $k$, we first impose an arbitrary total ordering on Lyndon words of length $k$, and then declare that they are greater than all Lyndon words of lower weights.
\end{construction}

The free $\Lie_{\R}$-algebra functor can  be computed explicitly as follows:

\begin{proposition}\cite[Proposition 7.4]{omar}\label{freeLR}
Let $V_\bullet$ be an $\F$-module with an ordered basis $B$ of $V_\bullet$. First take the free totally isotropic Lie-algebra with $\langle-,-\rangle$ the free $\tiLie$ bracket. Denote by $B'$ the set of Lyndon words on the letters $B$, which is an $\F$-basis of $\free^{\tiLie}_{\Mod_{\F}}(V_\bullet)$.  Then we take the free $\R$-module on the underlying $\F$-module of $\free^{\tiLie}_{\Mod_{\F}}(V_\bullet)$ and obtain a basis consisting of elements of the form $\Q^I w$ with $w\in B'$. Equip the free  $\R$-module $\free^{\Mod_{\R}}_{\Mod_{\F}}(\tiLie(V_\bullet))$ with a  $\Lie$ bracket $[-,-]$ defined on the induced basis by requiring $[\Q^I w_1, \Q^J w_2]=0$ if $I\neq \emptyset$ or $J\neq \emptyset$, and setting recursively along the ordering on $B'$

1) If $\langle w_1, w_2\rangle$ is a Lyndon word, then $[w_1,w_2]=\langle w_1, w_2\rangle$;

2) $[w,w]:=\Q^{|w|}w$;

3) $[w_1, w_2]:=[w_2, w_1]$ if $w_1>w_2$;

4) $[w_1,w_2]:=[w_3,[w_1, w_4]]+[w_4, [w_1, w_3]]$ if $w_1<w_2$ and $w_2=[w_3,w_4]$ with $w_1<w_3$.
\end{proposition}

Antol\'{i}n-Camarena showed that the monad $\Lie_{\R}$ parametrizes natural operations on the mod 2 homology of connected spectral Lie algebras. The connectivity assumption can be removed in view of \Cref{Rbarforall}. Denote by $\free^{s\lie}$ the free spectral Lie algebra functor  on Spectra given explicitly by $$X\mapsto\bigoplus_{n\geq 1} \partial_n(\mathrm{Id})\tens_{h\Sigma_n} X^{\tens n}.$$
\begin{theorem}\cite[Theorem 7.1]{omar}\label{omar}
The canonical map
$$ \free^{\Lie_{\R}}_{\Mod_{\F}}(
H_*(X;\F))\rightarrow H_*(\free^{s\lie}(X);\F)$$  of  $\Lie_{\R}$-algebras is an isomorphism for any spectrum $X$.
\end{theorem}
\begin{proof}
Behrens proved the theorem in the case when $X=\mathbb{S}^k$, $k>0$.  Antol\'{i}n-Camarena proved the isomorphism for $X$ a connected spectrum follows: 
To extend Behrens' theorem to a finite wedge of spheres, he made use of a result of Arone and Kankaarinta that applies Goodwillie calculus to the Hilton-Milnor Theorem \cite[Theorem 0.1]{ak}. To extend to all connected spectra, note that $X\tens \F$ can be written as a filtered  colimit of finite wedges of $\mathbb{S}^m \tens \F$ in the category of $\F$-module spectra. The same arguments work to extend the isomorphism in \Cref{Rbarforall} to all spectra.
\end{proof}

The category $\Mod_{\R}$ is stable under the desuspension functor $\Omega:=\Sigma^{-1}$ of $\F$-modules since the extended Dyer-Lashof operations are. Namely, for $M\in\Mod_{\R}$, the $\F$-module $\Omega M$ has an $\R$-module structure given by $\Q^j(\sigma^{-1}x)=\sigma^{-1}\Q^j(x)$ for any  $x\in M$. As a result, for $\g$ any $\Lie_{\R}$-algebra, there is an $\Lie_{\R}$-structure on $\Omega \g$ such that the bracket is trivial.

\begin{proposition}\label{stablization}
     There is a natural $\Lie_{\R}$-module structure on $\Omega^n \free^{\Mod_{\R}}_{\Mod_{\F}}(\Sigma^{n+k}\F)$ for $1\leq n\leq \infty$, where the bracket is trivial and  $\Q^j$ acts by $x\mapsto \sigma^{-n} \Q^j(\sigma^n x)$. The canonical map $$ \free^{\Mod_{\R}}_{\Mod_{\F}}(\Sigma^{k}\F)\cong\free^{\Lie_{\R}}_{\Mod_{\F}}(\Sigma^{k}\F)\rightarrow\Omega^n \free^{\Mod_{\R}}_{\Mod_{\F}}(\Sigma^{n+k}\F)$$ is surjective.
\end{proposition}
\begin{proof}
    There is a canonical colimit-to-limit comparison map
\begin{equation}
    \free^{s\mathcal{L}}(\Sigma^{k} \F)\rightarrow\Omega\free^{s\mathcal{L}}( \Sigma^{k+1} \F)
\end{equation}
of spectral Lie algebras over $\F$, which after taking homotopy groups is the composite of the top and right arrows of the diagram
\begin{center}
    \begin{tikzcd}
        \free^{\Mod_{\R}}_{\Mod_{\F}}(\Sigma^{k} \F)\ar[r]\ar[d,"ev"]&\free^{\Lie_{\R}}_{\Mod_{\F}}(\Omega\free^{\Mod_{\R}}_{\Mod_{\F}}(\Sigma^{k+1} \F))\ar[d,"ev"]\\
        \Sigma^{k} \F\ar[r,hook,"i"]&\Omega\free^{\Mod_{\R}}_{\Mod_{\F}}(\Sigma^{k+1} \F)
    \end{tikzcd}.
\end{center}
Let $x$ be the generator of $\Sigma^k\F$. By naturality of the $\Q^j$ operation, the class $\Q^j(x)$ on the top left corner is mapped to $\Q^j(i(x))$, which is sent to $\sigma^{-1}\Q^j\sigma(x)$ under evaluation. In general $\Q^J(x)$ is mapped to $\sigma^{-1}\Q^J\sigma(x)$ for any sequence $J$. Since the $\Lie$ bracket of operations always vanishes and $$[i(x),i(x)]=\Q^{|x|}(i(x))=\sigma^{-1}Q^{|x|}(\sigma x)=0,$$ the $\Lie$-bracket is trivial on $\Omega\free^{\Mod_{\R}}_{\Mod_{\F}}(\Sigma^{k+1} \F)$.   Applying \Cref{omar}, we see that the composite is surjective since $|\sigma(x)|=|x|+1$. Iterating the construction yields the claim.
\end{proof}
\subsection{Quillen homology of spectral Lie algebras}
Now we introduce the main object of interest.
The inclusion of trivial $\Lie_{\R}$-algebras admits a left adjoint $Q^{\Lie_{\R}}_{\Mod_{\F}}$ called the indecomposable functor, i.e. we have an adjunction
\begin{center}
\begin{tikzpicture}[node distance=2.8cm, auto]
\pgfmathsetmacro{\shift}{0.3ex}
\node (P) {$\Lie_{\R}$};
\node(Q)[right of=P] {$\Mod_{\F}$\ .};

\draw[transform canvas={yshift=0.5ex},->] (P) --(Q) node[above,midway] {\footnotesize $Q^{\Lie_{\R}}_{\Mod_{\F}}$};
\draw[transform canvas={yshift=-0.5ex},->](Q) -- (P) node[below,midway] {\footnotesize $T^{\Lie_{\R}}_{\Mod_{\F}}$}; 
\end{tikzpicture}
\end{center}
Denote again by $\Lie_{\R}$ the monad associated to the free $\Lie_{\R}$-algebra functor.

We would like to understand the left derived functor of this left adjoint, and we take a small detour to deal with the model structure. We mainly follow Sections 3.1 and 3.2 of \cite{jn} and Section 4 of \cite{bhk}.
 
\subsubsection{The derived indecomposable functor}
 Let $\mathbf{T}$ be an augmented monad on the category $\Mod  _{k}$ of weighted graded $k$-modules, where $k$ is a field. Denote by $\mathrm{Alg}_{\mathbf{T}}(\Mod  _{\F})$ the category of $\mathbf{T}$-algebras. The forgetful functor $U:\mathrm{Alg}_{\mathbf{T}}(\Mod  _k)\rightarrow \Mod  _k$ admits a left adjoint, the free functor $\free^{\mathbf{T}}:\Mod  _k\rightarrow\mathrm{Alg}_{\mathbf{T}}(\Mod_k)$. 

Denote by $s\Mod_k$ the category of simplicial weighted graded $k$-modules. Levelwise application of the adjunction $\free^{\mathbf{T}}\dashv U$ gives rise to an adjunction  between the corresponding categories of simplicial objects $$\free^{\mathbf{T}}\dashv U:\mathrm{Alg}_{\mathbf{T}}(s\Mod_{k})\rightarrow s\Mod_{k},$$ as well as a monad $\mathbf{T}$ on $s\Mod_{k}$. We equip $s\Mod_{k}$ with the standard cofibrantly generated model structure. Suppose that the path objects of $s\Mod_{k}$ lifts to $s\mathrm{Alg}_\mathbf{T}$, the category of simplicial $\mathbf{T}$-algebras. Then this adjunction induces a right transferred model structure on the category of simplicial $\mathbf{T}$-algebras, with weak equivalences and fibrations defined on the underlying simplicial weighted graded $k$-modules by \cite[Theorem 3.2, Remark 3.3]{jn}. In particular, this is true for $\mathbf{T}=\R, \Lie, \tiLie, \Lie_{\R}$, cf. \cite[Proposition 3.4, 4.14]{bhk}.

Denote by $T^{\mathbf{T}}:\Mod  _{k}=\mathrm{Alg}_{\mathbf{\mathrm{Id}}}(\Mod  _{k})\rightarrow\mathrm{Alg}_{\mathbf{T}}(\Mod_{k})$ the inclusion of trivial $\mathbf{T}$-algebras, which is induced by the augmentation. It has a left adjoint $Q^{\mathbf{T}}:\mathrm{Alg}_{\mathbf{T}}(\Mod  _{k})\rightarrow \Mod_{k}$, the indecomposable functor with respect to the $\mathbf{T}$-algebra structure, which satisfies $Q^{\mathbf{T}}\circ \free^{\mathbf{T}}\simeq \id$. Applying this adjunction levelwise to the corresponding categories of simplicial objects, we obtain a Quillen adjunction  $$Q^{\mathbf{T}}\dashv T^{\mathbf{T}}:s\mathrm{Alg}_\mathbf{T}\rightarrow s\Mod_{k}.$$ The left derived functor $\mathbb{L}Q^{\mathbf{T}}$ of $Q^{\mathbf{T}}$ can be computed by the following standard recipe.

\begin{construction}\label{constr: simplicial bar}
Given a right module $R:\Mod_{k}\rightarrow\mathcal{D}$ over $\mathbf{T}$, and a simplicial object $A$ in $\mathrm{Alg}_{\mathbf{T}}(\Mod _{k})$, one can apply the two-sided bar construction $\B(R,\mathbf{T},-)$ levelwise to $A$. The diagonal of the resulting bisimplicial complex is a simplicial object in $\mathcal{D}$, denoted by $\B(R,\mathbf{T}, A)$.
\end{construction}
In particular, if we regard a $\mathbf{T}$-algebra $A$ as the constant simplicial object on $U(A)$ equipped with a simplicial $\mathbf{T}$-algebra structure, denoted also as $A$ by abuse of notation, then $\B(R,\mathbf{T}, A)$ agrees with the usual two-sided bar construction.

Since the free resolution $\B(\free^{\mathbf{T}},\mathbf{T},A)$ is a cofibrant replacement of $A$ in the category of simplicial $\mathbf{T}$-algebras, the left derived functor of a functor $F$ can be computed by applying $F$ levelwise to a cofibrant replacement, so $$\mathbb{L}Q^{\mathbf{T}}(A)\simeq Q^\mathbf{T}\B(\free^{\mathbf{T}},\mathbf{T},A)=\B(\mathrm{id},\mathbf{T},A). $$ 

Now suppose that we have a composite monad $\mathbf{
R
}\circ \mathbf{L}$ in $\Mod_k$ with distributive law the natural transformation $\mathbf{L}\circ \mathbf{R}\Rightarrow\mathbf{R}\circ \mathbf{L}$ in the sense of Beck \cite[Section 1]{beck}. Suppose in addition that $\mathbf{L}, \mathbf{R}$ and $\mathbf{R}\circ \mathbf{L}$ are all compatibly augmented and each admit a cofibrant replacement given by the free resolution. 
Let $\mathrm{Alg}_{\mathbf{L}},\mathrm{Alg}_{\mathbf{R}}, \mathrm{Alg}_{\mathbf{R}\circ\mathbf{L}}$ be the respective categories of algebras. Then an $\mathbf{R}\circ\mathbf{L}$-algebra $A$ is an $\mathbf{R}$-algebra via the forgetful map $U^{\mathbf{R}\circ\mathbf{L}}_{\mathbf{R}}:\mathrm{Alg}_{\mathbf{R}\circ\mathbf{L}}\rightarrow\mathrm{Alg}_{\mathbf{R}}$ induced by the augmentation of $\mathbf{L}$, and an $\mathbf{L}$-algebra via the augmentation of $\mathbf{R}$.
Furthermore, we have adjunctions 
\begin{center}
\begin{tikzpicture}[node distance=2.8cm, auto]
\pgfmathsetmacro{\shift}{0.3ex}
\node(P) {$\Mod _{k}$};
\node(Q)[right of =P] {$\mathrm{Alg}_{\mathbf{L}}$};
\node(R)[right of =Q]{$\mathrm{Alg}_{\mathbf{R}\circ\mathbf{L}}$};
\draw[transform canvas={yshift=-0.5ex},->] (P) --(Q) node[below,midway] {\footnotesize $T^{\mathbf{L}}$};
\draw[transform canvas={yshift=0.5ex},->](Q) -- (P) node[above,midway] {\footnotesize $Q^{\mathbf{L}}$}; 
\draw[transform canvas={yshift=-0.5ex},->] (Q) --(R) node[below,midway] {\footnotesize $T^{\mathbf{R}\circ\mathbf{L}}_{\mathbf{L}}$};
\draw[transform canvas={yshift=0.5ex},->](R) -- (Q) node[above,midway] {\footnotesize $Q^{\mathbf{R}\circ\mathbf{L}}_{\mathbf{L}}$}; 
\end{tikzpicture}.
\end{center}

\begin{construction}\label{constrn:composite simplicial}
For $A$ an algebra over $\mathbf{R}\circ\mathbf{L}$, the free resolution $\B(\free^{\mathbf{R}}, \mathbf{R},A)$ has the structure of a simplicial $\mathbf{R}\circ\mathbf{L}$-algebra given as follows. Levelwise, the $\mathbf{R}\circ\mathbf{L}$-algebra structure map is given by $$(\mathbf{R}\circ\mathbf{L})\circ\mathbf{R}^{\circ n}(A)\rightarrow\mathbf{R}\circ(\mathbf{R}\circ\mathbf{L})\circ\mathbf{R}^{\circ (n-1)}(A)\rightarrow\cdots\rightarrow\mathbf{R}^{\circ n}\circ(\mathbf{R}\circ\mathbf{L})(A)\rightarrow\mathbf{R}^{\circ n}(A),$$ where the rightmost arrow is the $\mathbf{R}\circ\mathbf{L}$-algebra structure map on $A$ and the other arrows are induced from the distributive law $\mathbf{L}\circ\mathbf{R}\Rightarrow\mathbf{R}\circ\mathbf{L}$. The face and degeneracy maps are structure maps of the monad $\mathbf{R}$ and hence compatible with the levelwise $\mathbf{R}\circ\mathbf{L}$-algebra structure maps by naturality of the distributive law. 

Levelwise application of $Q^{\mathbf{R}\circ\mathbf{L}}_{\mathbf{L}}$ to $\B(\free^{\mathbf{R}}, \mathbf{R},A)$  yields a simplicial $\mathbf{L}$-algebra structure on the bar construction $\B(\id, \mathbf{R},A)=Q^{\mathbf{R}\circ\mathbf{L}}_{\mathbf{L}}\B(\free^{\mathbf{R}}, \mathbf{R},A)$.
\end{construction} 
We record the following result about factoring the left derived functor of the indecomposable functor of a composite monad, which generalizes \cite[Proposition 4.19]{bhk}.
\begin{lemma}\label{factortiLierg}
    Let $A$ be an $\mathbf{R}\circ\mathbf{L}$-algebra. The homotopy groups of $\B(\id,\mathbf{R}\circ\mathbf{L},A )$ are computed by the homotopy groups of the bisimplicial object $\B(\id,\mathbf{L},\B(\id,\mathbf{R},A))$.
\end{lemma}
 Recall that the homotopy groups of a bisimplicial $k$-module can be computed via the Eilenberg-Zilber theorem, i.e. by first taking associated chain complexes in both directions and then forming the total complex of the double complex. See for instance \cite[Chapter 4]{GJ}.
 \begin{proof}
     The augmentation $\mathbf{R}\circ\mathbf{L}\rightarrow\mathbf{R}$ induces a map of simplicial $\mathbf{R}\circ\mathbf{L}$-algebras $$\Psi:\B(\free^{\mathbf{R}\circ\mathbf{L}},\mathbf{R}\circ\mathbf{L},A)\rightarrow\B(\free^{\mathbf{R}},\mathbf{R},A),$$ where the simplicial $\mathbf{R}\circ\mathbf{L}$-algebra structure on the target is given by \Cref{constrn:composite simplicial}. This is an equivalence since both are free resolutions of $A$ as an $\mathbf{R}\circ\mathbf{L}$-algebra and an $\mathbf{R}$-algebra respectively, and weak equivalences in $s\mathrm{Alg}_{\mathbf{R}\circ\mathbf{L}}$ are detected by the underlying simplicial $k$-modules. We want to show that $Q^{\mathbf{R}\circ\mathbf{L}}_{\mathbf{L}}$ preserves this weak equivalence. Since $U^{\mathbf{L}}$ preserves weak equivalences, it suffices to show that  $U^{\mathbf{L}}\circ Q^{\mathbf{R}\circ\mathbf{L}}_{\mathbf{L}}\circ \Psi$ is a weak equivalence.

     Note that there is an isomorphism $$Q^{\mathbf{R}}\circ U^{\mathbf{R}\circ\mathbf{L}}_{\mathbf{R}} \cong U^{\mathbf{L}}\circ Q^{\mathbf{R}\circ\mathbf{L}}_{\mathbf{L}}.$$ Hence   $U^{\mathbf{L}}\circ Q^{\mathbf{R}\circ\mathbf{L}}_{\mathbf{L}}\circ\Psi$ is the map 
     \begin{align*}
        Q^{\mathbf{R}}\circ U^{\mathbf{R}\circ\mathbf{L}}_{\mathbf{R}} \B(\free^{\mathbf{R}\circ\mathbf{L}},\mathbf{R}\circ\mathbf{L},A)&\rightarrow Q^{\mathbf{R}}\circ U^{\mathbf{R}\circ\mathbf{L}}_{\mathbf{R}}\B(\free^{\mathbf{R}},\mathbf{R},A)\\&\simeq Q^{\mathbf{R}}\B(\free^{\mathbf{R}},\mathbf{R},A)\simeq\B(\id,\mathbf{R},A).
     \end{align*}
     Since both $U^{\mathbf{R}\circ\mathbf{L}}_{\mathbf{R}} \B(\free^{\mathbf{R}\circ\mathbf{L}},\mathbf{R}\circ\mathbf{L},A)$ and $\B(\free^{\mathbf{R}},\mathbf{R},A)$ are free resolutions of $A$ in $s\mathrm{Alg}_{\mathbf{R}}$ and $Q^{\mathbf{R}}$ is a left Quillen functor, this is indeed a weak equivalence. Hence 
     \begin{align*}
         \B(\id,\mathbf{L},\B(\id,\mathbf{R},A))&\simeq Q^{\mathbf{L}}\circ Q^{\mathbf{R}\circ\mathbf{L}}_{\mathbf{L}}\B(\free^{\mathbf{R}\circ\mathbf{L}},\mathbf{R}\circ\mathbf{L},A)\\
         &\simeq Q^{\mathbf{L}}\circ Q^{\mathbf{R}\circ\mathbf{L}}_{\mathbf{L}}\B(\free^{\mathbf{R}},\mathbf{R},A)\\
        & \simeq Q^{\mathbf{L}}\B(\id,\mathbf{R},A),
     \end{align*}
     and the lemma follows from \Cref{constr: simplicial bar}.
 \end{proof}

\subsubsection{Quillen homology of $s\Lie_{R}$-algebras}
Since the path object of $s\Mod_{\F}$ lifts to $s\Lie_{R}$, the discussion in the previous subsection guarantees that any $\Lie_{\R}$-algebra $\g$ admits a free  resolution $\B(\free^{\Lie_{\R}}_{\Mod_{\F}}, \Lie_{\R}, \g)$ in $\Lie_{\R}$.
The left derived functor of $Q^{\Lie_{\R}}_{\Mod_{\F}}$ is thus computed by 
$$\mathbb{L}Q^{\Lie_{\R}}_{\Mod_{\F}}(\g)\simeq Q^{\Lie_{\R}}_{\Mod_{\F}}\B(\free^{\Lie_{\R}}_{\Mod_{\F}}, \Lie_{\R}, \g)\simeq \B(\mathrm{id}, \Lie_{\R}, \g),$$
where $\mathrm{id}:\Mod _{\F}\rightarrow\Mod _{\F}$ is the identity functor considered as the trivial right module over the monad $\Lie_{\R}$ with structure map the augmentation.

\begin{definition}
The \textit{Quillen homology} of a $\Lie_{\R}$-algebra $\g$, denoted by $\AQ_*(\g)$, is the total left derived functor  $$\AQ_{*,*}(\g):=H_{*,*}\mathbb{L}Q^{\Lie_{\R}}_{\Mod_{\F}}(\g)\simeq \pi_{*,*}\B\big(\mathrm{id}, \Lie_{\R}, \g\big).$$

\end{definition}
We are interested in computing the Quillen homology of $\Lie_{\R}$-algebras, since it helps to understand the spectral Lie analog of the mod $p$ topological Andr\'{e}-Quillen homology of nonunital $\mathbb{E}_\infty$-algebras introduced by Kriz \cite{krizbp} and  Basterra \cite{taq}.

\begin{definition}
    For $L$ a spectral Lie algebra, its \textit{topological Quillen object} is the bar construction $$\TQ(L):=|\B(\id, s\lie, L)|.$$ We define its mod $p$ \textit{topological Quillen homology} to be  $$\TQ_{*}(L; \Fp):=\pi_*(|\B(\id, s\lie, L)|\tens\Fp).$$
\end{definition}
Using the skeletal filtration of the geometric realization of the bar construction, we obtain a bar spectral sequence
$$ E^2_{s,t}=\pi_s\pi_t \B\big(\mathrm{id}, s\lie, L\tens \Fp)\Rightarrow \TQ_{s+t}(L;\Fp)$$ converging to the mod $p$ topological Quillen homology. When $p=2$, we can apply Theorem \ref{omar} repeatedly and  deduce that:

\begin{proposition}\label{barsseq}
    There is a bar spectral sequence
    $$E^2_{s,t}=\pi_{s,t}\B\big(\mathrm{id}, \Lie_{\R},\pi_*(L\tens\F)\big)\cong \AQ_{s,t}(H_*(L;\F))\Rightarrow \TQ_{s+t}(L;\F).$$
\end{proposition}

\section{Computing the Quillen homology of spectral Lie algebras}\label{section 3}
In this section, we study the Quillen homology of  $\Lie_{\R}$-algebras when $p=2$ via comparison with two smaller double complexes that are easy to compute via Koszul duality arguments.

\subsection{May-type spectral sequence and an upper bound}
First we find an upper bound 
 for $\pi_{*,*}\B(\id,\Lie_{\R},\g)$ by constructing a May-type spectral sequence. The dimensions of its $E^1$-page is bounded above by the homotopy groups of the bar construction of the following variant of $\Lie_{\R}$-algebras whose unary and binary operations do not intertwine. 

\begin{definition}\label{DefinetiLieR}
Define a $\tiLie_{\R}$-algebra to be an $\F$-module $L$  with an allowable $\R$-module structure and a $\tiLie$-bracket $\langle-, -\rangle$ such that $\langle x,\Q^i(y)\rangle=0$ for all $x,y\in L$.
Denote by $\tiLie_{\R}$ the category of $\tiLie_{\R}$-algebras and also the monad associated to the free $\tiLie_{\R}$-algebra functor. 
\end{definition}

The underlying $\F$-module of the free $\tiLie_{\R}$-algebra on an $\F$-module $V$ is given by that of $\A_{\R}\circ\tiLie(V)$. Hence $\tiLie_{\R}$ admits an alternative description as the category of algebras over the composite monad $\A_{\R}\circ\tiLie$, with distributive law the natural transformation $\tiLie\circ\A_{\R}\Rightarrow\A_{\R}\circ\tiLie$ determined by $\langle -, \Q^i(-) \rangle=0$ for all $i$.
 \begin{remark}\label{levelwiseisom}
    Comparing with \Cref{freeLR}, we see that the underlying $\R$-modules of the free $\Lie_{\R}$ and $\tiLie_{\R}$-algebra on any $\F$-module agree. The only difference between the two free functors is that in the latter we do not change the $\tiLie$-algebra to a $\Lie$-algebra via the identification $\Q_0(x)=[x,x]$.
\end{remark}
 In particular, the bar construction $\B(\id,\Lie_{\R},\g)$ is levelwise isomorphic to $\B(\id,\tiLie_{\R},\g)$. The latter has simpler face maps in the sense that the face maps preserve the unary and binary structures respectively, whereas in the former, a Lie bracket that is not a self-bracket can be mapped to a self-bracket. To deal with these face maps, we draw inspiration from the May spectral sequence: suppose that we want to compute the Ext groups over a Hopf algebroid $(A,\Gamma)$. In good cases, there exists a filtration on $\Gamma$ such that the associated graded is a Hopf algebra $(A, \Gamma')$, i.e. the left and right unit are equal. Then we obtain a May spectral sequence with $E^1$-page the Ext group over the Hopf algebra $\Gamma'$, whose cochain complex has differentials simpler than the cobar complex for $\Gamma$. The higher differentials are determined by the difference between the left and the right unit.
 
 To construct a filtration on $\B(\id,\Lie_{\R},\g)$ so that the associated graded assembles to $\B(\id,\tiLie_{\R},\g)$,  first we need to construct a filtration on any $\Lie_{\R}$-algebra so that the two sides of the identification $\Q_0(x)=[x,x]$ live in different filtrations.

 \begin{construction}[Length filtration]\label{lengthfiltration}
    Consider the complete filtration $$\cdots\rightarrow\R(n)\rightarrow\R(n-1)\rightarrow\cdots\rightarrow \R(1)\rightarrow\R$$ of the homogeneous algebra $\R$, where $\R(n)$ is the ideal generated by monomials $\Q^I$ with $|I|=n$. Thus we obtain functors $\A_{\R(n)}$ on $\Mod_{\F}$, sending $M$ to the submodule of $\A_{\R}(M)$ consisting of $\Q^{I}(x)$ for $x\in M$ and $|I|\geq n$. In other words, this assembles to a filtered monad $\hat{\pazocal{A}}_{\R}$. The images of the induced evaluation maps $\A_{\R(q)}(M)\xrightarrow{\mathrm{ev}_q} M$ form a complete decreasing filtration for any $\R$-module $M$. Taking cokernels yields a complete increasing filtration $$\mathrm{F}_l^q(M)=\mathrm{coker}(\A_{\R(q)}(M)\xrightarrow{\mathrm{ev}_q} M).$$ We call this the \textit{length filtration} of $M$, which gives rise to a filtered object $\hat M$ as an algebra over the filtered monad $\hat{\pazocal{A}}_{\R}$ whose underlying $\R$-module is $M$.
\end{construction}

Given an arbitrary $\Lie_{\R}$-algebra $\g$, we would like  to equip $\g$ with the structure of an $\tiLie_{\R}$-algebra. This boils down to producing a method that equips any $\Lie$-algebra with a $\tiLie$-algebra structure.
\begin{construction}($\tiLie$-structure on $\Lie$-algebras.) \label{tibracket}
    For $\g$ is $\Lie$-algebra with bracket $[-,-]$, let $V'$ be the ideal of self-brackets. Thus we obtain a two-step filtration $V'\rightarrow\g$ of $\g$. Denote  by $\langle -, -\rangle$ the canonical $\tiLie$-bracket on the quotient $V=\g/V'=Q^{\Lie}_{\tiLie}(\g)$ and consider $V'$ as a trivial $\tiLie$-algebra. Thus we obtain a $\tiLie$-structure on the associated graded of $\g$ as the product of $V$ and $V'$ with the above $\tiLie$-structures. Denote by $\tilde{\g}$ the resulting $\tiLie$-algebra with $\langle-,-\rangle$ the $\tiLie$-bracket.
\end{construction}

Therefore, any $\Lie_{\R}$-algebra $\g$ admits a $\tiLie$-structure that is unique up to a choice of splitting of $\g\rightarrow V$. Denote by $\tilde{\g}$ the resulting $\tiLie_{\R}$-algebra, which has the same underlying $\R$-module structure as $\g$, cf. \Cref{levelwiseisom}.
\begin{remark}\label{rewritebracket}
    If we fix a choice of splitting for $\g\rightarrow V$, then any $\Lie$ bracket $[x,y]$ in $\g$ is equal to a sum of self-brackets and the $\tiLie$-brackets $\langle x,y\rangle$ in $\tilde{\g}$.
\end{remark}

Now we compute the $E^2$-page of the bar spectral sequence by constructing a May-type spectral sequence in the sense that the filtration comes from the length filtration of $\R$-modules in \Cref{lengthfiltration}.

\begin{theorem}\label{upperbound}
Let $\g$ be a $\Lie_{\R}$-algebra and $\tilde\g$ the associated $\tiLie_{\R}$-algebra via \Cref{tibracket}. Then there is a May-type spectral sequence with respect to the $\R$-module structure converging to  $\pi_{s,t}\B(\id, \Lie_{\R},\g).$
The $E^1$-page $\bigoplus_{q}E^1_{q,s,t}$ of the May-type spectral sequence has
dimensions bounded above by $\pi_{s,t}\B(\id, \tiLie_{\R},\tilde{\g})$, in the sense that there is another algebraic spectral sequence converging to the May $E^1$-page whose $E^1$-page is $\pi_{s,t}\B(\id, \tiLie_{\R},\tilde{\g})$.
\end{theorem}
\begin{proof}
We start by inductively constructing a filtration on $(\Lie_{\R})^{\circ n}(\g)$ that heuristically count the number of $\Q$ symbols in a given element.

Since any $\Lie_{\R}$-algebra $\g$ is an $\R$-module, it admits a length filtration. The filtration is compatible with the bracket since brackets of operations always vanish (\Cref{evenbracketvanish}). Furthermore, since any self-bracket $[x,x]=\Q_0(x)$ is in $\mathrm{F}_l^1(\g)$ and the right hand side is zero in $\mathrm{F}_l^0(\g)$, we deduce that $\mathrm{Gr}_l^0(\g)$ is a $\tiLie$-algebra, and the $\tiLie$-structure can be extended to $\bigoplus \mathrm{Gr}^q(\g)$ via trivial extension to positive $q$. On the other hand,  $\bigoplus \mathrm{Gr}^q(\g)$ is an $\R$-module since $\R$ is homogeneous. Hence $\tilde{g}=\bigoplus \mathrm{Gr}_l^q(\g)$ equipped with the $\tiLie$-bracket in \Cref{tibracket} is an algebra over the composite monad $\tiLie_{\R}=\A_{\R}\circ\tiLie$. 
  
Now we define a new filtration $F^\bullet$ on $\Lie_{\R}(\g)$ that combines the length filtration on $\g$, the length filtration on $\Lie_{\R}(M)$ for any $\F$-module $M$, and the effect of $\Lie$ brackets. We extend the length filtration on $\g$ to $\tiLie(\g)$ via the Day convolution, i.e. for $x\in \mathrm{F}_l^q(\g), y\in \mathrm{F}_l^r(\g)$, we have $\langle x, y\rangle\in F^{q+r}(\tiLie(\g))$, so on and so forth. Then we extend it to a new filtration on $\Lie_{\R}(\g)$ by combining with the length filtration on $\Lie_{\R}(M)$ for $M$ an $\F$-module, using the fact that when $\g=M$ is an $\F$-module,  $\mathrm{Gr}_l^0(\Lie_{\R}(M))=\tiLie(M)$. In particular, after passing to the associated graded, the evaluation map $\Lie_{\R}(\g)\rightarrow\g$ assembles to the $\tiLie_{\R}$-algebra structure map $ev:\A_{\R}\circ \tiLie(\tilde{\g})\rightarrow\tilde\g$. This is because for $x\in\g$, $[x,x]=\Q_0|x\in\Lie_{\R}(\g)$ is mapped to a nonzero element only if $x\in\mathrm{F}^0_l(\g)$, in which case $\Q_0|x\in F^1\Lie_{\R}(\g)$ and $\Q_0(x)\in\mathrm{F}^1_l(\g)$ while $[x,x]\in F^0 \Lie_{\R}(\g)$.

Iterating this process, we obtain a filtration $F^\bullet$ on $\Lie_{\R}\circ(\Lie_{\R})^{\circ n}(\g)$ for all $n>0$ by combining the filtration $F^\bullet$ on $(\Lie_{\R})^{\circ n}(\g)$ with the length filtration on $\Lie_{\R}$. This is the $n$th simplicial level of $\B(\id,\Lie_{\R},\g)$, with associated graded assembling to $(\tiLie_{\R})^{\circ n}(\tilde{\g})$. Explicitly,
$F^q\Big((\Lie_{\R})^{\circ n}(\g)\Big)$ is the collection of elements $\alpha|x$ in simplicial degree $n$ satisfying the following condition: if we rewrite $\alpha|x$ as an element in $(\tiLie_{\R})^{\circ n}(\g)$ via \Cref{levelwiseisom} and \Cref{rewritebracket}, so any $\Lie$ bracket in $\alpha|x$ is written as a linear combination of $\tiLie$ brackets and $\Q_0$ applies to other elements, then  the sum of the filtration degree of $x\in\g$ times the number of times $x$ appears and the number of  symbols $\Q^j$ in any term of $\alpha|x$ coming from applications of the monad $\tiLie_{\R}$ is at most $q$. 

Since $\R$ is a homogeneous algebra and evaluation of brackets do not increase the number of $\Q^j$'s in the expression, the structure map $\Lie_{\R}(\g)\rightarrow\g$ is compatible with this filtration, and so are the face maps and degeneracy maps in $\B(\id,\Lie_{\R},\g)$. The induced filtration $F^\bullet$ on the normalized complex of $\B(\id,\Lie_{\R},\g)$ gives rise to a May-type spectral sequence $$\bigoplus_{q}E^1_{q,s,t}=\bigoplus_{q}\pi_{s,t}\mathrm{Gr}^q\B(\id,\Lie_{\R},\g)\Rightarrow\pi_{s,t}\B(\id, \Lie_{\R},\g).$$ 
Note that the face maps $$(\tiLie_{\R})^{\circ n}(\tilde{\g})=\bigoplus_q\mathrm{Gr}^q\mathrm{Bar}_n(\id,\Lie_{\R},\g)\rightarrow(\tiLie_{\R})^{\circ (n-1)}(\tilde{\g})=\bigoplus_q\mathrm{Gr}^q\mathrm{Bar}_{n-1}(\id,\Lie_{\R},\g)$$ do not change the associated graded degree unless the differential creates self-brackets -- evaluating the $\R$-module structure or a $\tiLie$-bracket is either zero or does not change the number of $\Q$ symbols. Hence they assembles to the $\tiLie_{\R}$-algebra structure maps $(\tiLie_{\R})^{\circ j}(\tilde\g)\rightarrow(\tiLie_{\R})^{\circ (j-1)}(\tilde\g)$ except in the following situation: for $x\in\mathrm{Gr}^0_l(\g)$, the classes $\gamma_1(\Q^i|x):=[\Q^i|1|x,1|\Q^i|x]$ and $\Q_0|\Q^i|x$ are both in the second associated graded piece. Hence the total differential $\partial$ of the normalized complex of $\B(\id,\Lie_{\R},\g)$ sends $$\gamma_1(\Q^i|x):=[\Q^i|1|x,1|\Q^i|x]$$ to the element $$[\Q^i|x,\Q^i|x]+[\Q^i|x,1|\Q^i(x)]=\Q_0|\Q^i|x+[\Q^i|x,1|\Q^i(x)]$$ in $E^0_{2,*,*},$ i.e. the self-bracket has not been filtered away. Similarly, any class containing $\gamma_1(\Q^i|x)$ with $x\in\mathrm{Gr}^0_l(\g)$ has a face map whose target has at least one self-bracket term.   Whereas when $x\in\mathrm{F}^1_l(\g)$, the self-brackets in the target of such differentials are not visible in the associated graded because the number of $\Q^j$'s in the term decrease after we rewrite the self-brackets in terms of $\Q_0$. 

Hence we need to construct another spectral sequence to compute the $E^1$-page of the May-type spectral sequence. To further filter away the self-brackets in such differentials, we assign weight 1 to $\gamma_1(\Q^i|x)$ and $[\Q^i|x,1|\Q^i(x)]$ for all $i$ and $x\in\mathrm{Gr}^0_l(\g)$, weight 0 to everything else in $\g$, $\Lie_{\R}(\g)$, and $\Lie_{\R}\circ\Lie_{\R}(\g)$, including  $\Q_0|\Q^i|x$. Then we propagate the weight to  
$(\Lie_{\R})^{\circ n}(\g)$ for $n>2$ by stipulating that further applying operations $\Q$ does not change weight and brackets add weights. The associated graded of this weight filtration is precisely $\B(\id,\tiLie_{\R},\tilde\g)$, since the only face or degeneracy maps that are altered are the ones involving  $\gamma_1(\Q^i|x)$ for 
$x\in\mathrm{Gr}^0_l(\g)$, whose target no longer contains the self-bracket term $\Q_0|\Q^i|x$. Therefore we obtain an algebraic spectral sequence that converges to the $E^1$-page of the May-type spectral sequence. Then its $E^1$-page has dimensions precisely those of $\pi_{s,t}\B(\id,\tiLie_{\R},\tilde\g)$. Therefore we obtain an upper bound of the dimension of the $E^1$-page of the May-type spectral sequence $\bigoplus_{q}E^1_{q,s,t}$.
\end{proof}
We will call the spectral sequence above computing the $E^1$-page of the May-type spectral sequence the  $\gamma_1$-\textit{Bockstein spectral sequence}.
Since differentials preserve weights and the $\gamma_1$ operation on $\B(\id, \tiLie_{\R},L)$ appears in weight at least four, we immediately deduce the following from \Cref{upperbound}.
\begin{corollary}\label{smallweightisom}
    For any $\Lie_{\R}$-algebra $\g$, the homotopy groups of $\B(\id, \Lie_{\R},\g)$ and $\B(\id, \tiLie_{\R},\tilde\g)$ are isomorphic in weight less than four.
\end{corollary}
\begin{proof}
    In the  $\gamma_1$-Bockstein spectral sequence, the differentials do not appear until weight 4 since $\gamma_1(\Q_j|x)$ has weight 4.  By construction, differentials in the May-type spectral sequence occur when the source and target of a face map in $\B(\id,\Lie_{\R},\g)$ have different number of self-brackets. In weight two and three this cannot happen. Hence both spectral sequences collapse in weight less than four.
\end{proof}

\begin{remark}\label{rmk: modified may}
    In the case where the $\Lie_{\R}$-algebra $\g$ has vanishing Lie brackets, we can combine the two spectral sequences in \Cref{upperbound} into a single spectral sequence that converges to the $E^2$-page of the bar spectral sequence by simply shifting the filtration of any element in $\g$ up by 2. Since $\Q$ operations are additive and the Lie structure is trivial, the resulting filtered object $\g[2]$ is a module over the filtered monad $\hat{\pazocal{A}}_{\R}$, whose underlying object is still the $\R$-module $\g$. Then in the bar construction $\B(\id,\Lie_{\R},\g)$, the class $\Q_0|\alpha$ always lives in a lower filtration than $[\alpha,\alpha]$ for any class $\alpha$, since the filtration of $\alpha$ is at least 2. Hence the May-type spectral sequence has $E^1$-page given precisely by $\pi_{s,t}\B(\id, \tiLie_{\R},\tilde{\g})$.
\end{remark}

Since $\tiLie_{\R}=\pazocal{A}_{\R}\circ \tiLie$ is a composite monad, we apply  \Cref{constrn:composite simplicial} and \Cref{factortiLierg} to compute the homotopy groups of $\B(\id, \tiLie_{\R}, L)$ for $L$ a $\tiLie_{\R}$-algebra.
\begin{construction}\label{AR}
    For $L$ a $\tiLie_{\R}$-algebra with $\tiLie$-bracket $\langle -, -\rangle$, denote by $\AR(L)$ the bar construction $\B(\id,\A_{\R},L)$ equipped with a $\tiLie$-bracket $\langle -,-\rangle$ given levelwise by 
    $$\langle \alpha_1|\alpha_2|\ldots|\alpha_n|x, \beta_1|\beta_2|\ldots|\beta_n|y \rangle
= \left\{ \begin{array}{lcl}
 1|\cdots|1|\langle x,y\rangle & \mbox{if} &  \alpha_i=\beta_i=1, 1\leq i \leq n \\
0 &  &\mbox{otherwise} 
\end{array}\right.,$$
where $\alpha_i, \beta_j\in\R$ and $x,y\in L$. Here we use $L$ to mean the underlying $\R$-module $U^{\tiLie_{\R}}_{\Mod_{\R}}(L)$.
\end{construction}
\begin{corollary}\label{cor:factortiLierg}
For $L$ a $\tiLie_{\R}$-algebra with $\tiLie$-bracket $\langle -,-\rangle$, there is an isomorphism of bigraded homotopy groups
    $$\pi_{*,*}\B(\id,\tiLie_{\R},L )\cong \pi_{*,*}\B(\id, \tiLie,\AR(L)).$$  
\end{corollary}

\subsection{Homology groups of simplicial $\tiLie$-algebras.}
The homotopy groups of $\B(\id, \tiLie,V_\bullet)$ for $V_\bullet$ a simplicial $\tiLie$-algebra can be computed via a shifted version of the classical Chevalley-Eilenberg complex.

Recall from \cite{CE}, \cite[Section 5]{may} and \cite{priddy} that given a $\mathrm{Lie}^{\mathrm{ti}}$-algebra $L$, i.e., an unshifted totally isotropic ungraded Lie algebra over $\F$, its 
$\mathrm{Lie}^{\mathrm{ti}}$-algebra homology is computed by
$$H^{\mathrm{Lie}^{\mathrm{ti}}}_*(L):=H_*(\mathbb{L}Q^{\mathrm{Lie}^{\mathrm{ti}}_{\F}}_{\Mod_{\F}}(L)[1]\oplus \F)=H_*(CE(L)).$$ Here $CE(L)$ is the standard Chevalley-Eilenberg complex, defined to be the exterior algebra on $L[1]$ with differential 
$\delta$ given by
$$\delta(\sigma x_1\tens\cdots \tens \sigma x_n)=\sum_{1\leq i<j\leq n}[\sigma x_i,\sigma x_j]\tens \sigma x_1\tens\cdots\tens \widehat{\sigma x_i}\tens\cdots\tens \widehat{\sigma x_j}\tens\cdots \tens \sigma x_n.$$ There is no divided power part at $p=2$.  Since we are working with shifted, graded totally-isotropic Lie algebras, we use a modified version for ease of notation. First we note that given a $\tiLie$-algebra $L$, there are weak equivalences
\begin{equation}\label{shiftCE}
   N(\B(\mathrm{id}, \tiLie,L))\simeq N(\Sigma\B(\mathrm{id}, \mathrm{Lie}^{\mathrm{ti}}_{\F}, \Sigma^{-1} L))\simeq \Sigma \overline{CE}(\Sigma^{-1}L[1])[-1], 
\end{equation}
 where $\overline{CE}$ is the reduced complex. 

\begin{definition}\label{CEdef}
The \textit{Chevalley-Eilenberg complex} for a $\tiLie$-algebra $L$ is $\mathrm{CE}(L)=(\Lambda^\bullet(L),\delta)$, where $\Lambda^{\bullet}(L)$ is the free shifted graded exterior algebra on $L$ (placed in homological degree 0)  with a shifted graded exterior product $ \Sigma^{-1} \tens [1]$, which we continue to denote by $\tens$, that increases homological degree by one and decreases internal degree by one, reflecting the behavior of shifted graded Lie brackets. The differential 
$\delta$ is given by
$$\delta(x_1\tens\cdots \tens x_n)=\sum_{1\leq i<j\leq n}[x_i,x_j]\tens x_1\tens\cdots\tens \hat{x_i}\tens\cdots\tens \hat{x_j}\tens\cdots \tens x_n.$$
\end{definition}

Then the $\tiLie$-algebra homology of $L$ is given by the bigraded $\F$-module
\begin{align*}
  H^{\tiLie}_{*,*}(L):=\pi_{*,*}(\mathbb{L}Q^{\tiLie}_{\Mod_{\F}}(L)\oplus \F)\cong H_{*,*}(N(\B(\mathrm{id}, \tiLie,L)\oplus \F))\cong H_{*,*}(\mathrm{CE}(L)), 
\end{align*}
where the last isomorphism follows from rearranging the right hand side in  (\ref{shiftCE}).

\

In the case where $L$ is a simplicial $\tiLie$-algebra, its Chevalley-Eilenberg complex $\mathrm{CE}(L)$ is the simplicial chain complex obtained by applying the Chevalley-Eilenberg complex levelwise. Then Dold-Kan correspondence says that the homotopy groups of $\mathrm{CE}(L)$ are isomorphic to the homology groups of its total complex. A simplicial version of May's result is recorded in \cite[Section 3]{bhk}. Here we state the shifted version.

\begin{theorem}\cite[Theorem 3.13]{bhk}\label{CEbhk}
Let $L$ be a simplicial $\tiLie$-algebra. Then there is a natural isomorphism of bigraded $\F$-modules
$$H^{\tiLie}_{*,*}(L):=\pi_{*,*}(\mathbb{L}Q^{\tiLie}_{\Mod_{\F}}(L)\oplus \F)\cong H_{*,*}(\mathrm{CE}(L)).$$
\end{theorem}

In the total complex of $\mathrm{CE}(L)$, the differential in the homological direction is given by $\delta$ in \Cref{CEdef}. The differential $d$ in the simplicial direction is obtained by applying the shifted graded exterior algebra functor $\Lambda^{\bullet}$ to each simplicial differential $d_i$ of $L$ and taking the alternating sum, i.e. $$d=d_0\tens d_0\tens \cdots \tens d_0+\cdots +d_r\tens d_r\tens \cdots \tens d_r.$$ Both differentials preserve weights.

\

If the $\tiLie$-bracket on a simplicial $\tiLie$-algebra $L$ is trivial, then the differential $\delta$ in the homological direction vanishes and $H_{*,*}(\mathrm{CE}(L))\cong \pi_{*,*}(\Lambda^\bullet(L))$. The natural operations on the homotopy groups of  simplicial exterior algebras are well-understood by the work of Cartan, Bousfield, and Dwyer. We only state their results in the case of free algebras, and modify the grading to take into account the fact that we work with shifted, graded exterior algebras.
\begin{theorem}\cite[Theorem 2.1, Remark 4.4]{dwyer}\cite{bousfield}\cite{cartan}\cite[Theorem 3.9]{hm}\label{gamma}
Let $V_\bullet$ be a simplicial graded  $\F$-module. There are natural operations $$\gamma_i:\pi_{h,r,t}(\Lambda^h(V_\bullet))\rightarrow \pi_{2h+1,r+i,2t-1}(\Lambda^{2h+1}(V_\bullet)), 1\leq i \leq r$$ for all $r\geq 1$, satisfying the relations

$$\gamma_i\gamma_j(x)=\sum_{(i+1)/2\leq l\leq (i+j)/3}\binom{j-i+l-1}{j-l}\gamma_{i+j-l}\gamma_l(x) \mathrm{\ \ \ for\ all\  } i<2j.$$
\end{theorem}
Here in the trigrading $(h,r,t)$ records the number of exterior products $h$, the simplicial degree $r$ in $V_\bullet$, and the internal degree $t$.

Furthermore, they computed the homotopy groups of the free exterior algebra on a simplicial $\F$-module.
\begin{definition}
A sequence $I=(i_1,\ldots, i_m)$ is $\gamma$-\textit{admissible} if $i_l\geq 2 i_{l+1}$ for  $1\leq l\leq m-1$.  The \textit{excess} of $I$ is $e(I)=i_1-i_2-\cdots -i_m$.
\end{definition}

\begin{theorem}\cite[Theorem 8.6]{bousfield}\cite[Theorem 3.19]{hm}\label{exteriorhomotopy}
Let $A$ be a graded $\F$-basis for $\pi_*(V_\bullet)$. Then $\pi_{*,*}(\Lambda^{\bullet}(V_\bullet))$ is the (shifted graded) exterior algebra on generators
$\gamma_I(\alpha)$, where $\alpha\in A$ and $I=(i_1,\ldots, i_m)$ is $\gamma$-admissible with $e(I)\leq s(\alpha)$, where $s(\alpha)$ is the simplicial degree of the basis element $\alpha$.
\end{theorem}
The following is immediate by combining \Cref{CEbhk} and \Cref{exteriorhomotopy}.
\begin{corollary}\label{homotopywhentrivialbracket}
   Suppose that $L$ is a $\tiLie_{\R}$-algebra with trivial Lie brackets. Then $\pi_{*,*}\B(\mathrm{id}, \tiLie, \AR(L))$ is isomorphic as a bigraded $\F$-vector space to the (shifted graded) exterior algebra over $\F$ on generators $\gamma_I(\alpha)$, where $\alpha$ is a basis element of $\pi_{r,*}( \AR(L))$ (cf. \Cref{AR}) and $I$ is $\gamma$-admissible with $e(I)\leq r$.
\end{corollary}

Now we can compute the homotopy groups of $\B(\id,\tiLie_{\R},L)$ when the $\tiLie$ structure on $L$ is trivial.
First we recall the following result of Priddy that computes the Ext and Tor groups of a homogeneous Koszul algebra, which we make use of to compute the Tor groups over $\R$.
\begin{theorem}\cite[Theorem 2.5]{priddy}
Let $R$ be a homogeneous Koszul algebra over $\F$ on generators $a_i, i\in J$ in weight 1 and quadratic relations $r_j$. Let $B$ be a subset of the set $S$ of nonempty sequences on $J$ such that there is a basis of $R$ consisting of monomials $\{a_I\}_{I\in S}$. Then the cohomology algebra $H^{*,*}(A)=\mathrm{Ext}^{*,*}_R(\F,\F)$ is isomorphic to the tensor algebra on $a_i^\vee$ subject to relations that are linear dual to the $r_j$'s.
\end{theorem}
We record an explicit description of the procedure of cycle completion that produces a given class in the Tor groups, which will be useful later.
\begin{remark}\label{cyclecompletion}
Call $a_{i}a_j$ allowable if $(i,j)\in B$ and unallowable otherwise. Since we are working over $\F$ and the cohomology of $\R$ as a bigraded $\F$-module is finite in each  bidegree, we are allowed identify the bigraded $\F$-modules $\mathrm{Tor}^R_{m,n}(\F,\F)$ with the $\F$-linear dual of $\mathrm{Ext}^{m,n}_R(\F,\F)$. To simplify notation, we will name classes in $\mathrm{Tor}^R_{m,n}(\F,\F)$ by its the name of its dual in $\mathrm{Ext}^{m,n}_R(\F,\F)$.
 A cycle corresponding to the class $$a_{i_1}^\vee a_{i_2}^\vee\cdots a_{i_m}^\vee\in  \mathrm{Tor}^R_{m,*}(\F,\F)$$ with $(i_k, i_{k+1})$ unallowable for all $k$ in the reduced bar complex over $R$ is a sum $\sum_j [a_{j_1}|a_{j_2}|\cdots|a_{j_m}]\in  R^{\tens m}$ that contains the term $[a_{i_1}|a_{i_2}|\cdots|a_{i_m}]$ with nonzero coefficient. We call this the cycle completion of the monomial $[a_{i_1}|a_{i_2}|\cdots|a_{i_m}]$. To find the cycle explicitly, we start with $\alpha_0=[a_{i_1}|a_{i_2}|\cdots|a_{i_m}]$. The differential $\partial$ is a sum of face maps composing adjacent terms $a_{i_k}a_{i_{k+1}}$. We use the relation $a_{i_k}a_{i_{k+1}}=\sum b_{j_k}b_{j_{k+1}}$ to cancel out the terms $[a_{i_1}|\cdots|a_{i_k}a_{i_{k+1}}|\cdots|a_{i_m}]$ in the differential by adding $\sum[a_{i_1}|\cdots|a_{j_{k-1}}|b_{j_k}|b_{j_{k+1}}|a_{j_{k+2}}|\cdots|a_{i_m}]$ to $\alpha_0$ for all $k$ and denote the resulting sum $\alpha_1$. Then we pair off the differential for every term in $\alpha_1-\alpha_0$, i.e. for each nonzero term in $\partial (\alpha_1-\alpha_0)$ obtained by composing an unallowable 2-tuple via the $k$th face map, we use the relations in $R$ to find a sum in $R^{\tens m}$ whose image under the $k$th face map cancel out that term. Thus we obtain a new sum $\alpha_2$ such that all terms in the differential on $\alpha_1$ are paired off. Now we repeat the process again. It has to terminate since the number of unallowable adjacent pairs is nonincreasing for any term at each step and there are finitely many monomials with a given number of unallowable adjacent pairs. In other words,  $a_{i_1}a_{i_2}\cdots a_{i_m}$ can be written as a unique sum of basis monomials through this iterative process in finite steps. 
\end{remark}


\begin{lemma}\label{basisfortiLieAR}
(1).  Suppose that $L=\Sigma^{k}\F$ is a trivial $\tiLie_{\R}$-algebra.
    Then $\pi_{*,*}\B(\mathrm{id}, \tiLie, \AR(L))$ is the  exterior algebra over $\F$ on generators $\gamma_I\Q^J(x_k)$, where $x_k$ is the generator of $\pi_*(L)$, $J=(j_1,\ldots,j_r)$ satisfies $$j_{l+1}+\cdots+j_r+k-(r-l)\leq j_l\leq 2j_{l+1}$$ for $1\leq l<r$ and $j_r> k$,  and $I$ is $\gamma$-admissible with $e(I)\leq r$. In lower indexing, the generators are $\gamma_I\Q_J(x_k)$, where $J=(j_1,\ldots,j_r)$ satisfies $0\leq j_l\leq j_{l+1}+1$ for all $l$,  and $I$ is $\gamma$-admissible with $e(I)\leq r$.

(2). Let $L$ be the $\tiLie_{\R}$-algebra with  underlying $\R$-module $\Omega^n \free^{\Mod_{\R}}_{\Mod_{\F}}(\Sigma^{n+k}\F), n\geq 1$ and trivial Lie brackets. Then $\pi_{*,*}\B(\mathrm{id}, \tiLie, \AR(L))$ is the exterior algebra over $\F$ on generators $\gamma_I\Q_J(x_k)$, where $J=(j_1,\ldots,j_r)$ satisfies $0\leq j_l\leq j_{l+1}+1$ for all $l<r$ and $0\leq j_r<n$,  and $I$ is $\gamma$-admissible with $e(I)\leq r$.

\end{lemma}
\begin{proof}
(1). In light of \Cref{homotopywhentrivialbracket},
it suffices to compute $$\pi_{*,*}(\AR(L))=\pi_{*,*}\B(\id,\A_{\R},\Sigma^{k}\F),$$ where the right hand side is the unstable Tor groups $\mathrm{UnTor}^{\R}_{*,*}(\F,\Sigma^{k}\F)$. The unstable Tor groups are  computed by taking the homotopy groups of the subcomplex of the bar complex computing the Tor groups $\mathrm{Tor}^{\R}_{*,*}(\F,\Sigma^{k}\F)$ obtained by regarding $\Sigma^{k}\F$ as an unstable trivial module over $\R$ and imposing the instability conditions $[\Q^j|\alpha]=0$ for $j\leq |\alpha|$, cf. \cite[\S 3]{unstable}. 

The quadratic algebra $\R$ is a homogeneous Koszul algebra, since the canonical basis $\{\Q^{j_1}\cdots\Q^{j_r}, j_i>2j_{i+1} \forall i\}$  of  $\R$  is a  Poincar\'{e}-Birkhoff-Witt basis in the sense of Priddy \cite[Theorem 5.3]{priddy}.  In particular, it follows from Priddy's machinery \cite[Theorem 2.5, 3.8]{priddy} that the Tor group $\mathrm{Tor}^{\R}_{s,*}(\F,\F)$ has a basis consisting of cycles indexed by $\Q^{j_1}\cdots\Q^{j_s}$, where   $j_i\leq 2j_{i+1}$ for all $i$. 

To compute the unstable Tor groups on a class $x_k$ of internal degree $k$, we need to
impose the instability condition $\Q^j(x)=0$ for $j<|x|$,
then the basis of  $\mathrm{UnTor}^{\R}_{r,*}(\F,\F\{x_k\})$ consists of basis elements of $\mathrm{Tor}^{\R}_{r,*}(\F,\F)$ satisfying  $j_i> j_{i-1}-1+j_{i-2}-1+\cdots +j_r-1+|x|$ for all $i<r$ and $j_r\geq k$, or equivalently sequences $\Q_{j_1}\cdots\Q_{j_s}(x_k)$, where   $0\leq j_i\leq j_{i+1}+1$ for all $i$. 

(2). 
Iterating \Cref{stablization} yields a canonical map of $\R$-modules $$ L=\Omega^n \free^{\Mod_{\R}}_{\Mod_{\F}}(\Sigma^{n+k}\F)\rightarrow \Omega^\infty\free^{\Mod_{\R}}_{\Mod_{\F}}(\Sigma^\infty\Sigma^k\F)\cong\Sigma^k\F,$$ which gives rise to a surjective map of $\tiLie_{\R}$-algebras with trivial brackets. The underlying $\F$-module of  $L$ has basis $\Q^J x_{k}$, where $J=(j_1,\ldots,j_r)$ is a basis element of $\R$ satisfying $j_r\geq n+k$.   
Suppose that $\alpha\in\AR(L)$ is the cycle completion of an element $\Q^{j_1}|\cdots |\Q^{j_r}|x_k$ with $k\leq j_r<n+k$ and $j_{l+1}-1+\cdots +j_r-1+k\leq j_l\leq 2j_{l+1}$ for $l<r$.  Since cycle completion via Behrens' relations in the sense of \Cref{cyclecompletion} cannot increase the index of the right most operation, the differentials supported by $\alpha$ are the same as those supported by its image in $\AR(\F\{x_k\})$, so $\alpha$ is a nontrivial cycle. Otherwise, all but the rightmost face maps send $\alpha$ to zero, while the rightmost face map from at least one (distinct) term of $\alpha$ is nonzero, so it is impossible to complete the cycle. Switching to lower-indexing yields the desired answer.
\end{proof}

Combining \Cref{exteriorhomotopy} and \Cref{basisfortiLieAR}, we have the following:
\begin{corollary}\label{MayE1page}
For $\g=\Omega^n \free^{\Mod_{\R}}_{\Mod_{\F}}(\Sigma^{n+k}\F)$ with $1\leq n\leq \infty$, the $E^1$-page $\pi_{*,*}\B(\id,\tiLie_{\R},\g)$ of the algebraic $\gamma_1$-Bockstein spectral sequence (cf. \Cref{upperbound}) is the (shifted graded) exterior algebra on generators $\gamma_I\Q_J(x_k)$, where $I=(i_1,\ldots, i_m)$ is $\gamma$-admissible with $e(I)\leq r$ and $J=(j_1,\ldots,j_r)$ satisfies 
  $0\leq j_l\leq j_{l+1}+1$ for $l<r$, $0\leq j_r<n$.
\end{corollary}

\subsection{Quillen homology of  $\Lie_{\R}$-algebras with trivial brackets}
Next we want to identify the differentials in the May-type spectral sequence and the $\gamma_1$-Bockstein spectral sequence when $\g=\Omega^n \free^{\Mod_{\R}}_{\Mod_{\F}}(\Sigma^{n+k}\F)$. There is no canonical map from $\pi_{*,*}\B(\id,\Lie_{\R},\g)$ to the $E^1$-page $\pi_{*,*}\B(\id,\tiLie_{\R},\g)$ of the $\gamma_1$-Bockstein spectral sequence; instead we map  $\B(\id,\Lie_{\R},\g)$ into the bar construction of another variant of $\Lie_{\R}$-algebras.

\begin{definition}\label{DefinetiLieR>0}

Let $\Mod _{\rg}\subset\Mod_{\R}$ be the subcategory of allowable $\R$-modules $M$ such that $\Q_0(x)=0$ for all $x\in M$. Denote by $\free^{\Mod_{\rg}}_{\Mod_{\F}}$ the free $\rg$-module functor, and $\A_{\rg}$ the additive monad associated to the free functor.  Let $\tiLie_{\rg}= \A_{\rg}\circ \tiLie$, where the composite monad on the right  has distributivity given by $[\Q^j(-),(-)]=0$.
\end{definition}
 By \Cref{freeLR}, there is an equivalence of $\Lie_{\R}$-algebras $$\Lie_{\R}(M)=\A_{\R}\circ\Lie(M)/(\Q_0(x)=[x,x], x\in M).$$ In comparison, there is an equivalence of $\tiLie_{\rg}$-algebras $$\tiLie_{\rg}(M)= \A_{\rg}\circ \tiLie(M)=\A_{\R}\circ\Lie(M)/\langle\Q_0(x),[x,x], x\in M\rangle,$$ where the quotient is taken with respect to the left $\R$-algebra ideal. Hence the category $\tiLie_{\rg}$ of $\tiLie_{\rg}$-algebras is the subcategory of $\Lie_{\R}$-algebras $L$ satisfying the condition that $\Q_0(x)=[x,x]=0$ for all $x \in L$.
The inclusion $T^{\Lie_{\R}}_{\tiLie_{\rg}}(\g): \tiLie_{\rg}\rightarrow\Lie_{\R}$ of subcategory admits a left adjoint $Q^{\Lie_{\R}}_{\tiLie_{\rg}}(\g)$ that takes the quotient by the $\R$-algebra ideal of the self-brackets. 
 When $\g$ is a $\Lie_{\R}$-algebra with trivial $\Lie$ brackets, $Q^{\Lie_{\R}}_{\tiLie_{\rg}}(\g)$ is given by equipping the $\rg$-module $Q^{\Mod_{\R}}_{\Mod_{\rg}}(\g)$ with trivial $\tiLie$ brackets.
\begin{lemma}\label{comparison}
 Let $\g$ be an $\Lie_{\R}$-algebra. 
    There is a surjective map of simplicial $\F$-modules $$\varphi:\B(\id,\Lie_{\R},\g)\rightarrow\B(\id,\tiLie_{\rg},
    Q^{\Lie_{\R}}_{\tiLie_{\rg}}(\g)).$$
\end{lemma}
\begin{proof}
There is a map of monads $\Lie_{\R}\rightarrow\tiLie_{\rg}$ that sends the symbol $\Q_0$ to 0, and this induces the map of bar constructions in question.    
\end{proof}

The homotopy group of $\B(\id, \tiLie_{\rg},L)$ is computed in the same way as for $\B(\id,\tiLie_{\R},L)$ via  \Cref{factortiLierg}  and \Cref{basisfortiLieAR} ($\Q_0$ operation no longer appears in the generators). 
\begin{construction}\label{AR>0}
    For $L$ a $\tiLie_{\rg}$-algebra with $\tiLie$-bracket $\langle-,-\rangle$, denote by $\AR^{>0}(L)$ the bar construction $\B(\id,\A_{\rg},L)$ equipped with the simplicial $\tiLie$-algebra structure given levelwise by 
    $$\langle \alpha_1|\alpha_2|\ldots|\alpha_n|x, \beta_1|\beta_2|\ldots|\beta_n|y \rangle
= \left\{ \begin{array}{lcl}
 1|\cdots|1|\langle x,y\rangle & \mbox{if} &  \alpha_i=\beta_i=1, 1\leq i \leq n \\
0 &  &\mbox{otherwise} 
\end{array}\right..$$
\end{construction}
\begin{lemma}\label{homotopytiLieR}
(1).   There is an isomorphism
\begin{align*}
    \pi_{*,*}\B(\id, \tiLie_{\rg},\Sigma^k\F)&\cong\pi_{*,*}\B(\id, \tiLie,\B(\id, \A_{\rg},\Sigma^k\F))\\&\cong \pi_{*,*}\Lambda^\bullet(\mathrm{UnTor}^{\rg}_{*,*}(\F, \F\{x_k\})).
\end{align*}

Hence  $\pi_{*,*}\B(\id, \tiLie_{\rg},\Sigma^k\F)$ is the exterior algebra on generators $\gamma_I\Q_J(x_k)$, where $J=(j_1,\ldots,j_r)$ satisfies $1\leq j_l\leq j_l +1$ for all $l$ and $I$ is $\gamma$-admissible with $e(I)\leq r$.

(2). The homotopy group of $\B(\id,\tiLie_{\rg},\Omega^n\free^{\Mod_{\rg}}_{\Mod_{\F}}(\Sigma^{n+k}\F))$ is the exterior algebra on generators $\gamma_I\Q_J(x_k)$, where $J=(j_1,\ldots,j_r)$ satisfies $1\leq j_r<n$ and $1\leq j_l\leq j_l +1$ for $l<r$, while $I$ is $\gamma$-admissible with $e(I)\leq r$.

(3). For $L=\Omega^n\free^{\Mod_{\rg}}_{\Mod_{\F}}(\Sigma^{n+k}\F)$ with $1\leq n\leq\infty$, the quotient map of monads $\A_{\R}\rightarrow\A_{\rg}$ induces a surjective map $\pi_{*,*}\B(\id, \tiLie_{\R},L)\rightarrow\pi_{*,*}\B(\id, \tiLie_{\rg},L)$ that sends the symbol $\Q_0$ to 0.
\end{lemma}

In order to use the comparison map (cf. \Cref{comparison}) $$\varphi_*:\pi_{*,*}\B(\id,\Lie_{\R},\g)\rightarrow\pi_{*,*}\B(\id,\tiLie_{\rg},
    \g)$$ to detect differentials and permanent cycles, we make use of  explicit combinatorial formulae of   $\gamma_i$ by B\"{o}kstedt and Ottosen. The grading conventions are modified to suit our context.

For $r, i\in\mathbb{N}$ with $1 \leq i \leq r$, let $U(r,i)$  be the set of
pairs $(A, B)$ of ordered sequences $a_1 < \cdots < a_i, b_1 < \cdots < b_i$ such that 
$\{a_1,\ldots,a_i\}$ and $\{b_1,\ldots,b_i\}$ are complementary subsets of $\{r-i, r-i + 1,\ldots,r + i - 1\}$.
Let $V (r,i) \subset U(r,i)$ be the subset with $a_1 = r - i$.
\begin{proposition}\cite[Theorem 1.3, Lemma 3.1]{bokstedt} \label{bokstedt}
Suppose that $V_\bullet$ is a simplicial $\F$-module with face maps $d_j$.
Let $z$ be a representative of a class $[z]\in\pi_{s,t}(V_\bullet)$  in the normalized complex $N(V_\bullet)$. For $2\leq i \leq s$, define
$$\gamma_i(z) =\sum_{(A,B)\in V(s,i)}s_{a_i}\cdots s_{a_2}s_{a_1}(z)\tens s_{b_i}\cdots s_{b_2} s_{b_1}(z)\in \Lambda^2(V_\bullet).$$ Then  $d_j(\gamma_i(z))=0$ for $0\leq j\leq i+s$, and the induced operation $\gamma_i:\pi_{s,t}(V_\bullet)\rightarrow\pi_{s+i+1,2t-1}(\Lambda^2 (V_\bullet))$ are exactly the Dwyer-Bousfield operations in \Cref{exteriorhomotopy}. 
\end{proposition}

\begin{remark}\label{delta1}
    If in addition $V_\bullet$ is exterior, then the formula above for $i=1$ induces the operation $\gamma_1$ on $\pi_{*,*}(V_\bullet)$. The operation $\gamma_1$ is not well-defined when there is some element $a$ in the simplicial commutative algebra $V_\bullet$ such that $a\tens a \neq 0$. This is because in $N(V_\bullet)$ the differential sends $\gamma_1(a)$ to $a\tens a$, cf. \cite[Remark 4.3, 4.4]{dwyer}\cite[Remark 3.2]{bokstedt}.

Hence we obtain natural operations $\gamma_i$ for $1\leq i\leq s$ on $$\pi_{s,*}(\B(\id,\tiLie,\AR^{>0}(\Sigma^k\F)))\cong \pi_{s,*}(\Lambda^\bullet(\B(\id, \A_{\rg},\Sigma^k\F))),$$ and similarly on $$\pi_{s,*}(\B(\id,\tiLie,\AR(\Sigma^k\F)))\cong \pi_{s,*}(\Lambda^\bullet(\B(\id, \A_{\R},\Sigma^k\F))).$$
Suppose that $\xi$ is a cycle in $\mathrm{AR}^{>0}_s(\Sigma^k\F)$.
In the total complex of $\B(\id,\tiLie,\AR^{>0}(\Sigma^k\F))$, a representative for the homotopy class  $\gamma_i([\xi])$ is $$\gamma_i(\xi) =\sum_{(A,B)\in V(s,i)}\langle s_{a_i}\cdots s_{a_2}s_{a_1}(\xi), s_{b_1}s_{b_2}\cdots s_{b_i}(\xi)\rangle\in \tiLie\circ(\A_{\rg})^{\circ (s+i)}(\Sigma^k\F). $$ 
\end{remark}
When we iterate the $\gamma_i$ operations, the formula is harder to write down explicitly. 
\begin{notation}\label{bracketcyclecompletion}
Suppose that $V_\bullet$ is a simplicial $\F$-module as a trivial simplicial $\tiLie$-algebra. For distinct classes $[\xi_1],\ldots, [\xi_n]\in\pi_{*,*}(V_\bullet)$, denote by $B(\xi_1,\ldots,\xi_n)$ the cycle in the normalized complex of $\B(\id, \tiLie,V_\bullet)$ that represents the class  $[\xi_1]\tens\cdots\tens [\xi_n]\in \pi_{*,*}(\Lambda^{n-1}(V_\bullet))\subset \pi_{*,*}(\mathrm{CE}(V_\bullet))\cong \pi_{*,*}\B(\id, \tiLie,V_\bullet)$, which is obtained by cycle completion via the Jacobi identity in the sense of \Cref{cyclecompletion}.
\end{notation}
Therefore a homotopy class $[\xi_1]\tens\cdots\tens[\xi_l]$ with $l>1$ in $\pi_{s,*}(\Lambda^\bullet(\B(\id, \A_{\rg},\Sigma^k\F)))$
is represented by an element $B(\xi_1,\ldots, \xi_l)$ in the summand $(\tiLie)^{\circ (l-1)}\circ(\A_{\rg})^{\circ (s-l+1)}(\Sigma^k\F)$ of the total complex  of $\B(\id,\tiLie,\AR^{>0}(\Sigma^k\F))$. Since a representative for the homotopy class $\gamma_{j}\gamma_i(\xi)$ in the total complex of $\Lambda^\bullet(\B(\id, \free^{\Mod_{\rg}}_{\Mod_{\F}},\Sigma^k\F))$ is given by $$\gamma_j\gamma_i(\xi) =\sum_{(C,D)\in V(s+i+1,j)}\sum_{(A,B)\in V(s,i)}s_C\big(s_{A}(\xi)\tens s_{B}(\xi)\big)\tens s_D\big(s_{A}(\xi)\tens s_{B}(\xi)\big),$$
 a representative for $\gamma_{j}\gamma_i(\xi)$ in the total complex of $\B(\id,\tiLie,\AR^{>0}(\Sigma^k\F))$ is given the sum of  over all $(A,B)\in V(s,i), (C,D)\in V(s+i+1,j)$ of $B( s_Cs_{A}(\xi), s_Cs_{B}(\xi), s_Ds_{A}(\xi), s_Ds_{B}(\xi))$, with the three brackets coming from distinct simplicial filtrations.

\begin{theorem}\label{quillenhomologytrivial}
 The Quillen homology $$\AQ_{*,*}(\Omega^n\free^{\Mod_{\R}}_{\Mod_{\F}}(\Sigma^{n+k}\F))\cong\pi_{s,t}\B(\id,\Lie_{\R},\Omega^n\free^{\Mod_{\R}}_{\Mod_{\F}}(\Sigma^{n+k}\F))$$ of the  $\Lie_{\R}$-algebra $\Omega^n\free^{\Mod_{\R}}_{\Mod_{\F}}(\Sigma^{n+k}\F), 1\leq n\leq \infty$ is isomorphic as a bigraded vector space to the  exterior algebra on generators $\gamma_I\Q_J(x_k)$, where $I=(i_1,\ldots, i_m)$ is $\gamma$-admissible with $e(I)\leq r$ and $i_m\geq 2$, whereas $J=(j_1,\ldots,j_r)$ satisfies 
  $0\leq j_l\leq j_{l+1}+1$ for $l<r$, $0\leq j_r<n$ and if $j_1=0$ then either $r=1$ or $i_m=2$.
\end{theorem}
Recall from \Cref{stablization} that in the case $n=\infty$, $\Sigma^{n+k}\F\simeq\underset{n\rightarrow\infty}{\mathrm{lim}\ }\Omega^n\free^{\Mod_{\R}}_{\Mod_{\F}}(\Sigma^{n+k}\F)$ is the trivial $\Lie_{\R}$-algebra $\Sigma^k\F$.

Before we proceed to prove the theorem, we provide some intuition about the strategy. Since the input $\Lie_{\R}$-algebra $\Omega^n\free^{\Mod_{\R}}_{\Mod_{\F}}(\Sigma^{n+k}\F)$ has vanishing Lie brackets,  \Cref{rmk: modified may} allows us to consider a single May-type spectral sequence by considering the length filtration on $\Omega^n\free^{\Mod_{\R}}_{\Mod_{\F}}(\Sigma^{n+k}\F)$ shifted up by two.
From the construction of the May-type spectral sequence, we see that there is a higher differential on a class on the $E^1$-page $\cong\pi_{*,*}(\B(\id,\tiLie_{\R},L))$ if and only if its representative cycle, considered as an element in $\B(\id,\Lie_{\R},L)$, admits a face map that evaluates a non-self-bracket to a self-bracket. \Cref{delta1} and \Cref{MayE1page} indicate that  $\gamma_1$ is the only operation that arises in $\pi_{*,*}(\B(\id,\tiLie_{\R},L))\cong\Lambda\{\gamma_I\Q_J(x)\}$ with $I$ $\gamma$-admissible precisely because self-brackets are zero in $\tiLie_{\R}$-algebras and thus generates all the differentials in the May-type spectral sequence.  Hence we expect $\pi_{*,*}\B(\id,\Lie_{\R},L)$ to be a quotient of $\pi_*(\B(\id,\tiLie_{\R},L))$ (cf. \Cref{MayE1page}) by a suitable ideal generated by $\gamma_1(\alpha)$ for all $\alpha\in\pi_{*,*}(\AR(L))$, and we use the induced map on homotopy groups of $\varphi:\B(\id,\Lie_{\R},L)\rightarrow\B(\id,\tiLie_{\rg},L)$ from \Cref{comparison} to help detect the differentials and permanent cycles. 

\begin{proof}[Proof of \Cref{quillenhomologytrivial}]
We focus on the case $L=\Sigma^k\F$, since in the cases $n<\infty$ the only difference is an extra condition on the rightmost operation in basis elements, so the same argument applies with no change.

Consider the  map $\varphi_*:\pi_{*,*}\B(\id,\Lie_{\R},L)\rightarrow\pi_{*,*}\B(\id,\tiLie_{\rg},
    L)$ from \Cref{comparison}.
Its cokernel consists of all cycles in $\B(\mathrm{id}, \tiLie_{\rg}, \Sigma^k\F)$ whose preimage is the source of a differential to an element that is in the kernel of $\varphi$. Since $\varphi$ is surjective by \Cref{comparison}, this is equivalent to finding all classes $\alpha$ that are cycles in $\B(\mathrm{id}, \tiLie_{\rg}, \Sigma^k\F)$ precisely because the differential $\partial'$ in the normalized complex of $\B(\mathrm{id}, \tiLie_{\rg}, \Sigma^k\F)$  sends $\alpha$ to a linear combination of elements that contain self-brackets or $\Q_0$. In other words, via the inclusion to $\pi_{*,*}\B(\id,\tiLie_{\R},\Sigma^k\F)$ in \Cref{homotopytiLieR}.(3), all elements in the cokernel of $\varphi_*$ support differentials in the May spectral sequence.

We start with the generators of the exterior algebra, cf. \Cref{basisfortiLieAR}. Let $[\alpha]=\Q_{j_1}\Q_{j_2}\cdots\Q_{j_r}(x_k)$ be a basis element of $\pi_{*,*}\B(\mathrm{id}, \tiLie, \AR^{>0}(\Sigma^k\F))$, represented by a cycle $\alpha=\Q_{j_1}|\cdots|\Q_{j_r}|x_k+\sum_l \Q_{j'_1}|\cdots|\Q_{j'_r}|x_k$ in  $\B(\id,\tiLie_{\rg},
    \Sigma^k\F)$. The terms in the summation comes from cycle completion via Behrens' relations in the sense of \Cref{cyclecompletion}, with the condition that any term containing $\Q_0$ is 0. It has preimage $\tilde{\alpha}$ the cycle completion of $\Q_{j_1}|\cdots|\Q_{j_r}|x_k$ in $ \B(\id,\Lie_{\R},\Sigma^k\F)$ via Behrens' relations, which is the sum of $\alpha$ and terms $\Q_{j'_1}|\cdots|\Q_{j'_r}|x_k$ such that at least one of the $\Q_{j'_l}, l>1$ is equal to $\Q_0$. 
By \cite[Lemma 3.1]{bokstedt},  the differential $\partial$ in the normalized complex of $\B(\mathrm{id}, \tiLie_{\rg}, \Sigma^k\F)$ sends $\gamma_i(\alpha), i\geq 2$ to zero  because the terms are either zero or cancel out in pairs due to the simplicial identities of face and degeneracy maps. Hence its preimage $\gamma_i(\tilde{\alpha})$ is also a cycle in the normalized complex of  $ \B(\id,\Lie_{\R},\Sigma^k\F)$ and hence a permanent cycle in the May spectral sequence.  Similarly, for any $\gamma$-admissible sequence $I=(i_1,\ldots,i_m)$ with $i_m\geq 2$, $\gamma_I(\alpha)$ lifts to  a cycle $\gamma_I(\tilde{\alpha})$ in $ \B(\id,\Lie_{\R},\Sigma^k\F)$ and hence a permanent cycle in the May spectral sequence. By naturality of the $\gamma_i$ operations and \Cref{homotopytiLieR}.(3), the class $\gamma_{I}(\alpha)$ with $i_m\geq 2$ and $\alpha\in\pi_{*,*}(\AR(\Sigma^k\F))$ is also a permanent cycle. 

On the other hand, the differential $\partial$ sends $\gamma_1(\alpha)$ to $\langle \alpha, \alpha\rangle=0$ in $\B(\mathrm{id}, \tiLie_{\rg}, \Sigma^k\F)$, whereas its preimage $\gamma_1(\tilde{\alpha})=[s_0\tilde{\alpha}, s_1\tilde{\alpha}]$ maps to $[\tilde{\alpha},\tilde{\alpha}]=\Q_0|\tilde{\alpha}$ under the differential in $ \B(\id,\Lie_{\R},\Sigma^k\F)$, which is in the kernel of $\varphi$. In other words, there is a differential in the May-type spectral sequence  from $\gamma_1(\alpha)\in\pi_{*,*}\B(\id,\tiLie_{\R}, \Sigma^k\F)$ to $\Q_0\alpha$.  Similarly, for any $\gamma$-admissible sequence $I=(i_1,\ldots,i_m)$ with $i_m\geq 2$, $\gamma_I\gamma_1(\alpha)$ is a cycle in $\B(\mathrm{id}, \tiLie_{\rg}, \Sigma^k\F)$ because of the self-bracket in $\partial \gamma_I\gamma_1(\alpha)=\gamma_I(\partial(\gamma_1(\alpha)))$ if the simplicial degree of $\alpha$ is $r>1$ and $$\partial \gamma_I\gamma_1(\alpha)=\partial(\gamma_1(\alpha))\tens \gamma_1(\alpha)\tens \gamma_2\gamma_1(\alpha)\tens\cdots\tens \gamma_{2^{m-1}}\cdots\gamma_2\gamma_1(\alpha)$$ if $r=1$, cf. \cite[3.9.(i)]{hm}. On the other hand, its preimage $\gamma_I\gamma_1(\tilde\alpha)$ is mapped by the total differential in $\B(\id,\Lie_{\R},\Sigma^k\F)$ to the cycle completion $B(\Q_0|\tilde\alpha, \gamma_1(\tilde\alpha),\cdots, \gamma_{2^{m-1}}\cdots\gamma_2\gamma_1(\tilde\alpha))$ (cf. \Cref{bracketcyclecompletion}) if $r=1$, and to  $\gamma_I([\tilde{\alpha},\tilde{\alpha}])$ when $r>1$.  Note that in $\pi_{*,*}\B(\id,\tiLie_{\R},\Sigma^k\F)\cong \Lambda\{\gamma_I\Q_J(x_k)\}$ with $I$ $\gamma$-admissible and $J$ satisfying certain conditions, we have $[\gamma_I([\tilde{\alpha},\tilde{\alpha}])]=[\gamma_{I'}(\Q_0|\tilde{\alpha})]$ with $I'=(i_1+2^{m-1},\ldots,i_m+1)$. There is a shift in the indexing of the $\gamma$ operations because by construction the self-brackets appearing in the same bracket term live in distinct filtrations when more $\gamma$'s are applied, so replacing each self-bracket by a $\Q_0$ in a cycle will increase the index of the acting $\gamma_i$ by one. In particular, we note that $i_m+1\geq 3$. Hence there is a differential in the May-type spectral sequence  from $\gamma_I\gamma_1(\alpha)$ to $\gamma_{I'}(\Q_0|\alpha)$, and all the generators $\gamma_I\gamma_1(\alpha)$ of the exterior algebra $\pi_{*,*}\B(\mathrm{id}, \tiLie, \AR^{>0}(\Sigma^k\F))$ are in the cokernel of $\varphi_*$. Again by naturality of the $\gamma_i$ operations and \Cref{homotopytiLieR}.(3), the class $\gamma_{I}\gamma_1(\alpha)\in\pi_{*,*}\B(\id,\tiLie_{\R},\Sigma^k\F)$ supports a differential to $\gamma_{I'}(\Q_0\alpha)$ in the May-type spectral sequence.

In general, suppose  $[\alpha]$ is a basis element of  $\pi_{*,*}\B(\id,\tiLie_{\R},\Sigma^k\F)\cong\pi_{*,*}\B(\mathrm{id}, \tiLie, \AR(\Sigma^k\F))$ that is the exterior product of generators $\gamma_{I_1}([\alpha_1]),\ldots,\gamma_{I_n}([\alpha_n])$ with each $\alpha_i$ the cycle completion of a basis element $[\alpha_i]\in\pi_{*,*}\AR(\Sigma^k\F)$. It is represented by a cycle $\alpha=B(\gamma_{I_1}(\alpha_1),\ldots,\gamma_{I_n}(\alpha_n))$ in  the total complex of $\B(\mathrm{id}, \tiLie, \AR(\Sigma^k\F))$, cf. \Cref{bracketcyclecompletion}, since $d_j(\gamma_{I_l}(\alpha_l))=0$ for all $j$ and $l$ by \Cref{bokstedt}. Then $[\alpha]$ supports a differential in the May-type spectral sequence if and only if at least one of the $\gamma$-admissible sequences $I_l$ is of the form $I_l=(i_{l_1},\ldots, i_{l_m},1)$. By \Cref{MayE1page}, the above covers all classes in the $\F$-basis of the $E^1$-page of $\pi_{*,*}\B(\id,\tiLie_{\R},\Sigma^k\F)$.
\end{proof}

\begin{remark}\label{E2comonad}
  Note that $\pi_{*,*}\B\big(\mathrm{id},\Lie_{\R},\Sigma^k\F\big)$ is the cofree coalgebra on  $\Sigma^k\F$ over the comonad $$|\B\big(\id,\Lie_{\R},-)|:=\pi_{*,*}\B(\mathrm{id},\free_{\Lie_{\R}},-)$$ on $\Mod_{\F}$. The coalgebra structure map  is given by 
\begin{align*}
|\B\big(\id,\Lie_{\R},\Sigma^k\F\big)|&\xleftarrow{\simeq}|\B\big(\id,\Lie_{\R},|\B\big(\free^{\Lie_{\R}}_{\Mod_{\F}},\Lie_{\R},\Sigma^k\F\big)|\big)|\\&\rightarrow|\B\big(\id,\Lie_{\R},|\B\big(\id,\Lie_{\R},\Sigma^k\F\big)|\big)|,
\end{align*}
where the last map makes use of the augmentation $\free^{\Lie_{\R}}_{\Mod_{\F}}\rightarrow\id$, cf. \cite[Appendix D]{brantner}. In particular, $\pi_{*,*}\B\big(\mathrm{id},\Lie_{\R},\Sigma^k\F\big)$ records all natural unary operations on a degree $k$ class in the mod 2 Quillen homology of $\Lie_{\R}$-algebras, and \Cref{quillenhomologytrivial} gives us a dimension count.
\end{remark}

\section{Application to the Knudsen spectral sequence}\label{section4}
The rest of the paper is devoted to studying the mod $p$ homology of labeled configuration spaces using the computation of Quillen homology of spectral Lie algebras. The coefficients for homology is $\F$ unless otherwise specified.

\

Let $M$ be a manifold of dimension $n$ and $X$ a spectrum. The configuration space of $k$ points in $M$ labeled by $X$ is the spectrum $$B_k(M;X)=\Sigma^{\infty}_+\mathrm{Conf}_k(M)\tens_{\Sigma_k} X^{\tens k},$$ considered as a weighted spectra of weight $k$. Here  $\mathrm{Conf}_k(M)$ is the space of $k$-tuples of pairwise distinct points in $M$.
Denote by $s\lie$ the monad associated to the  free spectral Lie algebra functor $\free^{s\lie}$. The $\infty$-category of spectral Lie algebras is cotensored in Spaces, and we write $(-)^{M^+}$ for the cotensor with the one-point compactification of $M$ in this category. In \cite{ben}, Knudsen established the following equivalence using factorization homology, cf. \cite[Theorem 5.1]{bhk}.
\begin{theorem}\cite[Section 3.4]{ben}
Let $M$ be a parallelizable $n$-manifold and $X$ a spectrum. Consider $X$ as a weighted spectrum of weight one. Then there is an equivalence of weighted spectra
$$\bigoplus_{k\geq 1} B_k(M;X)  \simeq \mid \B(\mathrm{id}, s\lie, \free^{s\lie}(\Sigma^{n}X) ^{M^+}) \mid. $$ The left hand side is weighted by the index $k$; the weight filtration on the right hand side is given by propagating the weight on $X$ via the free spectral Lie operad functor.
\end{theorem}

 Applying the bar spectral sequence (\Cref{barsseq}) to the bar construction on the right hand side, we obtain the following:

\begin{proposition}
There is a weighted spectral sequence
\begin{equation}\label{knudsensseq}
    E^2_{s,t}=\AQ_{s,t}(H_*(\free^{s\lie}(\Sigma^{n}X) ^{M^+}))\Rightarrow \bigoplus _{k\geq 1}H_{s+t}(B_k(M;X)).
\end{equation}
\end{proposition}
The $\Lie_{\R}$-algebra structure on the $\F$-module $$H_*(\free^{s\lie}(\Sigma^{n}X) ^{M^+})\cong \widetilde{H}^*(M^+)\tens H_*(\free^{s\lie}(\Sigma^{n}X))\cong \widetilde{H}^*(M^+)\tens \free^{\Lie_{\R}}_{\Mod_{\F}}(H_*(\Sigma^n X))$$ has an explicit description.

\begin{proposition}\cite[Proposition 5.9]{bhk}\label{bracketbhk}
Let $\mathfrak{g}$ be a spectral Lie algebra. Then there is a spectral Lie algebra structure on the cotensor $\mathfrak{g} ^{M^+}$ in the category of spectra. The weight two structural map factors as
$$\partial_2(\mathrm{Id})\tens (\mathbb{D}(M^+)\tens \g)^{\tens2}_{h\Sigma_2}\rightarrow  \mathbb{D}(M^+)^{\tens2}_{h\Sigma_2}\tens(\partial_2(\mathrm{Id})\tens \g^{\tens2}_{h\Sigma_2})\xrightarrow{\mathbb{D}(\delta^*)\tens\xi_*}\mathbb{D}(M^+)\tens \g,$$ where $\mathbb{D}$ is the Spanier-Whitehead dual and $\delta$ the diagonal embedding.
\end{proposition}

As a result,  the shifted Lie bracket on
$\widetilde{H}^*(M^+)\tens H_*(\mathfrak{g})$ is given by $$[y_1\tens x_1,  y_2\tens x_2]:=(y_1\cup y_2)\tens[x_1,x_2].$$

On the other hand, the Steenrod operations on $H^*(M^+)$ induces a twisted $\R$-module structure in the cotensor.

\begin{proposition}\label{steenrodtwistQ}
The operations $\Q^j$ act on $\widetilde{H}^*(M^+)\tens H_*(\mathfrak{g})$ by
$$\Q^j(y \tens x)=\sum_{i}Sq^{i-j}(y)\tens \Q^i(x).$$
\end{proposition}
\begin{proof}
    Applying the Cartan formula $Q^j(y\tens x)=\sum_{i}Q^{j-i}(y)\tens Q^i(x)$ for the extended Dyer-Lashof operations $Q^j: x\mapsto e_{j-|x|}\tens x\tens x$ and the identification $Q^{-i}=Sq^i$ \cite
{maysq} to the definition of the $\Q^j$ operations, we have $$\Q^j(y\tens x)= \xi_*\sigma^{-1}(\sum_{i}Sq^{i-j}(y)\tens Q^i(x))=\sum_{i}Sq^{i-j}(y)\tens \xi_*\sigma^{-1}Q^i(x)=\sum_{i}Sq^{i-j}(y)\tens \Q^i(x)$$
Here $\sigma^{-1}$ is the desuspension isomorphism, and $\xi$ is the second structure map of spectral Lie algebras.
\end{proof}

\subsection{The universal case }
Now we apply \Cref{quillenhomologytrivial} to the case where $M$ is the Euclidean space. While the homology for $B_k(\mathbb{R}^n;X)$ is well-understood \cite{bmms}\cite{CLM}\cite{Einfinity}, we observe interesting patterns of higher differentials in the associated Knudsen spectral sequence. Furthermore,  the computation of the $E^2$-page in these cases will be useful in deducing the $E^2$-page for a general $M$.

Since $\widetilde{H}^*(S^n)=\F\{\iota_n\}$ is concentrated in one dimension, the only nonzero Steenrod operation is $Sq^0=\mathrm{id}$, so the $\R$-module structure on $\widetilde{H}^*(S^n)\tens H_*(\g)$ is given by $$\Q^j(\iota_n\tens x)=\sigma^{-n}\Q^j(x)=\Q^j(\sigma^{-n}x), x\in\g.$$

In the limiting case $M=\mathbb{R}^{\infty}=\underset{n\rightarrow\infty}{\mathrm{lim\ }}\mathbb{R}^n$, we have the stabilization
$$\lim_{n\rightarrow\infty} \Omega^n\free^{s\lie}(\Sigma^{n}X) \simeq X,$$ and the spectral sequence (\ref{knudsensseq}) becomes 
\begin{equation}\label{stablesseq}
    E^2_{s,t}=\AQ_{s,t}(\Sigma^k\F)\Rightarrow H_{s+t}(\free^{\mathbb{E}_\infty}(\mathbb{S}^k)).
\end{equation}
The $E^2$-page is computed in \Cref{quillenhomologytrivial}. Namely, it is the exterior algebra generated by one class $x_k$ and two types of operations on coalgebras over the comonad $\pi_{*,*}\B(\id,\Lie_{\R},-)$ $$\Q^j:E^2_{h,s,t}\rightarrow E^2_{h,s+1, t+j-1},\ \ j\geq t$$  $$\gamma_i:E^2_{h,s,t}\rightarrow E^2_{2h+1,s+i, 2t-1},\ \ 2\leq i\leq s$$  under a further splitting of the filtration degree into a sum of  homological degree $h$ counting the number of brackets and  simplicial degree $s$ counting the number of $\Q^j$'s.

Comparing with the computation of $H_*(\free^{\mathbb{E}_\infty} (\mathbb{S}^k))$ \cite{Einfinity}\cite{bmms}, which is the $E^\infty$-page, we can immediately deduce that the $E^2$-page is much larger. Using sparsity arguments, we can identify higher differentials in low degrees, which allows us to make the following conjecture. 

\begin{conjecture} \label{conjecture}
Each page of the spectral sequence
$$E^2_{s,t}=\AQ_{s,t}(\Sigma^k\F)\Rightarrow \pi_{s+t}\B(\id,s\lie,\Sigma^k\F)\cong H_{s+t}(\free^{\mathbb{E}_\infty}(\mathbb{S}^k))$$ is an exterior algebra. The higher differentials  on the exterior generators of the $E^2$-page are given as follows:
\begin{enumerate}
    \item For an exterior generator $\alpha=\Q_{j_1}\cdots \Q_{j_m}(x_k)$ on the $E^2$-page,  we have
 $$d^{r+1}\gamma_{r+1}(\alpha)=\Q_{r}(\alpha)$$ for  $r<m$ and $r\leq j_1+1$.
 \item For an exterior generator $\beta=\gamma_{n+1}\Q_{j_1}\cdots \Q_{j_{m-1}}(x_k)$  on the $E^2$-page, we have
 \begin{enumerate}
     \item $d^{n+1}(\beta)=\Q_{n}\Q_{j_1}\cdots \Q_{j_{m-1}}(x_k)$,
     \item $d^{n+1}\gamma_{m+n}(\beta)=d^{n+1}(\beta)\tens \beta$,
     \item 
 $\gamma_{l} d^{n+1}(\beta)=d^{2n+1}\gamma_{n+l-1}(\beta)$ for $n+3\leq l\leq m$.
 \end{enumerate}
\end{enumerate}
  These generate all higher differentials under further applications of $\gamma_i$ operations in accordance with (2).(b) and (2).(c), as well as the exterior product.
\end{conjecture}
The figure below is an illustration of the higher differentials in homological Adams grading $(s+t,s)$ for $\beta=\gamma_{n+1}\Q_{j_1}\cdots \Q_{j_m}(x_k)$ and $\alpha=\Q_{n}\Q_{j_1}\cdots \Q_{j_{m-1}}(x_k)$ with internal degree $b$. Set $a=2b+m+1$. Along the horizontal line $s=m+1$ we have generators $\Q_1(\alpha),\ldots, \Q_{n+1}(\alpha)$, each receiving a blue differential via \Cref{conjecture}.(1). Along the top slope we have, for each $l$ with $n+2< l\leq m$, a cyan arrow $d_{2n+1}(\gamma_{n+l-1}(\beta))=\gamma_{l}(\alpha)$, which correspond to the differentials in \Cref{conjecture}.(2).(c). Finally we have a gray arrow hitting the cross term, corresponding to \Cref{conjecture}.(2).(b).
\begin{center}
   \begin{tikzcd}[column sep=0.0cm, row sep=-0.1cm,font=\scriptsize, every arrow/.append style=-latex]\label{pattern}
   2m+2n+3&&&&&&&&&&&&&&&&&\bullet\ar[lddddddd,gray]\\
    2m+2n+2&&&&&&&&&&&&&&&&\bullet\ar[ldddddddd,cyan]\\
   2m+2n+1&&&&&&&&&&&&&&&\bullet\ar[ldddddddd,cyan]\\
   \vdots\\
   m+3n+5&&&&&&&&&&&&\bullet\ar[lddddddddd,cyan]\\
   m+3n+4&&&&&&&&&&&\bullet\ar[lddddddddd,cyan]\\
   \vdots&&&&&&&&&&&&&\\
   2m+n+2&&&&&&&&&&&&&\ldots&&&\bullet\\
   \vdots&&&&&\\
   2m+1&&&&&&&&&&&&&&&\bullet\\
   2m&&&&&&&&&&&&&&\bullet\\
   \vdots&&&&\\
   m+n+5&&&\\
   m+n+4&&&&&&&&&&&\bullet\\
 m+n+3& &&&&&&&&&\bullet\\ 
m+n+2&&&&&&&&&\bullet\ar[ddddddl,color={rgb,256:red,0;green,72;blue,152}]\\
  m+n+2&&&&&&&&\bullet\ar[dddddl,color={rgb,256:red,0;green,72;blue,152}]&&\\
  \vdots&\\
  m+4&&&&&&\bullet \ar[dddl,color={rgb,256:red,0;green,72;blue,152}]&\cdots&\\
  m+3&&&&&\bullet\ar[ddl,color={rgb,256:red,0;green,72;blue,152}]\\
  m+2&\\
m+1&&&&\bullet&\bullet&\cdots&\bullet&\bullet&\\[-5pt]
   & &&&a+1&&\cdots&&a+n+1&&&&&
\end{tikzcd}.
\end{center}

\begin{remark}\label{unviersaldiff}
  The pattern in the universal case is similar to the pattern of universal higher differentials  in \cite[Proposition 2.6]{dwyerdiff} and \cite{turner}, where divided squares kills off Steenrod operations that are not admissible. Here, the Dyer-Lashof operations $\Q^j$ on the $E^\infty$-page should be represented by the surviving $\Q^j$ operations. On the $E^2$-page, the admissibility condition for $\Q^j$ allows for more admissible sequences than the Dyer-Lashof algebra. The $\gamma_i$ operations eliminate the $\Q^j$ operations that do not satisfy the admissibility condition for Dyer-Lashof operations via higher differentials. 
  
  One major difference is that while Steenrod operations can be defined on the spectral sequence filtration-wise in \cite{dwyerdiff} and \cite{turner}, the operations $\Q^j$ increase filtration by one. Hence the classical methods of producing operations on spectral sequences by chain-level constructions no longer apply.
  
\end{remark}
 
 In forthcoming work with Robert Burklund and Andrew Senger, we use a suitable deformation of the comonad associated to the bar construction $|\B(\id,s\lie,-)|$ to the $\infty$-category of Postnikov-connective filtered $\F$-modules, which allows us to detect the pattern of higher differentials in \Cref{conjecture}.
\begin{remark}
The spectral sequence we study here is analogous to the bar spectral sequence 
$$E^2_{s,t}=\pi_s\pi_t\B(\id, \mathbb{E}^{\mathrm{nu}}_{\infty}\tens\Fp,\pi_*(A))\Rightarrow\pi_{s+t}\B(\id,\mathbb{E}^{\mathrm{nu}}_{\infty}\tens\Fp, A)$$ and its dual. The latter was used to identify operations on homotopy groups of spectral partition Lie algebras and
 mod $p$ TAQ cohomology operations of nonunital $\mathbb{E}_{\infty}$-$\Fp$-algebras in \cite{zhang}, which subsumes unpublished work of Kriz, Basterra and Mandell. The $E^2$-page of this spectral sequence is the Andr\'{e}-Quillen homology of $\mathrm{Poly}_R$-algebras, i.e., graded $\F$-modules equipped with Dyer-Lashof operations and a polynomial product that satisfying the Cartan formula. In contrast to \Cref{conjecture}, this spectral sequence collapses on the $E^2$-page. Heuristically, the phenomenon here arises from the nonadditivity of the free $\Lie$-algebra functor and the order of the factorization $Q^{\Lie_{\R}}_{\Mod_{\F}}=Q^{\Lie}_{\Mod_{\F}}\circ Q^{\Lie_{\R}}_{\Lie}$, which results in simplicial homotopy operations. Whereas the Dyer-Lashof operations are additive away from the bottom operations on even degree classes, so the factorization $Q^{\mathrm{Poly}_R}_{\Mod_{\Fp}}=Q^{\Mod_{R_{>0}}}_{\Mod_{\Fp}}\circ Q^{\mathrm{Poly}_R}_{\Mod_{R_{>0}}}$ does not introduce simplicial homotopy operations.
\end{remark}

\subsection{With coefficients}
Next, we take up a slightly more complicated case, where $M=\mathbb{R}^n$ with labels in an arbitrary spectrum $X$. Then $H_*(\free^{s\lie}(\Sigma^{n}X)^{M^+})\cong\Omega^n \free^{\Lie_{\R}}_{\Mod_{\F}}(\Sigma^n H_*(X))$ and the spectral sequence (\ref{knudsensseq}) becomes
\begin{equation}\label{coeffsseq}
    E^2_{s,t}=\AQ_{s,t}(\Omega^n \free^{\Lie_{\R}}_{\Mod_{\F}}(\Sigma^n H_*(X)))\Rightarrow H_{s+t}(\free^{\mathbb{E}_n} (X)).
\end{equation}
When $X=\mathbb{S}^k$, the $E^2$-page $\AQ_{s,t}(\Omega^n\free^{\Lie_{\R}}_{\Mod_{\F}}(\Sigma^{n+k}\F))$ is computed in \Cref{quillenhomologytrivial}.

Write $H_*(X)\cong \bigoplus_{k,l}\F\{x_{k,l}\}$, where $\{x_{k,l}\}_l$ is an $\F$-basis of $H_k(X)$ for each $k$. Then 
\begin{align*}
  \g=H_*(\Omega^n\free^{s\lie}(\Sigma^{n}H_*(X)))&\cong\F\{\iota_n\} \tens H_*(\free^{s\lie}(\Sigma^n H_*(X)))\\
  &\cong\F\{\iota_n\} \tens \Big(\bigoplus_{w\in W}\F\{\Q^{J} w, J\in \R(d(w))\}\Big)
\end{align*}
by \cite[Proposition 7.3]{omar}.
Here $\R(n)$ is the quotient of $\R$ by the relations $\Q^{j_1}\cdots \Q^{j_k}=0$ if $j_1<j_2+\cdots +j_k+n$, and $W$ is the set of Lyndon words on the set of letters $\{\sigma^n x_{k,l}\}_{k,l}$, which is a basis for the free $\tiLie$-algebra on generators $\{\sigma^n x_{k,l}\}_{k,l}$.

We define the \textit{degree} of a word $w\in W$ to be $d(w)=1+ \sum_{k,l}m_{k,l}(w)(n+k-1),$ where $m_{k,l}(w)$ counts the number of times the letter $\sigma^n x_{k,l}$ appears in $w$.  
Set $$\g_w=\F\{\iota_n\} \tens \F\{\Q^{J}w, J\in \R(n+|w|)\}.$$
Then $\g\simeq \bigoplus_{w\in W} \g_w$ with trivial brackets. 
 Note that this splitting is induced by an equivalence of $s\lie$-algebras in $\F$-module spectra
\begin{align*}
    \Big(\free^{s\mathcal{L}}(\Sigma^n X)\Big)^{(\mathbb{R}^n)^+}\tens \F
    \simeq& \mathbb{D}(S^n)\tens\free^{s\mathcal{L}}(\Sigma^n X\tens \F)\\
    \simeq& \mathbb{D}(S^n)\tens \free^{s\mathcal{L}}(\bigvee_{x_{k,l}} \Sigma^{n+k}\F)\\
    \simeq& \bigvee_{w\in W}\Big(\free^{s\mathcal{L}}(\Sigma^{d(w)}\F)\Big)^{(\mathbb{R}^n)^+},
\end{align*}
where the last step makes use of Corollary 5.13 in \cite{ab}. The equivalence above would  only be that of $\F$-module spectra if we did not kill the brackets by cotensoring with $(\mathbb{R}^n)^+$. Therefore we deduce the following:
\begin{proposition}\label{E2pageforRn}
    The spectral sequence $ E^2_{s,t}=\AQ_{s,t}(\Omega^n \free^{\Lie_{\R}}_{\Mod_{\F}}(\Sigma^n H_*(X)))\Rightarrow H_{s+t}(\free^{\mathbb{E}_n} (X))$ splits as 
\begin{align*}
    E^2_{s,t}\cong\bigoplus_{w\in W} \AQ_{s,t}(\g_w)\Rightarrow\bigoplus_{w\in W}\pi_{s+t}\B(\mathrm{id}, s\mathcal{L},\Omega^n\free^{s\mathcal{L}}(\Sigma^n\Sigma^{d(w)-n}\F)).
\end{align*}
\end{proposition}

\begin{remark}
    The canonical  map of spectral Lie algebras $$\Omega^n\free^{s\mathcal{L}}(\Sigma^n \mathbb{S}^k)\rightarrow \Omega^\infty\free^{s\mathcal{L}}(\Sigma^\infty \mathbb{S}^k)$$ via stabilization induces an embedding of the $E^2$-pages $$\AQ(\Omega^n\free^{\Mod_{\R}}_{\Mod_{\F}}(\Sigma^{n+k}\F))\rightarrow\AQ(\Sigma^k\F)$$ by \Cref{stablization} and \Cref{quillenhomologytrivial}. We expect that the higher differentials in the target (\Cref{conjecture}) pull back to higher differentials in the source. Indeed, combinatorially this will yield the computation of the free $\mathbb E_n$-algebra on a single generator. If $H_*(X)$ has multiple generators, then the splitting of the spectral sequence above via Lyndon words corresponds precisely the Browder bracket on the free $\mathbb E_n$-algebra on those generators, cf. \cite[III]{CLM}. 
\end{remark}
\section{Upper bounds and low weight computations}\label{section5}

For a general parallelizable manifold $M$ of dimension $n$, the $\Lie_{\R}$-algebra $$\g=\widetilde{H}^*(M^+)\tens\free^{\Lie_{\R}}_{\Mod_{\F}}(\Sigma^n H_*(X))$$ has non trivial $\Lie$-brackets and the precise image of the comparison map $\varphi_*$ in \Cref{comparison} becomes much harder to pin down. Nonetheless, \Cref{upperbound} and  \Cref{cor:factortiLierg} allow us to obtain a formula for an upper bound of $\pi_{*,*}\B(\id,\Lie_{\R},\g)$ by
$$\pi_{*,*}\B(\id,\tiLie_{\R},\tilde{\g})\cong \pi_{*,*}(\mathrm{CE}(\AR(\tilde{\g})))$$ that is an equivalence in weight less than four. Here
$\tilde\g=\widetilde{H}^*(M^+)\tens \free^{\tiLie_{\R}}(H_*(X))$ is the associated $\tiLie_{\R}$-algebra, where $\widetilde{H}^*(M^+)$ is equipped with the $\tiLie$-bracket coming from the associated $\tiLie$-algebra of the $\Lie$-algebra $H^*(M^+)$ with its usual cup product, cf. \Cref{tibracket}. 
In particular, it follows from \Cref{smallweightisom} that in weight less than four, the two homotopy groups are isomorphic.

\subsection{General upper bounds}
We will see that  $\pi_{*,*}(\mathrm{CE}(\AR(\tilde\g)))$ admits a description in terms of the $\tiLie$-algebra homology of $\tilde\g$. The key observation is that for $\tilde\g=\widetilde{H}^*(M^+)\tens \free^{\tiLie_{\R}}(H_*(X))$, $\AR(\tilde\g)$ has trivial $\tiLie$-structure away from simplicial degree 0 and its degeneracies, cf. \Cref{AR}, and the $\tiLie$-bracket on $\tilde\g$ vanishes on elements that involve $\Q^i$ operations. 

\begin{definition}
For a $\tiLie$-algebra $\g$, we say that its $\tiLie$-structure is \textit{supported entirely} by a  sub-$\tiLie$-algebra $\g'$ if 
the $\tiLie$-algebra $\g$ is isomorphic to the product $\tiLie$-algebra $N\oplus \g'$, where the $\tiLie$ bracket vanishes on the complement $N\subset \g$.
\end{definition}

\begin{lemma}\label{modifyCE}
Let $\tilde\g=L\tens \free^{\tiLie_{\R}}_{\Mod_{\F}}(V)$ be a $\tiLie_{\R}$-algebra, where $L$ is a non-unital graded commutative algebra over $\F$ and the $\tiLie$-structure on $\tilde\g$ is the usual one on the tensor product. Then $$\pi_{*,*}(\mathrm{CE}(\AR(\tilde\g)))\cong\Lambda \{\gamma_I(\alpha),\alpha\in A\}\tens H^{\tiLie}_{*,*}(\tilde\g),$$ where $\alpha\in A$ is an element of an $\F$-basis for $\pi_{\geq 1,*}(\AR(\tilde\g))$ with simplicial degree $s(\alpha)$, and $I$ is $\gamma$-admissible with $e(I)\leq s(\alpha)$.
\end{lemma}

\begin{proof}
Since brackets of operations are zero, the $\tiLie$-algebra $\tilde\g$ is supported entirely by the sub-$\tiLie$-algebra $\g_0'=L\tens \free^{\tiLie}_{\Mod_{\F}}(V)$. Furthermore, for all $m\geq 1$, the $\tiLie$-algebra $\mathrm{AR}_m(\tilde\g)$ is supported entirely by the degeneracies coming from $\g'_0$ by \Cref{AR}. Hence each simplicial level  
$\mathrm{AR}_m(\tilde\g)$ is isomorphic to the product $\tiLie$-algebra $T_m\oplus \g'_m$, where $\g'_m$ is the sub-$\tiLie_{\F}$-algebra consisting of degeneracies of $\g'_0$ and $T_m$ a trivial $\tiLie$-algebra. Since the splittings respect the simplicial $\tiLie$-algebra structure of $\AR(\tilde\g)$, we deduce that  $\AR(\tilde\g)\cong T_\bullet\oplus \g'_\bullet$ as simplicial $\tiLie_{\F}$-algebras. This induces a splitting of chain complexes $$\mathrm{CE}(\AR(\tilde\g))\cong \mathrm{CE}(T_\bullet)\tens \mathrm{CE}(\g'_\bullet),$$ where $T_\bullet$ is a trivial simplicial $\tiLie$-algebra and $\g'_\bullet$ the constant simplicial object on $\g_0'$. The lemma then follows from Theorem \ref{exteriorhomotopy}, noting that $H^{\tiLie}_{*,*}(\tilde\g)\cong H^{\tiLie}_{*,*}(T_0)\tens H^{\tiLie}_{*,*}(\g'_0)$.\end{proof}

It remains to compute $\pi_{*,*}(\AR(\tilde\g))$  for $\tilde\g=\widetilde{H}^*(M^+)\tens \free^{\tiLie_{\R}}(H_*(X))$. Since $\g$ and $\tilde\g$ are isomorphic as $\R$-modules (cf. \Cref{levelwiseisom}), we will not distinguish the two. Recall from \Cref{steenrodtwistQ} that the $\R$-module structure on $\g$ is twisted by the Steenrod operations in the sense that
$$\Q^j(y\tens \alpha)=\sum_{0\leq s\leq n} Sq^{j+s}(y)\tens \Q^s(\alpha).$$ 

\begin{notation}\label{notationbasisforV}
 Let $H\cup\{z\}$ be an $\F$-basis of the cohomology ring $H^*(M^+)$, where $z$ corresponds to the added point in the one-point compactification and $H$ is a basis for $\widetilde{H}^*(M^+)$. For $y\in H$, denote by $|y|$ the cohomological degree of $y$.
 
Let $\widetilde{B}=\{x_a\}_a$ be a totally ordered basis for $V=H_*(X)$ and $B=\{\sigma^n x_a\}_a$ with the induced ordering, where $n$ is the dimension of $M$. Denote by $W$ the set of basic products on the set $B$. Then 
\begin{align*}
  \g&=\widetilde{H}^*(M^+)\tens H_*(\free^{s\lie}(\Sigma^{n}X))\cong\bigoplus_{w\in W, y\in H}\F\{y\} \tens \F\{\Q^{J} w, J\in \R(|w|)\}.
\end{align*} 
\end{notation}
\begin{proposition}\label{filtersteenrod}
    The bigraded homotopy group $\pi_{*,*}(\AR(\tilde\g))=\pi_{*,*}(\AR(\g))$ is isomorphic to $\pi_{*,*}(\AR(\g_{\mathrm{triv}}))$, where the untwisted $\R$-module $\g_{\mathrm{triv}}$ has the same underlying $\F$-module as $\g$ and the $\R$-module structure is given by $\Q^j(y\tens x)=y \tens \Q^j(x)$ for all $j$. 
\end{proposition}

\begin{proof}
We make use of a spectral sequence to filter away the twisting by the action of the Steenrod operations. We abuse notation here and denote again by $\AR(\g)$ the associated chain complex of $\AR(\g)$. Filter $\g$ in terms of  decreasing cohomological degree of $\widetilde{H}^{*}(M^+)$, so we have $$F_{-p}(\g)=\widetilde{H}^{\geq {p}}(M^+)\tens \free^{\Lie_{\R}}_{\Mod_{\F}}(V)\cong\bigoplus_{w\in W, y\in H, |y|\geq p}\F\{y\tens \Q^{J}(w), J\in \R(|w|)\}$$ with associated graded pieces given by $$G_{-p}(\g)=F_{-p}(\g)/F_{-p-1}(\g)\cong \bigoplus_{w\in W, y\in H, |y|= p}\F\{y\tens \Q^{J} (w), J\in \R(|w|)\}.$$ Since action by Steenrod operations does not decrease cohomological degree, the induced filtration $$F_{-p}(\AR(\g)):=\AR(F_{-p}(\g))$$ makes $\AR(\g)$ a filtered chain complex. The associated graded pieces are
$$G_{-p}(\AR(\g))=\AR(G_{-p}(\g))=\bigoplus_{w\in W, y\in H, |y|= p}\AR(\F\{y\tens \Q^{J} (w), J\in \R(|w|)\})$$  and the induced differential preserves direct summands.

Using the case $M=\mathbb{R}^n$ in \Cref{E2pageforRn}, we deduce that 
\begin{align*}
    E^1_{-p,q}=H_{-p+q}(G_p(\AR(\g)))&\cong\bigoplus_{w\in W, y\in H, |y|=p}\pi_*\Big(\AR\big(\F\{y\tens\Q^{J}( w), J\in \R(|w|)\}\big)\Big)\\
    &\cong \bigoplus_{w\in W, y\in H, |y|=p}\F\{\Q^{j_1}\cdots\Q^{j_m}(y\tens w), (j_1,\ldots, j_m)\in \R(p, |w|)\},
\end{align*}
where $\R(p, |w|)$ is the set of sequences $(j_1,\ldots, j_m)$ such that
\begin{enumerate}
    \item $j_l\leq 2 j_{l+1}$ for $1\leq l<m$ and $|w|-p\leq j_m<|w|$;
    \item If $m\geq 2$ then $j_l\geq j_{l+1}+\cdots +j_m+|w|-(m-l)$ for $2\leq l\leq m-1$ and $j_1> j_2+\cdots +j_m+|w|-(m-1)$.

\end{enumerate}
We claim that every class on the $E^1$-page survives to a class on the $E^\infty$-page by induction on $\widetilde{H}^*(M^+)$ along decreasing cohomological degree.

For $y\in \widetilde{H}^n(M^+)\in F_{-n}(\g)$ a top cohomology class, there is no nonzero action by a Steenrod operation on $y$ other than $Sq^0$, so the differential on $\beta$ in $\AR(\g)$ is the same as the differential in $G_{-n}(\AR(\g))$, i.e. $\beta$ survives to a nonzero class on the $E^\infty$-page. 

Suppose that in $F_{-p-1}(\AR(\g))=\AR(F_{-p-1}(\g))$, any basis element $\beta'=\Q^{j'_1}\cdots\Q^{j'_m}(y'\tens w')$ of the $E^1$-page is a permanent cycle and they span all permanent cycles in $F_{-p-1}(\AR(\g))$. Let $[\beta]=\Q^{j_1}\cdots\Q^{j_m}(y\tens w)$ be a basis element on the $E^1$-page, with $y\in\widetilde{H}^p(M^+)$.  A cycle representing this class in $\AR(G_{-p}(\g))$ is a finite sum $$\beta=\Q^{j_1}|\cdots|\Q^{j_m}|(y\tens w)+\sum_l \Q^{l_1}|\cdots|\Q^{l_m}|(y\tens w)$$  obtained by cycle completion via Behrens' relations in the sense of \Cref{cyclecompletion}. Note that $l_m\leq j_m<|w|$ for all $l$. Let $d_m$ be the rightmost face map. Then in $\AR(\g)$ 
\begin{align*}
   \partial \beta&=\partial\Big (\Q^{j_1}|\cdots|\Q^{j_m}|(y\tens w)+\sum_l \Q^{l_1}|\cdots|\Q^{l_m}|(y\tens w)\Big)\\&=0+d_m\Big(\Q^{j_1}|\cdots|\Q^{j_m}|(y\tens w)+\sum_l \Q^{l_1}|\cdots|\Q^{l_m}|(y\tens w)\Big)\\&=\sum_{s\geq 0}\Q^{j_1}|\cdots|\Q^{j_{m-1}}|Sq^s(y)\tens \Q^{j_m+s} (w)+\sum_l\sum_{s\geq 0}\Q^{l_1}|\cdots|Q^{l_{m-1}}|Sq^s(y)\tens \Q^{l_m+s} (w). 
\end{align*}

Note that the sum of these $\theta_l=\Q^{l_1}|\cdots|Q^{l_{m-1}}|Sq^s(y)\tens \Q^{l_m+s} (w)$ or $Q^{j_1}|\cdots|\Q^{j_{m-1}}|Sq^s(y)\tens \Q^{j_m+s} (w)$ over $s\geq0$ is a boundary in $\AR(\g)$: 
If $l_m+s<|w|$ then $\theta_l=0$.
If $l_m+s\geq |w|$ or $j_m+s\geq |w|$, then $s\geq 1$, since $l_m\leq j_m<|w|$, so $\theta_l\in F_{-p-1}(\AR(\g))$. By the inductive hypothesis, the sum of such $\theta_l$ is not a nonzero cycle on the $E^\infty$-page and thus the boundary of a finite sum of classes in $F_{-p-1}(\AR(\g))$ of the form $\Q^{j'_1}|\cdots|\Q^{j'_m}|(y'\tens w')$ with $|y'|\geq p+s>p$. Denote by $\xi$  this finite sum, so $ \partial(\beta+\xi)=0$ in $\AR(\g)$. Note that $\xi$ is not a boundary because it is maximally nondegenerate and $\xi\neq \beta$ since $\beta$ is not in $F_{-p-1}(\AR(\g))$.
Hence $\beta+ \xi$ is a cycle in $\AR(\g)$ corresponding to the basis element $\beta=\Q^{j_1}\cdots\Q^{j_m}(y\tens w)$ on the $E^1$-page. Therefore the dimension of the $E^1$-page is at most that of the $E^\infty$-page, so no differential can happen in the spectral sequence.
\end{proof}

Combing \Cref{modifyCE}, \Cref{filtersteenrod} and \Cref{smallweightisom}, we deduce the following general upper bound and low weight computation of the $E^2$-page of the Knudsen spectral sequence.

\begin{theorem}\label{E2upperbound}
Let $M$ be a parallelizable manifold of dimension $n$ and $X$ any spectrum. Let $\g$ denote the $\Lie_{\R}$-algebra $\widetilde{H}^*(M^+)\tens\free^{\Lie_{\R}}_{\Mod_{\F}}(\Sigma^n H_*(X))$ with $\F$-basis $B$, and $\tilde\g$ the associated $\tiLie_{\R}$-algebra.
 An upper bound for the $E^2$-page of the weighted spectral sequence
 \begin{equation}\label{sseq}
     E^2_{s,t}=\AQ_{s,t}(\g)\Rightarrow \bigoplus_{k\geq 1} H_{s+t}(B_k(M;X))
 \end{equation}
 is given by
$$\pi_{*,*}(\CE(\AR(\tilde{\g})))\cong\Lambda\{\gamma_I\Q_J(y\tens w), y\tens w\in H\tens B\}\tens H^{\tiLie}_{*,*}(\tilde\g),$$
where $\gamma_I\Q_J(y\tens w)$ satisfies the conditions that
\begin{enumerate}
    \item $J=(j_1,\ldots,j_m)$ with $m\geq 1$, $0\leq j_l\leq j_{l+1}+1$ for $1\leq l<m$, and $0\leq j_m<|y|$
    \item $I$ is $\gamma$-admissible with $e(I)\leq m$. 
\end{enumerate}
Furthermore, in weight less than four equality is achieved.
\end{theorem}

\subsection{Low weight computations}
\Cref{E2upperbound} allows us to deduce the degeneration of the spectral sequence at weight two and three using sparsity arguments.  
Denote by $\mathrm{wt}_k(M)$ the weight $k$ part of a weighted (bi)graded $\F$-module $M$ and set $E^r(k)=\wt_k(E^r)$. 
\begin{corollary}\label{weight2}
Let $\g,\tilde\g$ be the same as in \Cref{E2upperbound} and $B$, $H$ bases given in \Cref{notationbasisforV}.
The weight two part of the spectral sequence (\ref{sseq})
$$E^2_{s,t}(2)=\wt_2(\AQ_{s,t}(\g))\Rightarrow  H_{s+t}(B_2(M;X))$$
collapses on the $E^2$-page, and hence
$$E^\infty(2)\cong E^2(2)\cong \wt_2(H^{\tiLie}_{*,*}(\tilde\g))\oplus \bigoplus_{x\in B,y\in H}\{\Q_j(y \tens x), 0\leq j<|y|\}.$$
\end{corollary}

\begin{proof}
    Since classes in the tensor factor $$A=\Lambda\{\gamma_I(\Q_J(y\tens w)), y\tens w\in H\tens B\}$$ of \Cref{E2upperbound} have weight at least two, classes of weight two lie in exactly one of the two tensor components $A$ and $H^{\tiLie}_{*,*}(\tilde\g)$. The weight two classes in $A$ are of the form $\Q_j (y \tens w)$ where $w$ has weight one, i.e. $w$ is an element of the $\F$-basis $B$ of $V=H_*(X)$, cf. \Cref{notationbasisforV}. The weight two classes in $H^{\tiLie}_{*,*}(\tilde\g)$ are of the form $y\tens \langle x_a, x_b \rangle$ and $(y\tens x_a)\tens (y'\tens x_b)$. 
Hence the weight two part of the spectral sequence has $E^2$-page concentrated in simplicial degrees $0$, $1$ and thus cannot admit higher differentials.
\end{proof}
In particular, this demonstrates that for a parallelizable $M$, the $\F$-module $H_*(B_2(M;X))$ depends on and only on the cohomology ring $H^*(M^+)$ when $H_*(X)$ has at least two generators.
\begin{remark}
    This is in contrast to the case where $X=\mathbb{S}^r$ has only one generator in its homology: B\"{o}digheimer-Cohen-Taylor showed that for any $n$-manifold $M$, $$\bigoplus_{k\geq 1} H_*(B_k(M;\mathbb{S}^r))\cong \bigotimes^n_{i=0} H_*(\Omega^{n-i}S^{n+r})^{\tens \mathrm{\ dim}\ H_i(M)}$$ depends only on $H^*(M)$ as an $\F$-module \cite{BCT}. 
    
    There is a clear bijection from the weight 2 part of their decomposition to the basis above: let $x_k$ denote the generator of the free $\mathbb E_{n}$-algebra $H_*(\Omega^{n}\Sigma^n S^k)$. For $y$ a basis element of $\widetilde H^i(M^+)\cong H_i(M)$, the bijection sends $\Q_j(y\tens x_{n+r})$ to the tensor with $Q_j(x_{r+i})$ in the tensor factor $H_*(\Omega^{n-i}S^{n+r})$ corresponding to $y$ and 1 in all other tensor factors. The $\tiLie$-algebra $\tiLie\g$ is trivial, so $\wt_2(H^{\tiLie}(\tiLie\g))\cong \{yy'\}$ where $y, y'$ ranges over distinct basis of $\widetilde H^i(M^+)$ and the bijection sends $yy'$ to the tensor with factors $y$, $y'$ and 1 in all other slots.
    
    On the other hand, the homology of $\mathrm{Conf}_2(M)$, the space of ordered configurations of two points in $M$, also depends only on the cup product structure of $H^*(M)$ as discussed in \cite[Section 1.1]{peterson}. 
\end{remark}

\begin{corollary}\label{closedweight3}
If in addition $M$ is a closed manifold, then
the weight three part of the spectral sequence (\ref{sseq}) collapses on the $E^2$-page, and a basis for $H_*(B_3(M;X))$ is given by
\begin{align*}
   E^\infty(3)\cong E^2(3)\cong&\bigoplus_{x,x'\in B,y,y'\in H} \F\{(\Q_{j}(y\tens x))\tens (y'\tens x'), 0\leq j<|y|\}\\
   &\oplus\mathrm{wt}_3(H^{\tiLie}_{*,*}(\tilde\g)),
\end{align*} 
\end{corollary}
\begin{proof}
Let $d$ denote the generator for $\widetilde{H}^0(M^+)\cong H^0(M)$. Then any element that is a sum of $y\tens \langle\langle x_1,x_2\rangle,x_3\rangle\in H\tens B$ is killed by a sum of $(y\tens \langle x_1,x_2\rangle)\tens (d\tens x_3)$. Since classes in $A$ have weights positive powers of two, weight three classes on the $E^2$-page either live in $\mathrm{wt}_3(H^{\tiLie}_{*,*}(\tilde\g))$ with simplicial degree one or two, or have the form $$(\Q^j(y\tens x))\tens (y'\tens x')\in \wt_2(A)\tens\mathrm{wt}_1(H^{\tiLie}_{*,*}(\tilde\g)) $$ with simplicial degree two. Hence $E^2(3)$ is concentrated in simplicial degree $1$ and $2$, so there cannot be any higher differentials.
\end{proof}

At weight four we can no longer deduce that the spectral sequence (\ref{sseq}) collapses on the $E^2$-page using sparsity arguments.
An upper bound for the bigraded $\F$-module $E^2(4)$ is given by the weight four part of $A\tens H^{\tiLie}(\tilde\g)$, which consists of:
\begin{enumerate}
\item $\Q_i(y\tens \langle x,x' \rangle)$ in simplicial degree one,
    \item $\Q_i\Q_j(y\tens x)$ and $\Q_i(y\tens x)\tens (y'\tens \langle x_1, x_2\rangle)$ in simplicial degree two,
    \item $\Q_i(y\tens x)\tens\Q_j(y'\tens x')$ and $\Q_i(y\tens x)\tens(y_1\tens x_1)\tens(y_2\tens x_2)$  in simplicial degree three,
    \item the weight four part of $H^{\tiLie}(\tilde\g)$.
\end{enumerate}
There could well be a $d^2$-differential from degree considerations.

We close this section by a few example computations: the closed torus, the punctured genus $g$ surfaces with $g\geq 1$ and the (punctured) real projective space $\mathbb{RP}^3$.

\subsection{Example computations: closed torus and punctured genus $g$ surfaces}\label{torus}

Let $\Sigma_{g,1}$ be a once-punctured surface of genus $g\geq 1$ and $\Sigma_1$ the closed torus. Let $\widetilde{B}=\{x_i\}_i$ be a totally ordered basis for $H_*(X)$ and $B=\{\sigma^2 x_i\}_i$ with the induced ordering. Then $$\widetilde{H}^*(\Sigma_{g,1}^+) = \left\{ \begin{array}{lcl}
 \F\{a_i\oplus b_i,i=1,\ldots,g\} &  *=1 \\
\F\{c\} & *=2\\
0&\mbox{otherwise}
\end{array}\right.
$$ with nonzero cup products $a_i\cup b_i=c$ for all $i$ and no nontrivial Steenrod operations.

For the closed surface $\Sigma_1$, we have $$\widetilde{H}^*(\Sigma_1^+)\cong H^*(\Sigma_1) = \left\{ \begin{array}{lcl}
\F\{d\} & *=0\\
 \F\{a, b\} &  *=1 \\
\F\{c\} & *=2\\
0&\mbox{otherwise}
\end{array}\right.$$
with nonzero cup products $a\cup b=c$ and $d\cup y =y$ for all $y\in H^*(\Sigma_1)$.

\subsubsection{Weight two}
For $M=\Sigma_{g,1}$, the weight two classes supporting nonzero $\CE$ differentials are  $\delta(a_i\tens x_1, b_i \tens x_2)=c\tens\langle x_1, x_2\rangle $ for $x_1\neq x_2\in B$, since these are the only nonzero cup products not involving the unit. Denote by $H^1$ the set of generators $\{a_i, b_i, i=1,\ldots,g\}$ for $\widetilde{H}^1(\Sigma_{g,1}^+)$. Impose a total ordering on $H^1\cup\{c,d\}$. By Corollary \ref{weight2}, a basis for $H_*(B_2(\Sigma_{g,1};X))$ is given by
\begin{align*}
   E^\infty(2)&=E^2(2)\cong\bigoplus_{x\in B} \F\{\Q_0(y\tens x), y\in H^1; \Q_0(c\tens x), \Q_1(c\tens x) \} \\
   &\oplus \bigoplus_{x_1< x_2 \in B}\F\{y\tens\langle x_1, x_2\rangle , (y\tens x_1)\tens(y\tens x_2), y\in H^1; (c\tens x_1)\tens(c\tens x_2)\}\\
   & \oplus \bigoplus_{x_1, x_2\in B}\F\{(y\tens x_1)\tens(c\tens x_2), y\in H^1\}\oplus \bigoplus_{x\in B}\F\{(y\tens x)\tens (y'\tens x), y< y'\in H^1\cup\{c\}\}\\
   & \oplus \bigoplus_{x_1<x_2\in B}\F\{(y\tens x_1)\tens(y'\tens x_2)+(a_1\tens x_1)\tens(b_1\tens x_2), y\neq y'\in H^1, (y,y')\neq(a_i, b_i)\}.
\end{align*}

For $M=\Sigma_1$, the weight two classes supporting $\CE$ differentials are  $$\delta((a\tens x_2)\tens( b \tens x_2))=c\tens\langle x_1, x_2\rangle \mathrm{\ and\ }\delta((d\tens x_1)\tens(y \tens x_2))=y\tens\langle x_1, x_2\rangle $$ for $x_1\neq x_2\in B$ and $y\in\widetilde{H}^*(\Sigma_1^+)$. By Corollary \ref{weight2}, a basis for $H_*(B_2(\Sigma_{1};X))$ is given by
\begin{align*}
   E^\infty(2)&=E^2(2)\cong\bigoplus_{x\in B} \F\{\Q_0(y\tens x), y\in H^1; \Q_0(c\tens x), \Q_1(c\tens x) \} \\
   &\oplus \bigoplus_{x_1< x_2 \in B}\F\{ (y\tens x_1)\tens(y\tens x_2), y\in H^1; (z\tens x_1)\tens(z\tens x_2)\}\\
   & \oplus \bigoplus_{x_1\neq x_2\in B}\F\{(y\tens x_1)\tens(z\tens x_2), y\in H^1\}\\
    &\oplus \bigoplus_{x\in B}\F\{(y\tens x)\tens (y'\tens x), \{y< y'\}\in \{a,b,c,d\}\}\\
    &\oplus \bigoplus_{x_1< x_2\in B}\F\{(y\tens x_1)\tens(y'\tens x_2), y,y'\in H^1, \{y,y'\}\neq\{a,b\} \}\\
   & \oplus \bigoplus_{x_1< x_2\in B}\F\{(y\tens x_1)\tens(y'\tens x_2)+(d\tens x_1)\tens(c\tens x_2), \{y,y'\}=\{a,b\} \mathrm{\ or\ } (y,y')=(c,d)\}.
\end{align*}
\begin{example}
    When $X=\mathbb{S}^k$ with $k\geq 1$, we have $B=\{x=\sigma_2\iota_k\}$, so $H_*(B_2(\Sigma_1,S^k))$ has $\F$-basis $$\{\Q_0(a\tens x),\Q_0(b\tens x),\Q_0(c\tens x),\Q_1(c\tens x); (y\tens x)\tens (y'\tens x), \{y<y'\}\subset\{a,b,c,d\}\}.$$ The weight two part of B\"{o}digheimer-Cohen-Taylor's decomposition \cite{BCT}
    \begin{equation}\label{BCTtorus}
        \bigoplus_{k\geq 1} H_*(B_k(\Sigma_1;\mathbb{S}^k))\cong \bigotimes^n_{i=0} H_*(\Omega^{2-i}S^{2+k})^{\tens \mathrm{\ dim}\ H_i(M)}\cong H_*(\Omega^2\Sigma^2 S^k)\tens H_*(\Omega\Sigma S^{1+k})^{\tens 2}\tens H_*(S^{2+k})
    \end{equation}
    is an $\F$-module on generators $Q_0(x_k)\tens 1\tens1\tens 1,Q_1(x_k)\tens 1\tens1\tens 1,1\tens \Q_0(x_{k+1})\tens1\tens 1,1\tens 1\tens\Q_0(x_{k+1})\tens 1$, as well as 6 other elements where we let two of the four tensor factors be 1 and the other two be the weight 1 generators. There is a one-to-one correspondence by sending $y\tens x$ to $x_{k+2-|y|}$ and $\Q_i(y\tens x)$ to $Q_i(x_{k+2-|y|})$ for $y=a,b,c,d$.
\end{example}
\subsubsection{Weight three}
 Classes in $A=\Lambda\big(\gamma_I(\Q^{j_1}|\cdots|\Q^{j_m}|(y\tens w)), m\geq 1\big)$ have weights positive powers of 2. Hence weight three classes in $E^2(3)$ either live in $\mathrm{wt}_3(H^{\tiLie}_{*,*}(\tilde\g))$ or has the form $$(\Q_j(y\tens x))\tens (y'\tens x')\in A\tens H^{\tiLie}_{*,*}(\tilde\g),\ \  x,x'\in B.$$

Let $H$ be the set of generators for $\widetilde{H}^*(\Sigma_{g,1}^+)\cong \widetilde{H}^*(\Sigma_g)$ and $H^1$ the set of generators for $\widetilde{H}^1(\Sigma_{g,1}^+)$. Recall that $\tilde{\g}=\widetilde{H}^*(\Sigma_g)\tens\free^{\tiLie_{\R}}_{\Mod_{\F}}(\Sigma^n H_*(X))$.  Then we have
\begin{align*}
   E^2(3)\cong&\bigoplus_{x_1,x_2\in B} \F\{(\Q_0(y\tens x_1))\tens (y'\tens x_2), y\in H^1, y'\in H\}\\
   &\oplus \bigoplus_{x_1,x_2\in B} \F\{(\Q_0(c\tens x_1))\tens (y\tens x_2), (\Q_1(c\tens x))\tens (y\tens x_2), y\in H \} \\
   &\oplus\mathrm{wt}_3(H^{\tiLie}_{*,*}(\tilde\g)).
\end{align*}   
A complete list of an $\F$-basis of $\mathrm{wt}_3(H^{\tiLie}_{*,*}(\tilde\g))$ can be written down in a straight forward way.

The $E^2$-page is concentrated in simplicial degree $0,1,2$. We need to investigate all classes in $E^2_{2,*}(3)$ to see if they support  nontrivial $d^2$-differentials to $E^2_{0,*+1}(3)$. Note that all classes in $E^2_{0,*}(3)$ are of the form $y\tens\langle \langle x_1, x_2\rangle ,x_3\rangle $ for $y\in H^1$. Since $E^2(3)$ is natural in $H_*(V)$, we can assume $x_1, x_2,x_3\in B$ have internal degree $k$ respectively. There are two cases:
\begin{enumerate}
    \item The class $(\Q_j(y_1\tens x_1))\tens (y_2\tens x_2)\in E^2_{2,*}(3)$  has internal degree at most $3k-5$ for all $y_1, y_2\in H$, while the class $y\tens\langle \langle x_1, x_2\rangle ,x_1\rangle $ has internal degree $3k-3$ for all $y\in H^1$. Hence they do not support $d^2$-differentials.
    \item The other type of classes in filtration 2 are of the form $(y_1\tens x_1)\tens(y_2\tens x_2)\tens(y_3\tens x_3)$ with internal degrees at most $3k-5$, while the class $y\tens\langle \langle x_1, x_2\rangle ,x_3\rangle $ has internal degree $3k-3$. Hence these classes do not support $d^2$-differentials either.
\end{enumerate}
Therefore the weight three part of the spectral sequence collapses at the $E^2$-page, and we obtain a basis for $H_*(B_3(\Sigma_{g,1};X))$.

For the closed surface $\Sigma_1$, $\tilde{\g}=H^*(\Sigma_{1})\tens\free^{\tiLie_{\R}}_{\Mod_{\F}}(\Sigma^n H_*(X))$ and Corollary \ref{closedweight3} says that 
\begin{align*}
   E^\infty(3)=E^2(3)\cong&\bigoplus_{x_1,x_2\in B} \F\{(\Q_0(y\tens x_1))\tens (y'\tens x_2), y\in H^1, y'\in H\cup\{d\}\}\\
   &\oplus \bigoplus_{x_1,x_2\in B} \F\{(\Q_0(c\tens x_1))\tens (y\tens x_2), (\Q_1(c\tens x))\tens (y\tens x_2), y\in H\cup\{d\} \} \\
   &\oplus\mathrm{wt}_3(H^{\tiLie}_{*,*}(\tilde\g)).
\end{align*}   

We do not list the $\F$-basis of $\mathrm{wt}_3((H^{\tiLie}_{*,*}(\tilde\g))$ for simplicity.
\begin{example}
    As in the weight two case, our basis for $H_*(B_3(\Sigma_1,\mathbb{S}^k)), k\geq1$ is in bijection with the weight 3 part of \Cref{BCTtorus} by sending $y\tens x$ to $x_{k+2-|y|}$ and $\Q_i(y\tens x)$ to $Q_i(x_{k+2-|y|})$ for $y=a,b,c,d$.
\end{example}
\subsection{Example computations: (punctured) real projective space}\label{RP3}
The simplest examples of parallelizable manifolds admitting nontrivial Steenrod actions other than $Sq^0$ are the  real projective space $\mathbb{RP}^3$ and the once-punctured real projective space $\dot{\mathbb{RP}^3}$.

Let $y$ be a generator for $H^1(\mathbb{RP}^3)$. Then $$\widetilde{H}^*((\mathbb{RP}^3)^+)\cong H^*(\mathbb{RP}^3)=\F[y]/(y^4),\widetilde{H}^*((\dot{\mathbb{RP}^3})^+)=\widetilde H^*(\mathbb{RP}^3)=\F\{y,y^2,y^3\}$$ with the obvious cup products and one nontrivial Steenrod operation $Sq^1(y)=y^2$. 

\subsubsection{Weight two}
We deduce $H_*(B_2(\dot{\mathbb{RP}^3};X))$ and $H_*(B_2(\mathbb{RP}^3;X))$ from Corollary \ref{weight2}.
For $M=\dot{\mathbb{RP}^3}$, there is only one nontrivial cup product $y\cup y^2=y^3$, so
\begin{align*}
 E^\infty(2)=E^2(2)&=\bigoplus_{x\in B,a=1,2,3}\F\{\Q_j(y^a\tens x), 0\leq j<a\}\\
  &\oplus\bigoplus_{x\in B}\F\{(y^a\tens x)\tens(y^b\tens x), 1\leq a<b\leq 3\}\\
 &\oplus \bigoplus_{x_1< x_2 \in B}\F\{y\tens\langle x_1, x_2\rangle ; (y^a\tens x_1)\tens(y^a\tens x_2),a=2,3\}\\
   & \oplus \bigoplus_{x_1\neq x_2\in B}\F\{(y^a\tens x_1)\tens(y^3\tens x_2), a=1,2\}
   \\
   & \oplus \bigoplus_{x_1< x_2\in B}\F\{(y^1\tens x_1)\tens(y^2\tens x_2)+(y^2\tens x_1)\tens(y^1\tens x_2) \}.
\end{align*}
For $M=\mathbb{RP}^3$, the nonzero cup products are $y\cup y=y^2, y\cup y^2=y^3$ and $1\cup y^a=y^a$ for $0\leq a\leq 3$, so
\begin{align*}
 E^\infty(2)=E^2(2)&=\bigoplus_{x\in B,a=1,2,3}\F\{\Q_j(y^a\tens x), 0\leq j<a\}\\
 &\oplus\bigoplus_{x\in B}\F\{(y^a\tens x)\tens(y^b\tens x), 0\leq a<b\leq 3\}\\
 &\oplus \bigoplus_{x_1< x_2 \in B}\F\{ (y^a\tens x_1)\tens(y^a\tens x_2),a=2,3\}\\
   & \oplus \bigoplus_{x_1\neq x_2\in B}\F\{(y^a\tens x_1)\tens(y^3\tens x_2), a=1,2\}\\
   & \oplus \bigoplus_{x_1< x_2\in B}\F\{(y^a\tens x_1)\tens(y^b\tens x_2)+(y^3\tens x_1)\tens(1\tens x_2), (a,b)\neq (3,0) \}\\
   & \oplus \bigoplus_{x_1< x_2\in B}\F\{(y^a\tens x_1)\tens(1\tens x_2)+(1\tens x_1)\tens(y^a\tens x_2), a=1,2,3 \}.
\end{align*}
\begin{example}
    When $X=\mathbb{S}^k$ with $k\geq 1$, we have $B=\{x=\sigma_2\iota_k\}$, so $H_*(B_2(\mathbb{RP}^3,S^k))$ has $\F$-basis $$\{\Q_j(y^a\tens x), 0\leq j<a, a=1,2,3; (y^a\tens x)\tens (y^b\tens x), 0\leq a<b\leq 3\}.$$ A bijection with weight 3 part of B\"{o}digheimer-Cohen-Taylor's decomposition \cite{BCT}
    \begin{align*}
        \bigoplus_{k\geq 1} H_*(B_k(\mathbb{RP}^3;\mathbb{S}^k))\cong& \bigotimes^n_{i=0} H_*(\Omega^{3-i}S^{3+k})^{\tens \mathrm{\ dim}\ H_i(M)}\\ \cong& H_*(\Omega^3\Sigma^3 S^{k})\tens H_*(\Omega^2\Sigma^2 S^{k+1})\tens H_*(\Omega\Sigma S^{k+2})\tens H_*(S^{k+3})\\
        \cong& \free^{E_3}(\F\{x_k\})\tens\free^{E_2}(\F\{\{x_{k+1}\})\tens\free^{E_1}(\F\{x_{k+2}\})\tens\F\{x_{k}\}
    \end{align*}
    is given by sending $y^a\tens x$ to $x_{k+3-a}$ and $\Q_i(y^a\tens x)$ to $Q_i(x_{k+3-a})$ for $0\leq a\leq 3.$ 
\end{example}
\subsubsection{Weight three}
For the closed manifold $\mathbb{RP}^3$ and $\tilde{\g}=H^*(\mathbb{RP}^3)\tens\free^{\tiLie_{\R}}_{\Mod_{\F}}(\Sigma^n H_*(X))$, it follows from Corollary \ref{closedweight3} that
\begin{align*}
    &E^\infty(3)=E^2(3)
    \cong\mathrm{wt}_3(H^{\tiLie}_{*,*}(\tilde\g))\oplus \bigoplus_{x_1,x_2\in B, 1\leq a\leq 3, 0\leq b\leq 3} \F\{(\Q_j(y^a\tens x_1))\tens (y^b\tens x_2), 0\leq j<a\}.
\end{align*}

\

For the punctured real projective space $\dot{\mathbb{RP}^3}$ and  $\tilde{\g}=\widetilde H^*(\mathbb{RP}^3)\tens\free^{\tiLie_{\R}}_{\Mod_{\F}}(\Sigma^n H_*(X))$, weight three classes in $E^2(3)$ either live in $\mathrm{wt}_3(H^{\tiLie}_{*,*}(\tilde\g))$ or has the form $$(\Q_j(y^a\tens x))\tens (y^b\tens x')\in A\tens H^{\tiLie}_{*,*}(\tilde\g)$$ with $x,x'\in B$ and $1\leq a,b \leq 3$. Therefore
\begin{align*}
   E^2(3)=\mathrm{wt}_3(H^{\tiLie}_{*,*}(\tilde\g))\oplus \bigoplus_{x_1,x_2\in B, 1\leq a,b\leq 3} \F\{(\Q_j|(y^a\tens x_1))\tens (y^b\tens x_2), 0\leq j<a\}.
\end{align*}  
A complete list of an $\F$-basis for $\mathrm{wt}_3(H^{\tiLie}_{*,*}(\tilde\g))$ is given by
\begin{enumerate}
    \item $y\tens \langle \langle x_1,x_2\rangle ,x_3\rangle $ for  $x_1,x_2,x_3\in B, x_1<x_2, x_1<x_3$ in simplicial degree 0;
    \item $(y^3\tens\langle x_1,x_2\rangle )\tens(y^b\tens x_3)+(y^3\tens\langle x_1,x_3\rangle )\tens(y^b\tens x_2)+(y^3\tens\langle x_2,x_3\rangle )\tens(y^b\tens x_1)$ for $b=1,2$ 
    
    and 
    $(y\tens\langle x_1,x_2\rangle )\tens(y^2\tens x_3)+(y\tens\langle x_1,x_3\rangle )\tens(y^2\tens x_2)+(y\tens\langle x_2,x_3\rangle )\tens(y^2\tens x_1)$ for distinct $x_i\in B$ in simplicial degree 1;
    \item $(y^a\tens x_1)\tens (y^b\tens x_2)\tens (y^c\tens x_3)$ for $\{1,2\},\{1,1\}\nsubseteq\{a,b,c\}$ and $x_i\in B$;
    
    $\sum_{\{i,j,k\}=\{1,2,3\},i<j} (y\tens x_i)\tens (y\tens x_j)\tens (y^2\tens x_k)$,
    
    $\sum_{\{i,j,k\}=\{1,2,3\},j<k} (y\tens x_i)\tens (y^2\tens x_j)\tens (y^2\tens x_k)$ for distinct $x_1, x_2, x_3\in B$
    in simplicial degree 2.
\end{enumerate}

Again  the $E^2$-page is concentrated in simplicial degrees $0,1,2$, and we use sparsity to rule out higher differentials.
Suppose that $x_1, x_2, x_3$ have internal degree $k$. We examine the two cases that could potentially support a $d^2$-differential.

\begin{enumerate}
    \item The class $(\Q_j(y^a\tens x_1))\tens (y^b\tens x_2)\in E^2_{2,*}(3)$  has internal degree at most $3k-5$ for all $1\leq a,b\leq 3$, while the class $y\tens\langle \langle x_1, x_2\rangle ,x_1\rangle $ has internal degree $3k-3$. Hence they do not support $d^2$-differentials.
    \item The other type of classes in simplicial degree 2 are of the form $(y^a\tens x_1)\tens(y^b\tens x_2)\tens(y^c\tens x_3)$ with internal degrees at most $3k-5$, while the class $y\tens\langle \langle x_1, x_2\rangle ,x_3\rangle $ has internal degree $3k-3$. Hence these classes do not support $d^2$-differentials either.
\end{enumerate}
Therefore the weight three part of the spectral sequence collapses on the $E^2$-page, and we obtain a basis for $H_*(B_3(\dot{\mathbb{RP}^3};X))$.

\section{Odd primary homology}\label{section6}
In this last section, we apply the same methods to study the mod $p$ homology of $B_k(M;X)$ for $p>2$ via the Knudsen spectral sequence with $\Fp$ coefficient.

\subsection{Odd primary Knudsen spectral sequence}
We start by recalling partial progress in understanding the unary operations on the mod $p$ homology of spectral Lie algebras by Kjaer \cite{kjaer}. He constructed  weight $p$ Dyer-Lashof-type operations in analogy to Behrens' construction of $\Q^j$, which was further clarified by the work of Konovalov.
\begin{proposition}\cite[Definition 3.2]{kjaer}\cite[Definition 2.5.17]{konovalov}\label{defofoddop}
    Let $L$ be a spectral Lie algebra. Then $H_*(L;\Fp)$ admits unary operations $$\overline{ \beta^{\epsilon} Q^j}: H_*(L;\Fp)\rightarrow H_{*+2(p-1)i-\epsilon-1}(L;\Fp),\ \epsilon\in\{0,1\}, j\in\mathbb{Z}.$$ On a class $x\in H_*(L;\Fp)$ such that if $|x|$ is even then $2j\neq x$, the class $\overline{ \beta^{\epsilon} Q^j}(x)$ is given by $\xi_*( \sigma^{-1}\beta^{\epsilon} Q^j(x))$, where $\beta^{\epsilon}Q^j$ is a mod $p$ Dyer-Lashof operation, $\sigma^{-1}$ the desuspension isomorphism, and $\xi:\partial_p(\mathrm{Id})\tens_{h\Sigma_p}L^{\tens p}\rightarrow L$ the $p$th structure map of the spectral Lie algebra $L$. When $|x|=2l$, define   $\overline{ \beta Q^l}(x)$ via the isomorphism $H_*(\partial_{p}(\id)\tens_{h\Sigma_{p}}(S^{2l})^{\tens p})\cong H_*(\Sigma^{-1}(\partial_{p}(\id)\tens_{h\Sigma_{p}}(S^{2l-1})^{\tens p}))$.
\end{proposition}
It follows from the instability condition of Dyer-Lashof operations that the allowability condition for the operations $\overline{ \beta^{\epsilon}}$ are given by  $\overline{ \beta^{\epsilon} Q^j}(x)=0$ if $j<\frac{|x|}{2}$. Analogous to the case $p=2$,  brackets of unary operations always vanish.
\begin{proposition}\cite[Proposition 3.7]{kjaer}\label{oddvanish}
    For $L$ a spectral Lie algebra, $[\overline{ \beta^{\epsilon} Q^j}(x), y]=0$ for any $\epsilon, j$ and $x,y\in H_*(L;\Fp)$ . 

\end{proposition}
The relations among the unary operations were obtained by Konovalov.

\begin{proposition} \cite[Theorem 8.2.14]{konovalov}
Let $\R$ be the free algebra over $\mathbb{F}_p$ on generators $\overline{ \beta^{\epsilon} Q^j},\epsilon\in\{0,1\}$, subject to the relations
\begin{align*}
 \overline{ \beta^{\epsilon} Q^j}\cdot\overline{ \beta Q^i}=&(-1)^{\epsilon+1}\sum^{i+j-1}_{m=pi}\binom{p(m - i) - (p - 1)j + \epsilon -1}{m-pi} \overline{ \beta Q^m} \cdot \overline{ \beta^{\epsilon} Q^{j+i-m}}\\
 +&(1-\epsilon)\sum^{i+j-1}_{m=pi+1}\binom{p(m-i)-(p-1)j}{m-pi} \overline{ Q^m} \cdot \overline{\beta Q^{j+i-m}}
 \end{align*} 
 for $j<pi$, and
 $$ \overline{ \beta^{\epsilon} Q^j}\cdot\overline{ Q^i}=\sum^{i+j-1}_{m=pi+1}\binom{p(m - i) - (p - 1)j -1}{m-pi-1} \overline{ \beta^\epsilon Q^m} \cdot \overline{ Q^{j+i-m}}$$
 for $j\leq pi$.
 Then the mod $p$ homology of a spectral Lie algebra is an allowable module over $\R$.
\end{proposition}

Denote by $\pazocal{A}_{\R}$ the free allowable $\R$-module monad. Let $\Lie_{\R}:\Mod_{\mathbb{F}_p}\rightarrow\Mod_{\Fp}$ be the composite monad $\pazocal{A}_{\R}\circ \Lie_{\Fp}$ subject to the commuting relations \Cref{oddvanish} when $p>3$, and the monad given by $\Lie_{\R}(M)=\pazocal{A}_{\R}\circ \Lie_{\mathbb{F}_3}(M)/\langle\overline{ \beta^{\epsilon} Q^{|x|/2}}(x)=[[x,x],x]\rangle$, where we take the quotient by the $\R$-module ideal ranging over $x\in M$ in even degree. For $M\in\Mod_{\Fp}$, let $A$ be an $\Fp$-basis for the free shifted Lie algebra $\free^{\Lie_{\Fp}}_{\Mod_{\mathbb{F}_p}}(M)$. The graded $\Fp$-module $\Lie_{\R}(M)$ has basis
$$\{\overline{ \beta_1^{\epsilon_1} Q^{j_1}}\cdots \overline{ \beta_k^{\epsilon_k} Q^{j_k}}| x,\ \  x\in A,  j_k\geq \frac{|x|}{2}, j_i\geq pj_{i+1}-\epsilon_{i+1}\forall i\}.$$
\begin{theorem}\cite[Theorem 5.2]{kjaer}\cite[Theorem 8.2.17]{konovalov}\label{kjaer}
For $X$ a spectrum. there is an isomorphism of $\Lie_{\R}$-algebras
$$\Lie_{\R}(H_*(X;\mathbb{F}_p))\rightarrow H_*(\free^{s\lie}(X);\Fp).$$
\end{theorem}
\begin{remark}\label{kjaermistake}
    For $p=3$, Kjaer claimed in \cite[Corollary 4.7]{kjaer} that the triple bracket  on an even degree homology class $\iota_{2l}$ of a spectral Lie algebra  is zero by showing that $$[[\iota_{2l},\iota_{2l}],\iota_{2l}]\in H_*(\partial_3(\id)\tens_{h\Sigma_3}(S^{2l})^{\tens 3})$$ vanishes. The claim is incorrect in light of \Cref{odddiff} below, and was independently observed by Nikolai Konovalov. Specifically, Kjaer argued that in the long exact sequence
    $$\cdots\rightarrow H_{6l-2}(\Sigma^{-2}(S^{2l})^{\tens 3}_{h\Sigma_3}) \rightarrow H_{6l-2}(\partial_3(\id)\tens_{h\Sigma_3}(S^{2l})^{\tens 3})\rightarrow H_{6l-2}((\Sigma^{-1} (S^{2l})^{\tens 3}_{h\Sigma 3})\rightarrow\cdots,$$
    the middle group is generated as an $\mathbb{F}_3$-module by the bottom operation  $\overline{ \beta Q^{l}}\iota_{2l}$, which is mapped isomorphically onto $\sigma^{-1}\beta Q^{l}\iota_{2l}$ by definition of the bottom operation in Definition 3.2.  However, $\sigma^{-1}\beta Q^{l}\iota_{2l}\in H_{6l-2}(\Sigma^{-1} (S^{2l})^{\tens 3}_{h\Sigma 3})=0$. In fact, one can see that the confusion was caused by incorrect placement of parentheses. Since the left term is one-dimensional on $[[\iota_{2l},\iota_{2l}],\iota_{2l}]$, we see that $[[\iota_{2l},\iota_{2l}],\iota_{2l}]=\mu_{l}\overline{ \beta Q^{l}}\iota_{2l}$, where $\mu_{l}=\pm 1$. This also motivates the modification of the definition of the bottom operation on an even class in \Cref{defofoddop}.
\end{remark}
Now we turn to the odd primary Knudsen spectral sequence 
\begin{equation}\label{oddprimarysseq}
    E^2_{s,t}(k)=\pi_{s,t}\big( \B\big(\mathrm{id}, \Lie_{\R}, \g\big)\tens \mathbb{F}_p \big)(k)\Rightarrow H_{s+t}(B_k(M;X);\mathbb{F}_p),
\end{equation}where $$\g=H_*(\free^{s\lie}(\Sigma^{n}X) ^{M^+};\Fp)\cong \widetilde{H}^*(M^+;\Fp)\tens \Lie_{\p}(\Sigma^n H_*(X;\Fp)).$$ Furthermore, $\g$ has a $\Lie_{\Fp}$-structure  given by \Cref{bracketbhk}, i.e., $$[y_1\tens x_1,  y_2\tens x_2]:=(-1)^{|x_1||y_2|}(y_1\cup y_2)\tens[x_1,x_2].$$ 
 
We proceed to compute the $E^2$-page of the spectral sequence (\ref{oddprimarysseq}) in low weights in terms of $\Lie_{\Fp}$-algebra homology.

\begin{definition}\cite{CE}\cite{may}\label{FpCE}
    For a shifted Lie algebra $L$ over $\Fp$, let $L_{\text{even}}$ and $L_{\text{odd}}$ denote the elements in $L$ with even and odd degrees, respectively. The \textit{Chevalley-Eilenberg} complex of $L$ is the chain complex  
    $$\CE(L) = (\Gamma^\bullet(L_{\text{even}})\otimes \Lambda^\bullet (L_{\text{odd}}), \partial),$$
    where $\Gamma^\bullet$ and $\Lambda^\bullet$ are respectively the graded, shifted divided power and exterior algebra functor over $\Fp$, and the differential $\partial$ on a general element
    $$\gamma_{k_1}(x_1)\gamma_{k_2}(x_2) \cdots \gamma_{k_m}(x_m)\langle y_1, y_2, \dots, y_n\rangle\in\Gamma^{\bullet}(L_{\text{even}})\otimes \Lambda^\bullet (L_{\text{odd}})$$
    is given by
    \begin{align*}
        &\sum_{1\le i <j \le m} \gamma_{k_1}(x_1)\cdots \gamma_{k_i-1}(x_i)\cdots \gamma_{k_j-1}(x_j) \cdots \gamma_{k_m}(x_m)\langle [x_i,x_j], y_1, \dots y_n\rangle \\
        + &\sum_{1\le i < j \le n}(-1)^{i+j-1}\gamma_{k_1}(x_1)\cdots \gamma_{k_m}(x_m) \langle [y_i,y_j], y_1, \dots, \widehat{y_i},\dots \widehat{y_j},\dots y_n\rangle \\
        + &\frac 12 \sum_{i=1}^m \gamma_{k_1}(x_1)\cdots \gamma_{k_i-2}(x_i)\cdots \gamma_{k_m}(x_m) \langle [x_i, x_i], y_1, \dots, y_n\rangle\\
        + &\sum_{i=1}^m\sum_{j=1}^n (-1)^{j-1}\gamma_1([x_i, y_j])\gamma_{k_1}(x_1)\cdots \gamma_{k_i-1}(x_i)\cdots \gamma_{k_m}(x_m) \langle y_1, \dots, \widehat{y_j}, \dots, y_n\rangle.
    \end{align*}
\end{definition}

\begin{proposition}\label{oddE2}
Let $M^n$ be a parallelizable manifold and $X$ any spectrum.
\begin{enumerate}
    \item For $k<p$, the weight $k$ part of the spectral sequence $$E^2_{s,t}(k)=\pi_s\pi_t\big( \B\big(\mathrm{id}, s\lie, \free^{s\lie}(\Sigma^{n}X) ^{M^+}\big)\tens \mathbb{F}_p \big)(k)\Rightarrow H_{s+t}(B_k(M;X);\mathbb{F}_p)$$ has $E^2$-page given by $\wt_k(H_{*,*}(\CE(\g))$, where $\g= \widetilde{H}^*(M^+;\Fp)\tens \Lie_{\Fp}(\Sigma^n H_*(X;\Fp)).$  
    \item For $p\geq 5$, the weight $p$ part of the spectral sequence has $E^2$-page given by $$E^2_{*,*}(k)\cong \wt_p(H_{*,*}(\CE(\g)))\oplus \bigoplus_{y\in H, x\in B} \Fp\Big\{\overline{ \beta^{\epsilon} Q^{j}}|y\tens x, \ \frac{|x|-|y|}{2}\leq j<\frac{|x|}{2}\Big\},$$ where $H$ is an $\Fp$-basis of $\widetilde{H}^*(M^+;\Fp)$ and $B$ an $\Fp$-basis of $H_*(X;\Fp)$.
\end{enumerate}
\end{proposition}
\begin{proof}
For $k<p$, all elements in the weight $k$ part of the $E^2$-page of the spectral sequence  do not contain unary operations $\overline{ \beta^{\epsilon} Q^{j}}$. When $k=p$, nondegenerate elements of weight $p$ on the $E^2$-page are either of the form $\overline{ \beta^{\epsilon} Q^{j}}|y\tens x\in\Lie_{\R}(\g)$, $\overline{ \beta^{\epsilon} Q^{j}}(y\tens x)\in\g$, or a bracket of weight $p$. When $p\geq 5$, the unary operation $ \beta^{\epsilon} Q^{j}$ cannot be an iteration of brackets on a single element, since $[[x,x],x]=0$ for any $x$ by the Jacobi identity. Hence there is no $d_1$-differential from a weight $p$ bracket to $\overline{ \beta^{\epsilon} Q^{j}}|y\tens x$ or $y\tens \overline{ \beta^{\epsilon} Q^{j}}(x)$. The same argument in \Cref{filtersteenrod} implies that the twisting of the action of $\overline{ \beta^{\epsilon} Q^{j}}$ by Steenrod operations can be ignored when computing a basis for the $E^2$-page.
\end{proof}

The condition $p\geq 5$ in part (2) is necessary in light of the following computation for Euclidean spaces.
\begin{proposition}\label{odddiff}
    For $p\geq 5$, the only higher differential  in the weight $p$ part of the spectral sequence (\ref{oddprimarysseq}) for $M=\mathbb{R}^n, 2\leq n\leq \infty$, which converges to $H_*(B_p(\mathbb{R}^n;\mathbb{S}^{2l});\Fp)$, is a $d_{p-2}$-differential $\gamma_p(x)\mapsto \overline{ \beta^{\epsilon} Q^{l}}|y_n\tens\sigma^n(x)$.

    When $p=3$, the above spectral sequence has a $d^1$-differential $\gamma_3(x)\mapsto\overline{ \beta^{\epsilon} Q^{l}}|x$.
\end{proposition}
Heuristically, this is because the bottom non-vanishing mod $p$ Dyer-Lashof operation on a class $x$ of degree $2l$ in the mod $p$ homology of an $\mathbb{E}_n$-algebra  is given by $\Q^l(x)=x^{\tens p}$, so $\gamma_p(x)$ is redundant.
\begin{proof}
   Consider the spectral sequence (\ref{oddprimarysseq}) when $M=\mathbb{R}^n$ and $X=\mathbb{S}^{2l}$ with $n>2$, so $$\g= \Fp\{y_{n}\}\tens \Lie_{\p}(\Fp\{\sigma^n(x_{2l})\})$$ with $y$ in internal degree $-n$ and $x_{2l}$ in degree $2l$. Set $x=y_{n}\tens \sigma^n(x_{2l})$.  Then the weight $p$ part of the $E^2$-page has basis $$\Big\{\overline{ \beta^{\epsilon} Q^{j}}|x,\  l\leq j<\frac{2l+n}{2};\ \gamma_p(x)\Big\}.$$ Comparing with the weight $p$ part of the $E^\infty$-page, which is the weight $p$ part of the mod $p$ homology of the free $\mathbb{E}_n$-algebra on the $\mathbb{S}^{2l}$, we see that there are two  classes that do not survive to the $E^\infty$-page, i.e., $\gamma_p(x)$ in bidegree $(p-1,2pl-(p-1))$ and $\overline{ \beta Q^{l}}|x$ in bidegree $(1,2pl-2)$ (cf. \cite[III]{CLM}). Hence there has to be a $d_{p-2}$-differential from $\gamma_p(x)$ to $\overline{ \beta^{\epsilon} Q^{l}}|x$. 

   When $p=3$, $\gamma_3(x)$ is represented by the element $[[x,x],x]\in \Lie_{\Fp}\circ \Lie_{\Fp}(\g)\subset \Lie_{\p}\circ \Lie_{\p}(\g)$. It is mapped by the differential to $[[x,x],x]\in \Lie_{\Fp}(\g)$, which by \Cref{kjaermistake} is indeed $\overline{ \beta^{\epsilon} Q^{l}}|x$.
\end{proof}

As an immediate corollary to \Cref{oddE2}, we see that the weight two part of the spectral sequence (\ref{oddprimarysseq}) collapses on the $E^2$-page, since the $E^2$-page is concentrated in simplicial degree 0 and 1. When $p>3$, weight three elements on the $E^2$-page are in simplicial degree 1 or 2 since $[[x,x],x]=0$ by the Jacobi identity. Hence the weight three part of the spectral sequence (\ref{oddprimarysseq})  also collapses on the $E^2$-page.
\begin{corollary} \label{oddsmallweight}
Let $M^n$ be a parallelizable manifold and $X$ any spectrum. Let  $\g$ be the $\Lie_{\Fp}$-algebra $\widetilde{H}^*(M^+;\Fp)\tens \Lie_{\Fp}(\Sigma^n H_*(X;\Fp))$
\begin{enumerate}
    \item For all $i$, there is an isomorphism of $\Fp$-modules $$H_i(B_2(M;X);\Fp)\cong\bigoplus_{s+t=i}\wt_2(H_{s,t}(\CE(\g)).$$ 
    \item If $p\geq 5$, then for all $i$ $$H_i(B_3(M;X);\Fp)\cong\bigoplus_{s+t=i}\wt_3(H_{s,t}(\CE(\g)).$$
\end{enumerate}
\end{corollary}
\begin{remark}\label{dependoncupproduct}
For $M$ a connected $n$-manifold, B\"{o}digheimer-Cohen-Taylor showed that $$\bigoplus_{k\geq 1}H_*(B_k(M;S^r);\Fp)\cong \bigotimes^n_{i=0} H_*(\Omega^{n-i}S^{n+r};\Fp)^{\tens \mathrm{\ dim}\ H_i(M;\Fp)}$$ for $r+n$ odd and $r\geq 0$ \cite{BCT}. Their proof does not work in the case where $r+n$ is even due to the existence of nontrivial self-brackets in $H_*(\Omega^m \Sigma^m S^{l});\Fp)$ when $l$ is even. Roughly speaking, their inductive proof relies on  the canonical map $$H_*(\Omega^m \Sigma^m S^{l};\Fp)\rightarrow H_*(\Omega^\infty \Sigma^\infty S^{l};\Fp)$$ being an injection, which is only true when $l$ is odd. \Cref{oddsmallweight} shows that when $l$ is even, the mod $p$ homology of $B_k(M;\mathbb S^r), k=2,3$ depends on the cup product structure on $H^*(M^+;\Fp)$: if $a\cup b=c$ in $\widetilde H^*(M^+;\Fp)$, then the $d_1$-differential sends $(a\tens x)\tens(b\tens x)$ to $$c\tens [x,x]\in \g=\widetilde H^*(M^+;\Fp)\tens \Lie_{\p}(\Fp\{x\}),$$ which is not zero since $x$ has internal degree $l$.
\end{remark}

At higher weights, there generally will be higher differentials in the odd primary Knudsen spectral sequence (\ref{oddprimarysseq}).
In recent work with Matthew Chen \cite{cz}, we make use of \Cref{oddE2} and Drummond-Cole-Knudsen's computation of the rational homology of the unordered configurations space $B_k(M)$ where $M=\Sigma_1$ or $\Sigma_{g,1}$ \cite{dck} to identify the differentials in the Knudsen spectral sequence for $B_k(\Sigma_g;\mathbb{S})$. As a result, we show that the integral homology of $B_k(\Sigma_1)$ is $p$-torsion-free for $k\leq p$. The same argument works for the punctured surface $\Sigma_{g,1}$ with $g\geq 0$, thereby providing a more elementary proof for \cite[Theorem 1.10]{bhk}

\end{document}